\documentclass[10pt,reqno]{amsart}
\usepackage{amssymb,latexsym}
\usepackage{amsmath}
\usepackage{amsthm}
\usepackage{graphicx}
\usepackage{hyperref}
\usepackage{titletoc}
\usepackage{pdfsync}
\usepackage{color,xcolor}
\usepackage{amsthm,amsmath,amssymb}
\usepackage{amsthm,amsmath,amssymb}

\usepackage{mathrsfs}

\numberwithin{equation}{section}
\newtheorem{theorem}{Theorem}[section]
\newtheorem{proposition}[theorem]{Proposition}
\newtheorem{lemma}[theorem]{Lemma}

\newtheorem{definition}[theorem]{Definition}

\usepackage{multirow}
\usepackage[font=bf,aboveskip=15pt]{caption}
\usepackage[toc,page]{appendix}

\theoremstyle{definition}

\newcommand{\inn}{{\quad\hbox{in } }}

\newcommand{\LLL}{{\mathcal L}  }

\newcommand{\nn}{ {\nabla}  }
\newcommand{\A}{\alpha }
\newcommand{\vp}{\varphi}
\newcommand{\B}{\beta }

\newcommand{\rH}{{\mathcal H}}
\newcommand{\rR}{{\mathcal R}}

\newcommand{\NN}{ {\mathcal N}}

\newcommand{\R} {\mathbb R}

\newcommand{\cuad}{{\sqcap\kern-.68em\sqcup}}

\newcommand{\ve}{\e}
\newcommand{\e}{\epsilon}

\newcommand{\be}{\begin{equation}}
\newcommand{\ee}{\end{equation}}

\newtheorem{remark}[theorem]{Remark}

\def\A{{\mathfrak A}}
\def\B{{\mathfrak B}}

\def\G{{\mathfrak G}}

\def\R{{\mathfrak R}}

\def\begeq{\begin{equation}}
\def\endeq{\end{equation}}
\def\p{\partial}

\def\R{\Bbb R}

\def\G{\Gamma}

\allowdisplaybreaks

\begin{document}

\title
[The vortex filaments of Ginzburg-Landau system]
{The helical vortex filaments of Ginzburg-Landau system in ${\mathbb R}^3$}

\author{Lipeng Duan}
\address{Lipeng Duan,
\newline\indent School of Mathematics and Information Science, Guangzhou University,
\newline\indent Guangzhou 510006, P. R. China.
}
\email{lpduan777@sina.com}

\author{Qi Gao}
\address{Qi Gao,
\newline\indent Department of Mathematics, School of Science, Wuhan University of Technology,
\newline\indent Wuhan 430070, P. R. China
}
\email{gaoq@whut.edu.cn}

\author{Jun Yang$^\S$}
\address{Jun Yang,
\newline\indent School of Mathematics and Information Science, Guangzhou University,
\newline\indent Guangzhou 510006, P. R. China.
}
\email{jyang2019@gzhu.edu.cn}

\begin{abstract}
\vspace{4mm}
We consider the following coupled  Ginzburg-Landau system in ${\mathbb R}^3$
\begin{align*}
\begin{cases}
-\epsilon^2 \Delta  w^+ +\Big[A_+\big(|w^+|^2-{t^+}^2\big)+B\big(|w^-|^2-{t^-}^2\big)\Big]w^+=0,
\\[3mm]
-\epsilon^2 \Delta  w^- +\Big[A_-\big(|w^-|^2-{t^-}^2\big)+B\big(|w^+|^2-{t^+}^2\big)\Big]w^-=0,
\end{cases}
\end{align*}
where $w=(w^+, w^-)\in \mathbb{C}^2$ and the constant coefficients  satisfy
$$
A_+, A_->0,\quad B^2<A_+A_-, \quad t^\pm >0, \quad {t^+}^2+{ t^-}^2=1.
$$
If $B<0$, then for every $\epsilon$ small enough, we construct a family of entire solutions $w_\epsilon (\tilde{z}, t)\in \mathbb{C}^2$ in the cylindrical coordinates $(\tilde{z}, t)\in \mathbb{R}^2 \times \mathbb{R}$ for this system via the approach introduced by J. D\'avila, M. del Pino, M. Medina and R. Rodiac in {\tt arXiv:1901.02807}. These solutions are $2\pi$-periodic in $t$ and have multiple interacting vortex helices.
The main results are the extensions of the phenomena of interacting helical vortex filaments for the classical (single) Ginzburg-Landau equation in $\mathbb{R}^3$ which has been studied  in {\tt arXiv:1901.02807}. Our results negatively answer the Gibbons conjecture  \cite{Gibbons conjecture} for the Allen-Cahn equation in Ginzburg-Landau system version, which is an extension of the question originally proposed by H. Brezis.

\vspace{2mm}
{\textbf{Keywords:}  Ginzburg-Landau model,    Helical Vortex Filaments,    Lyapunov-Schmidt reduction}
\vspace{2mm}

{\textbf{AMS Subject Classification:}  35A01,  35B25, 35Q56.}
\end{abstract}

\date{\today}

\thanks{ Corresponding author: Jun Yang, jyang2019@gzhu.edu.cn}

\maketitle

\section{introduction}
In this paper, we study the entire solutions of the following Ginzburg-Landau  system  in $\R^N$, for complex vector-valued functions $w=(w^+, w^-): \R^N\rightarrow \mathbb{C}^2$:
 \begin{align}\label{original}
 \begin{cases}
 -\epsilon^2 \Delta_{{\mathbb R}^N}  w^+ +\Big[A_+\big(|w^+|^2-{t^+}^2\big)+B\big(|w^-|^2-{t^-}^2\big)\Big]w^+=0,
 \\[3mm]
-\epsilon^2 \Delta_{{\mathbb R}^N}  w^- +\Big[A_-\big(|w^-|^2-{t^-}^2\big)+B\big(|w^+|^2-{t^+}^2\big)\Big]w^-=0,
\end{cases}
\end{align}
here $A_\pm >0$, $t^\pm$, $B$ and $\epsilon>0$ are parameters, the notation $\Delta_{{\mathbb R}^N}  $  is the Laplacian operator in $\R^N$. Throughout the paper we make
the following assumptions concerning the constants appearing in \eqref{original}:
\begin{align} \label{H1}
A_+, A_->0,\ B^2<A_+A_-,  \tag{H1}
\end{align}
\begin{align}
\ t^\pm>0,\ \text{and}~{t^+}^2+{ t^-}^2=1.
\end{align}
The corresponding energy functional to system (\ref{original}) is
 \begin{align*}
{\mathbf E}( w)=\frac{1}{2}\int_{\mathbb{R}^N}|\nabla w|^2\ +\
&\frac{1}{4\epsilon^2}\int_{\mathbb{R}^N}\Big[\,A_+(| w^+|^2-{t^+}^2)^2+A_-(| w^-|^2-{t^-}^2)^2
\\[1mm]
&\qquad\quad+2B(| w^+|^2-{t^+}^2)(| w^-|^2-{t^-}^2)\,\Big].
\end{align*}
The hypothesis \eqref{H1}  ensures that  the potential term in the energy is positive definite and attains its minimum  when $|w^\pm|=t^\pm$.    From the previous work about the classical Ginzburg-Landau equation in \cite{BrezisMelerRiviere}, we   would  look for solutions which satisfy
\begin{align*}
\int_{\mathbb{R}^N}\Big[\,A_+(|w^+|^2-{t^+}^2)^2+A_-(|w^-|^2-{t^-}^2)^2
+2B(|w^+|^2-{t^+}^2)(|w^-|^2-{t^-}^2)\,\Big] <+\infty.
\end{align*}
 Therefore  the function $w=(w^+, w^-)$ satisfies the superconducting boundary condition
\begin{align}  \label{asymptotic}
|w|^2= |w^+|^2+  |w^-|^2 \rightarrow {t^+}^2+{ t^-}^2=1,       \quad \text{when} \quad |x|\rightarrow +\infty.
 \end{align}
When $N=2,3$ in \eqref{original}, Ginzburg-Landau systems of this type have been introduced in physical models of $p$-wave superconductors \cite{AB} and
two-component Bose-Einstein condensates (BEC) \cite{AG, KMM}.
When the coupled term $B$ vanishes, system \eqref{original} is decoupled and is precisely the classical (single-component)  Ginzburg-Landau equation
\begin{align}  \label{classicalepsilon}
\epsilon^2 \Delta_{{\mathbb R}^N}  u + (1-|u|^2)u= 0,  \quad u\in \mathbb{C},
\end{align}
 which arises in the theory  of superconductivity  and superfluids \cite{GL} without external applied magnetic field. From the theory of  superconductivity,
   $|u|^2$   is proportional to the density of superconducting electrons, i.e,      $|u|^2 \approx 1$  corresponds to the superconducting state and
     $|u|^2 \approx 0$ corresponds to the  normal  state.
The zero set of $u$,   which the superconductivity is not present, is said to be vortex set or vorticity.
The  solutions  with vortex lines or vortex filaments to the system \eqref{original} are the main objects in this paper.

\medskip
\subsection{The case $B=0$}
For the classical Ginzburg-Landau equation in a smooth bounded domain $\Omega \subset \R^N$,
\begin{align}\label{eq:GLequationeps}
\begin{cases}
\epsilon^2\Delta u_\epsilon +(1-|u_\epsilon|^2)u_\epsilon=0  \quad \text{in}~ \Omega \subset \R^N,  \quad u_\epsilon \in \mathbb{C},
\\[3mm]  u_\epsilon =f \quad \text{on}\quad \p\Omega,
\end{cases}
\end{align}
the solutions with vortices have been widely studied  in the past  years.

\medskip
 For $N=2$, F. Bethuel, H. Brezis, and F. H\'elein \cite{BethuelBrezisHelein1994} considered the asymptotic behavior of $u_\epsilon$ with a boundary condition  $f: \p \Omega \rightarrow \mathbb{S}^1$ of degree $k\ge 1$. They showed that there exists a subsequence $\{\epsilon_j\}$  and exactly $k$ points  $b_1, \cdots b_k$ in $\Omega$ such that
\begin{align}\label{forma}
u_{\epsilon_j} (x) \rightarrow   {\bf u}=  e^{i\vp(x) }\prod_{l=1}^k   \frac{x-b_l}{|x-b_l|} \quad \text{in}~C_{loc} (\Omega\setminus\{b_1,  \cdots b_k\}),  \quad \text{as $\epsilon_j \rightarrow  0$},
\end{align}
where $\varphi$ is a harmonic function satisfying $\varphi|_{\p \Omega} = f. $
Furthermore, they also showed that the  $k$-tuple  $(b_1,  \cdots b_k)$ globally minimizes  the  {\em renormalized energy functional } $\mathcal W(b_1,  \cdots b_k)$ which is characterised by
\begin{align}
\mathcal W(b_1,  \cdots b_k)  = : \lim_{r\rightarrow 0}   \int_{\Omega\setminus \cup_{l=1}^k B_r(b_l)} \Big[ |\nabla {\bf u}|^2 - k\pi  |\ln r| \Big].
\end{align}
 And in \cite{BethuelBrezisHelein1994}, an implicit expression  of  $\mathcal W(b_1,  \cdots b_k)$ in terms of Green's  function  was given, see also \cite{delPinoFelmer1997, struwe} for similar results.
 One better approximation than ${\bf u}$ in \eqref{forma}  is
\be
 e^{i\vp(x) }\prod_{l=1}^k   \omega\Big(\frac{x-b_l}{\epsilon}\Big),
\ee
where $\omega$
 is the {\em standard degree $+1$ vortex solution} of the  following Ginzburg-Landau equation  in $\R^2$
\begin{equation} \Delta u + (1-|u|^2)u = 0,\quad u\in \mathbb{C}.
\label{stationary}
\end{equation}
In other words, equation \eqref{stationary} has a unique solution
of the form  \be\label{v1}
 \omega (\tilde z)= e^{i\theta} U(r), \quad \tilde z=re^{i\theta}, \quad U\in \R,\ee where $U>0$ is a solution of
\begin{equation*}
\left\{  \begin{aligned}
&U'' +  \frac {U'}r  - \frac U{r^2} +  (1-U^2) U  =  0 \inn (0,\infty),\\
& U(0^+) = 0 , \quad U(+\infty) = 1 ,
\end{aligned} \right.
\end{equation*} see \cite{ChenElliottQi1994, HerveHerve1994}.
On the other hand,  a natural question is to  find solutions  of \eqref{eq:GLequationeps}  with vortices at other critical points of the renormalized energy functional $\mathcal W(b_1,  \cdots b_k)$ (see \cite{almeida, delPinoKowalczykMusso2006, Lin19951, Lin19952, LinLin, PacardRiviere2000} and the references therein for researchs  on  this  problem).
  In \cite{PacardRiviere2000}, F. Pacard and T. Rivi\`ere     constructed  solutions of  \eqref{eq:GLequationeps}      with vortices  of combined  degree $\pm1$   and showed that  the vortex points  converge to non-degenerate critical point of corresponding   renormalized energy.
    In \cite{delPinoKowalczykMusso2006}, M.~del Pino, M.~Kowalczyk and M.~Musso considered  the Ginzburg-Landau equation  \eqref{eq:GLequationeps} with  added  zero Neumann  boundary condition  and constructed solutions  by using Lyapunov-Schmidt reduction without non-degenerate assumption on the critical point of renormalized energy.

\medskip
  In the high dimension $N\ge 3$,  the locations of vortices to the classical Ginzburg-Landau with suitable boundary conditions and energy levels do not occur at points, but  along a generalized  minimal sub-manifold structure with  co-dimension $2$, which is naturally interpreted as ``vortex sub-manifold".
   In $\R^3$,  the structure of vortices should typically be the form of  curves which are called vortex filaments.  In the work \cite{delPinoKowalczyk2008},   M.~del Pino and  M.~Kowalczyk studied \eqref{eq:GLequationeps} in a cylinder   $ \Omega  =B_R(0)\times (0,2\pi) $  in $\R^3$  and  formally   gave an  approximation  of  solution to \eqref{eq:GLequationeps} with  helical vortex filaments structure. Moreover, they conjectured in \cite{delPinoKowalczyk2008}  that \eqref{eq:GLequationeps}   has a solution $  u_{\epsilon {\bf g}}$ of the form
  \begin{align}  \label{conjectureapprox}
  u_{\epsilon {\bf g}} \approx  e^{i\vp(\tilde z, t) }\prod_{j=1}^k \omega\Big( \frac{\tilde z-g_j(t)} \epsilon \Big), \quad (\tilde z, t)\in \R^3,
  \end{align}
where  $\varphi$ is a harmonic function satisfying the boundary condition and $ {\bf g} =(g_1,\cdots,g_k)$ represents  $k$ curves : $ t\rightarrow \big(g_j(t), t\big),  j \leq k$. For simplicity, the authors assumed that the $k$ curves are $2\pi$-periodic in $t$.
The asymptotic expansion of  energy functional   $E_{0 \epsilon} (  u_{\epsilon {\bf g}} )  $ of \eqref{eq:GLequationeps}  is
  \be
I_\ve({\bf g} ):=  E_{0 \epsilon} (  u_{\epsilon {\bf g}} ) \approx    2\pi \times k\pi|\log\ve|  +   \mathcal I _\epsilon({\bf g}  ) , \nonumber
\ee
where
\be
E_{0 \epsilon}(u) = \frac 12  \int_\Omega |\nn u|^2    + \frac 1{4\ve^2} \int_\Omega (1-|u|^2)^2,\nonumber
\ee
and
\begin{equation*}
\mathcal I _\epsilon(\mathbf{ g}):=\pi \int_0^{2\pi}  \Big ( \, |\ln \epsilon| \frac12 \sum_{l=1}^k|g'_l(t)|^2-\sum_{j\neq l}\log |g_j(t)-g_l(t)| \, \Big )\,{\mathrm d}t.
\end{equation*}
Therefore,      the  asymptotic expressions
   for  equilibrium location     in \eqref{conjectureapprox} are
   \begin{align}\label{location of vortex line}
   g_j (t)  \approx \frac{1}{\sqrt{|\ln \epsilon| }} \sqrt{k-1} e^{i t}e^{2i(j-1)\pi/k}, \ \ j=1,\cdots,k,
   \end{align}
and  they   can be  obtained by solving the following   Euler-Lagrange equations of
    $\mathcal I _\epsilon(\mathbf{g})$
    \[ {\bf g}(t) =\frac{1}{\sqrt{|\ln \epsilon| }} \tilde{\bf g}(t), \quad  \tilde{\bf g} =(\tilde{g}_1,\cdots,\tilde{g}_k),  \quad   -\tilde{g}_l''(t)= 2\sum_{i\neq l}\frac{\tilde{g}_l(t)-\tilde{g}_i(t)}{|\tilde{g}_l(t)-\tilde{g}_i(t) |^2}.     \]
    Recently, J. D\'avila, M. del Pino, M. Medina and R. Rodiac  in \cite{DDMR}  have rigorously proved this conjecture in $\R^3$  by constructing  a family of entire solutions $u_\epsilon(\tilde z,t)$ as $\epsilon\rightarrow0$ with the form
      \begin{align*}
  u_{\epsilon} (\tilde z,t)\approx  \prod_{l=1}^k \omega\Big( \frac{\tilde z-g_j(t)} \epsilon \Big), \quad (\tilde z, t)\in \R^3.
  \end{align*}
Now we recall the main    results provided in \cite{DDMR}.

  \medskip
\begin{theorem}  (\cite{DDMR}).
For every integer  $k\geq 2$ and $\epsilon>0$ sufficiently small,  there exists a solution $u_\ve (\tilde z,t) \in \mathbb{C}$ of \eqref{classicalepsilon} with $N=3$, $2\pi$-periodic in the
$t$-variable,   with the following asymptotic profile:
\begin{equation*}
u_\epsilon(\tilde z,t)=\prod_{j=1}^k \omega \left(\frac{\tilde z-    {g}^\ve_j(t)}{\epsilon} \right)+  \varphi_\ve(\tilde z,t),
\end{equation*}
where ${g}_j^\ve(t)$ is $2\pi$-periodic with the asymptotic behavior \eqref{location of vortex line} and
\begin{equation*}
|\varphi_\epsilon(\tilde z,t)| \, \le\,\frac{C}{|\ln \epsilon|}.
\end{equation*}
\qed
\end{theorem}

 It is worth pointing out that  the solutions $  u_{\epsilon} (\tilde z,t)$ constructed in \cite{DDMR} satisfy
 \begin{equation}\label{uniform}
  \lim_{|\tilde z| \rightarrow +\infty}  | u_{\epsilon} (\tilde z,t) |= 1 \quad \text{uniformly in}~t,
  \end{equation}
 and still depend  on $t$.  As we know that H. Brezis has postulated the following  Gibbons conjecture: {\em whether a solution  $u_\epsilon (\tilde z,t)\in\mathbb{C}$ of  classical Ginzburg-Landau  with the uniform convergence condition \eqref{uniform} must necessarily be a function in $\tilde z$.}
However, the results obtained in \cite{DDMR} actually negatively answered the conjecture proposed by H. Brezis when $N=3$. The reader can refer to \cite{Gibbons conjecture}, \cite{GIBBONS} for more information about  the Gibbons conjecture.

\subsection{The case $B\ne0$}

When $N=2$, with space rescaling, we can reduce the system \eqref{original} into the following one:
  \begin{align} \label{2dgl}
 \begin{cases}
 -  \Delta w^+ +\Big[A_+\big(|w^+|^2-{t^+}^2\big)+B\big(|w^-|^2-{t^-}^2\big)\Big]w^+=0,
 \\[3mm]
 -  \Delta w^- +\Big[A_-\big(|w^-|^2-{t^-}^2\big)+B\big(|w^+|^2-{t^+}^2\big)\Big]w^-=0,
\end{cases}
  \quad  \text{in}~\R^2.
\end{align}
The two-component model \eqref{2dgl} was studied by \cite{AlamaBronsardMironescu1} and \cite{AlamaBronsardMironescu2} in the ``balanced" case, $A_+=A_-$ and $t^+=t^-=1/\sqrt{2}$.
S. Alama, L. Bronsard and P. Mironescu in \cite{AlamaBronsardMironescu1} and \cite{AlamaBronsardMironescu2} proved the existence, uniqueness, monotonicity and stability of radial
solutions with symmetric vortices.
Furthermore, under the assumption \eqref{H1},  S. Alama and  Q. Gao obtained the  existence,  uniqueness, asymptotic behaviors and quantization results of symmetric equivariant vortex solution   with degree pair $(n^+, n^-)$  in the form
\begin{align*}
w_n^+(\ell, \theta)=W_n^+(\ell)e^{in^+\theta},
\qquad
w_n^-(\ell, \theta)=W_n^-(\ell)e^{in^-\theta},
\end{align*}
 to  the system \eqref{2dgl}, where $(\ell, \theta)$ are the polar coordinates (see \cite{AG}).

 In order to give our results in a  more precise way, we use the notation in the rest of this paper to denote the radially symmetric solution  with  vortices of degree pair $(n_+,n_-)=(+1,+1)$ as below:
\begin{align*}
w(\tilde z)=\big(w^+(\tilde z), w^-(\tilde z)\big)
\quad\mbox{where}\quad
w^\pm(\tilde z)=W^\pm(\ell) e^{i\theta}, \quad  ~\tilde z=\ell e^{i\theta}.
\end{align*}
Then we get the corresponding ODE system
\begin{align}\label{1.2}
 \begin{cases}
-{W^+}''-\frac{1}{\ell}{W^+}'+\frac{1}{\ell^2}{W^+}+\Big[\,A_+({W^+}^2-{t^+}^2)+B({W^-}^2-{t^-}^2)\,\Big]W^+=0,\vspace{3mm}
\\
-{W^-}''-\frac{1}{\ell}{W^-}'+\frac{1}{\ell^2}{W^-}+\Big[\,A_-({W^-}^2-{t^-}^2)+B({W^+}^2-{t^+}^2)\,\Big]W^-=0,
 \end{cases}
\end{align}
where $\ell\in [0, +\infty)$.
 The facts below about  the properties of $W=(W^+, W^-)$  are given  in \cite{AG}.
\begin{lemma} \label{lemmaofW}   (\cite{AG}).
Suppose $W=(W^+, W^-)$ is a solution of the  ODE system \eqref{1.2}.Then
$W^+$ and $W^-$  satisfy the following:
\begin{align}\label{0011}
\begin{cases}
&W^\pm\in C^\infty(0,\infty),
\qquad
0< W^\pm<t^\pm  ~\text{for~all }~\ell>0,
\\[2mm]
&W^\pm\sim \ell ~\text{as }~\ell\rightarrow 0,
\qquad
W^\pm\sim t^\pm-\frac{c_\pm}{2\ell^2}~\text{as }~\ell\rightarrow \infty,
\\[2mm]
&{W^\pm}'>0~\text{when}~B<0,
\qquad
{W^\pm}'\sim \frac{c_\pm}{\ell^3}~\text{as }~\ell\rightarrow \infty,
\end{cases}
\end{align}
where
$$
c_\pm=\frac{A_\mp-B}{(A_+A_--B^2)t^\pm}.
$$
Moreover, $(W^+, W^-)$ is the unique solution of \eqref{1.2}.
\qed
 \end{lemma}

\medskip
The linear Ginzburg-Landau operator  of \eqref{2dgl}  around the standard vortex $w=(w^+,w^-)$ is given by
\[  L_0 (\phi) := \big(L_0^+, L_0^-\big)(\phi),   \]
with
\begin{align}
L_0^\pm  (\phi)   &=  \Delta \phi^\pm + \Big[ A_\pm \big({t^\pm}^2-{W^\pm}^2 \big)+B\big({t^\mp}^2- {W^\mp}^2\big)\Big]\phi^\pm  \nonumber
\\[2mm]
&\quad-2A_\pm {\rm{Re}}\Big(w^\pm\overline{\phi^\pm}\Big)w^\pm\,-\,2B {\rm{Re}}\Big(w^\mp \overline{\phi^\mp}\Big)w^\pm.
\end{align}
S. Alama and Q. Gao in \cite{AG2}  also  studied the stability   of $w(\tilde z)$ in the sense of spectrum of the  linear operator $L_0(\phi)$  in a disk of $\R^2$.
Furthermore, in \cite{Duan-Yang},   L. Duan and J. Yang studied the  linearized operator $L_0$. Under an extra assumption $B < 0$, they proved the non-degeneracy result of degree pair $(+1,+1)$ vortex solution in a natural Hilbert space $\mathscr H$  endowed with the norm $\|\cdot\|_{\mathscr H}$ in the following form
\begin{align}
    \| \phi\|_{\mathscr H}& = \int_{\R^2}  (    |\nn \phi^+|^2 +  |\nn \phi^-|^2)
   +    \int_{\R^2}  \Big[ A_+ \big({t^+}^2-{W^+}^2 \big)-B\big({t^-}^2- {W^-}^2\big)\Big]|\phi^+|^2
  \nonumber
\\[2mm]
& \quad
 +    \int_{\R^2}  \Big[ A_+ \big({t^-}^2-{W^-}^2 \big)-B\big({t^+}^2- {W^+}^2\big)\Big]|\phi^-|^2.
\end{align}
For convenience of reader, we   present the   non-degeneracy result   in the following  lemma.
 \begin{lemma}\label{nodegeneracy-theorem1.1} (\cite{Duan-Yang}).
Assume that (\ref{H1}) and $B<0$ hold.
Suppose that $L_0(\phi)=0$ for some $\phi\in \mathscr H$.
Then we have
$$
\phi=c_1\frac{\partial w}{\partial x_1}+c_2\frac{\partial w}{\partial x_2},
$$
for certain constants $c_1$ and $c_2$.
\qed
\end{lemma}

\medskip
Note that we also get the non-degeneracy result for $w$ with the degree pair $   (-1,-1), (-1,1) $ or $  (1,-1) $  in \cite{Duan-Yang} under the same constraints.
Furthermore,   in the Appendix,  we  show that  the  operator $L_0(\phi)$ does have one kernel if $\phi$   enjoys  some symmetry and decay assumptions,
see Lemmas \ref{lem:ellipticestimatesL0}-\ref{lem:ellipticestimatesL0-b}.

\medskip
There are  extensive researches  on solutions  with  vortex   structures  to the classical Ginzburg-Landau equation  with suitable boundary conditions, see \cite{BethuelBrezisOrlandi, jerrardsoner, LinRiviere1999, LinRiviere2001, Riviere1996}.  However,   for  the Ginzburg-Landau system,
much less is known about  the vortices, especially for solutions with vortex filaments in the case with  higher dimension $N\ge 3$.
The non-degeneracy arguments    provided  in \cite{Duan-Yang} for the vortex solutions with standard degree pair $(+1, +1)$ to the system \eqref{2dgl} make it possible to construct vortex solutions by using the singular perturbation methods.
Inspired by \cite{DDMR}, we  construct new solutions of \eqref{original}  with helical vortex filaments  in  $\R^3$.
We now  present the first result as follows:

\begin{theorem} \label{theorem1}
Assume that (\ref{H1}) and $B<0$ hold.
For any integer  $k \ge 2$ and $\epsilon$ small enough,  there exists a  solution
$w_\epsilon(\tilde z,t) = \big(w_\epsilon^+(\tilde z,t), w_\epsilon^-(\tilde z,t) \big)\in \mathbb{C}^2$
of \eqref{original} with  $N=3$  in the form
\begin{align}   \label{wepsilon}
w_\epsilon(\tilde z,t) = \Bigg(( t^+)^{1-k}\prod_{j=1}^{k}w^+\Big(\frac{\tilde z-g^+_{\epsilon, j}(t)} {\epsilon}\Big), \, ( t^-)^{1-k}\prod_{j=1}^{k}w^-\Big(\frac{\tilde z-g^-_{\epsilon, j}(t)} {\epsilon}\Big) \Bigg) + \phi_\epsilon(\tilde z,t),
\end{align}
where  $\phi_\epsilon(\tilde z, t) = \big( \phi_\epsilon^+(\tilde z,t), \phi_\epsilon^-(\tilde z,t)\big)\in \mathbb{C}^2$
and all $g^\pm_{\epsilon, j}(t)$'s are $2\pi$ periodic in variable $t$ which are given by
\begin{align}
 g^\pm_{\epsilon, j} (t)\approx\frac 1{\sqrt{|\ln \epsilon|}} \sqrt{k-1} e^{it} e^{i  \frac{2(j-1) \pi}{k}  }, \quad\text{for}~j=1,\cdots k.
\end{align}
Furthermore, we also have
\begin{align}
|\phi_\epsilon^\pm |  \le \frac{D}{|\ln \epsilon|}, \quad \text{for some constant $D>0$},
\end{align}
and
\begin{align}  \label{uniform convengence for system}
  \lim_{|\tilde z| \rightarrow +\infty}  | w_{\epsilon} (\tilde z,t) |=   \lim_{|\tilde z| \rightarrow +\infty}\sqrt{| w^+_{\epsilon} (\tilde z,t) |^2 + | w^-_{\epsilon} (\tilde z,t) |^2}= 1, \quad \text{uniformly in}~t.
\end{align}
\qed
\end{theorem}

\medskip
 Note that,  due to  \eqref{uniform convengence for system},  the solutions   given in Theorem \ref{theorem1} satisfies   the asymptotic property  \eqref{asymptotic}.
Theorem \ref{theorem1} also gives  a negative answer to an analogue  of  Gibbons conjecture (or De Giorgi  conjecture) for the   Ginzburg-Landau system.

\medskip
In Theorem \ref{theorem1}, we use the degree pair $(+1, +1)$ vortex solution $w$   as the  building block to construct  solutions.
Thus the solutions constructed in   Theorem \ref{theorem1}  consist in $k$ vortex helices of degree pair  $ (+1, +1)$.
Note that in \cite{Duan-Yang}, the authors showed that  $\hat w =(  w^+, \overline{w^-}) $,  the degree pair $(+1, -1)$ vortex solution of \eqref{2dgl},  is   non-degenerate.
Taking $\hat w$ as the building block and  using the same   method,  under the same conditions in Theorem \ref{theorem1} we can  construct solution    $\hat w_\epsilon(\tilde z,t) $  which  consists  in  $ k $   vortex helices of degree pair  $(1, -1)$  with the form
\begin{align}
\hat w_\epsilon(\tilde z,t)  \approx \Bigg(( t^+)^{1-k}\prod_{j=1}^{k}  w^+\Big(\frac{\tilde z-g^+_{\epsilon, j}(t)} {\epsilon}\Big), \, ( t^-)^{1-k}\prod_{j=1}^{k} \overline{ w^-}\Big(\frac{\tilde z-g^-_{\epsilon, j}(t)} {\epsilon}\Big) \Bigg).
\end{align}
%
%
%

\medskip
 Analogous to the work in \cite{DDMR},
 we also consider  solutions to \eqref{original}  which consist in $ k\ge 4$ vortex helices  of degree pair $(+1, +1)$  rotating around a straight vortex filament of degree pair $(-1,-1)$. The result is as below.

\begin{theorem}\label{theorem2}
Assume that (\ref{H1}) and $B<0$ hold.
For any integer $k\ge 4$,   there exists  $\epsilon_0 >0$  small such that for  every  $\epsilon <\epsilon_0$,
there exists a solution $w_\epsilon$ solving \eqref{original} with $N=3$.
Furthermore, the solution can be written in the form
\begin{align}
 w_\epsilon(\tilde z,t)  \approx \Bigg( ( t^+)^{-k}\, \overline{w^+} (\tilde z) \prod_{j=1}^{k} w^+\Big(\frac{\tilde z-  {\tilde g}^+_{\epsilon, j}(t)} {\epsilon}\Big),
 \,
 ( t^-)^{-k}\,\overline{w^-} (\tilde z)\prod_{j=1}^{k}  w^-\Big(\frac{\tilde z-{\tilde g}^-_{\epsilon, j}(t)} {\epsilon}\Big) \Bigg),
 \end{align}
 where $\tilde g^\pm_{\epsilon, j}$ is $2 \pi $ periodic in $t$ and satisfies  the asymptotic behavior
 \[ \tilde g^\pm_{\epsilon, j} (t)\approx\frac 1{\sqrt{|\ln \epsilon|}} \sqrt{k-3} e^{it} e^{i  \frac{2(j-1) \pi}{k}  }, \quad\text{for}~j=1,\cdots, k.\]
\qed
 \end{theorem}

\medskip
Theorems \ref{theorem1} and \ref{theorem2}   establish the existence  of  multiple helical  vortex filaments for Ginzburg-Landau system \eqref{original} in $\mathbb{R}^3$.
In fact, the   existence  of a solution with a  multiple-helices vortex  structure is an interesting question.
In   classical fluid and Bose-Einstein condensates, experiments and numerical methods show that     the possible existence of arbitrary number of interacting helical vortex filaments, see \cite{AKO, HK}.
The reader can also refer to   \cite{DDMR1, DDMW, JiangWangYang, LinWei, WeiJun2016} for the constructions of solutions with helical vortex structure, involving Schr\"odinger map equation, Gross-Pitaevskii equation and Euler equation.

\medskip
For  convenience, we  just provide  the proof of Theorem \ref{theorem1} in the present paper and Theorem \ref{theorem2} can be proved by an exactly similar way.
Here is the outline of procedure.

\smallskip  \noindent $\bullet$
Using  the {\em screw symmetry} in \eqref{conditionsforscrew},  we reduce the  original problem \eqref{original} to  a $2$-dimensional  one, see \eqref{GL2-dimensionalrescaled}.
This step will be  done in  Section \ref{section2.1}.

\smallskip  \noindent $\bullet$
Next, we will  construct an approximation to the real solutions of \eqref{GL2-dimensionalrescaled} which is denoted by $v_d=(v_d^+, v_d^-)$, see \eqref{eq:Approx2-dimensionalrescaled}-\eqref{eq:Approx2-dimensionalrescaledpm}.

\smallskip  \noindent $\bullet$
Note that the    coupling between the real part and  imaginary part  in   the linearized problem   of Ginzburg-Landau system \eqref{2dgl}  around $v_d$
would  make the setting very complicated.
We would use the idea in \cite{delPinoKowalczykMusso2006}  to overcome this obstacle.
More precisely speaking,  for a perturbation,
$$
\psi= (\psi^+, \psi^-) = (\psi^+_1+ i\psi^+_2, \psi^-_1+ i\psi^-_2) \in \mathbb{C}^2
$$
with symmetry \eqref{symetry1}, we take the solution in the form of \eqref{eq1}.
Then we  rewrite  the original problem to a  perturbation of the linearized problem (see \eqref{result})
 \[ \mathcal L_\epsilon\psi +\rR+ \mathcal N(\psi) =0,   \]
where $ \mathcal L_\e = \Big(\mathcal L_\e^+, \mathcal L_\e^-\Big)$ is the linearized operator,
$\rR=\Big(\rR^+, \rR^-\Big) $ is the error term,
and $ \mathcal N(\psi) = \Big(\mathcal N^+(\psi), \mathcal N^-(\psi) \Big) $ is the nonlinear term of higher order.
When  far away from  the vortices,
the linearized operator $\mathcal L_\e\psi$ behaves as below
\begin{align*}
\mathcal L_\e\psi  \approx    \begin{cases}
 \Delta \psi^+   - 2i A_\pm |v_d^+|^2 {\rm{Im}}(\psi^+) -2 i B   |v_d^-|^2 \rm{Im}(\psi^-)
 \\[3mm]
  \Delta \psi^-   - 2i A_\pm |v_d^+|^2 {\rm{Im}}(\psi^-) -2 i B   |v_d^-|^2 \rm{Im}(\psi^+).
 \end{cases}
\end{align*}
Although there are also  coupling between the  $``+"$ part and $``-"$ part,
 the operator   $\mathcal L_\e\psi $ is good far  away from vortices
due to the  positive definite  assumption \eqref{H1}.    On the other hand, when near the vortices,  from  Lemma \ref{Lepsilon},
Lemmas \ref{lem:ellipticestimatesL0}-\ref{lem:ellipticestimatesL0-b}, we would  deal with the problem by  setting up a projected nonlinear problem, see \eqref{eq:linear5.1}.

\smallskip  \noindent $\bullet$
Section \ref{section Error estimate of Approximation} is devoted to estimating  the  error in a suitable weighted norm.

\smallskip  \noindent $\bullet$
We study and solve  the  projected nonlinear problem \eqref{eq:linear5.1} by using a  fixed point argument.   This step is more or less  standard and will be done in Sections \ref{section44}-\ref{section6}.

\smallskip  \noindent $\bullet$
In order to finish the proof of Theorem \ref{theorem1},  in Section \ref{section77} the reduction process will be applied to force the vanishing of Lagrange multiplier $c$ in \eqref{eq:linear5.1}.
In other words, we need to adjust the parameter $\tilde d$ such that
 \[
{\rm Re}\int_{B_{\mathbb{R}_\e}(0)}   \Bigg( i\overline {w^\pm_{x_1}} (z) w^\pm (z)\Big[  \LLL^\pm_\e(\psi) +\rR^\pm+\mathcal N^\pm(\psi) \Big] (z+{\tilde d})   \Bigg)  =0.
\]
 The first two terms
  \[ {\rm Re}\int_{B_{\mathbb{R}_\e}(0)}   \Bigg( i\overline {w^\pm_{x_1}} (z) w^\pm (z)   \LLL^\pm_\e(\psi)  (z+{\tilde d})   \Bigg)\quad \text{and}\quad {\rm Re}\int_{B_{\mathbb{R}_\e}(0)}   \Bigg( i\overline {w^\pm_{x_1}} (z) w^\pm (z)    \rR^\pm (z+{\tilde d})   \Bigg) \]
   in the above identity  would be of  size $O(\e|\ln \e|)$, while the third term
  \[ {\rm Re}\int_{B_{\mathbb{R}_\e}(0)}   \Bigg( i\overline {w^\pm_{x_1}} (z) w^\pm (z)    \mathcal N^\pm(\psi) (z+{\tilde d})   \Bigg)  \]
   is of size $O(|\ln \e|^{-2})$.
Hence, much more delicate analysis will be provided to capture fine estimates.
 This will be bone by adopting the ideas  and  approaches  of Fourier decomposition to the error and  nonlinear terms introduced in \cite{DDMR, DDMR1} to  handle this problem, see Sections \ref{section3.2} and \ref{section4.2}.
In fact, a crucial  linear theory  which gives more precise estimates for the perturbation term $\psi$ will be established, see Proposition \ref{prop:sharp2b} in Section \ref{section4.2}.

\medskip
\noindent{\bf Notation and Definitions:}
We end this  introduction  by giving  some notation and definitions.
 For any complex-valued  vector  function $F = (F^+, F^-)\in \mathbb{C}^2$,  we define its conjugation
\[ \overline{F} = : (\overline{F^+},\, \overline{F^-}). \]
Next, we give the definition of screw-symmetry  in the coordinates $(r, \theta, t) $ with  $ (\tilde z, t) = (r e^{i \theta}, t)$.
\begin{definition}\label{screwsim}
We say that a function $u = (u^+, u^-):\R^3 \rightarrow \mathbb{C}^2$ has screw-symmetry  if
\begin{equation}\label{conditionsforscrew}
u(r,\theta+h,t+h)=u(r,\theta,t), \quad \text{i.e.,}\quad  u^\pm (r,\theta+h,t+h) = u^\pm(r,\theta,t),
\end{equation}
for any $h \in\R$.
\qed
\end{definition}

 \bigskip
\section{formulation of the problem} \label{section2}

 \subsection{Reduction to a 2-dimensional problem by using screw-symmetry}   \label{section2.1}

 \vspace{2mm}
 By using the  screw-symmetry, we will transform the problem  \eqref{original}, which is of $3$ dimensional in nature, to a $2$ dimensional  one.
 For simplicity, we only treat the case $k=2$ in Theorem \ref{theorem1} and the arguments  for general $k>2$ can be easily adapted.

\medskip
Notice that the condition \eqref{conditionsforscrew}  in Definition \ref{screwsim}  is equivalent to
$$
u(r,\theta,t+h)=u(r,\theta-h,t)\mbox{ for any }h\in \R.
$$
Then a  function with screw-symmetry  can be expressed as a function of two variables, i.e. for any \((r,\theta,t) \in \R^+\times\R\times\R\),
\[
u(r,\theta,t)=u(r,\theta-t,0)=:{\tilde u}(r,\theta-t).
\]
We will apply these facts to make a setting to reduce the dimension.

\medskip
Here is the observation in \cite{DDMR}.
We can write the standard degree pair $(+1, +1)$ vortex solution  for \eqref{2dgl}  in polar coordinates $(\ell, \theta)$, i.e.,
$$w(\tilde z)= \big(w^+(\tilde z), w^-(\tilde z)\big) =  \Big(W^+(r)e^{i\theta}, W^-(r)e^{i\theta}\Big),
$$
and consider a function $w_d$ in the form
\begin{align}\label{Approx1}
w_d(\tilde z,t)  &=  \Big(w_d^+(\tilde z,t), w_d^- (\tilde z,t)\Big),
\end{align} where
\begin{align}\label{Approx2}
 w_d^\pm (\tilde z,t)  =  ({t^\pm})^{-1}\prod_{j=1}^{2} w^\pm\left(\frac{{{\rm{Re}}(\tilde z)}-d \cos \big(t+(j-1)\pi\big) }{\epsilon},\,
 \frac{{{\rm{Im}}(\tilde z)}-d \sin\big(t+(j-1)\pi\big)}{\e}\right).
\end{align}
It can be checked that
\[
w_d(r,\theta,t+h)= \Big(e^{2ih}w^+_d(r,\theta-h,t), e^{ 2 i h} w^-_d(r,\theta-h,t)\Big)=e^{ 2 i h} w_d(r, \theta-h,t),
\]
for any \(h\) in \(\R\).
That is, $w_d$ does not have screw-symmetry  but $\tilde{w}_d(r,\theta,t):=e^{-2it}w_d(r,\theta,t)$ fulfills the symmetry.

\medskip
The above arguments suggest that we should look for a solution $w_\epsilon$ of \eqref{original} that can be written as
\begin{align}
w_\epsilon(r,\theta,t)= \Big(w_\epsilon^+(r,\theta,t), w_\epsilon^-(r,\theta,t) \Big) = \Big( e^{2 it} \tilde{w}^+(r,\theta,t),   e^{2it} \tilde{w}^-(r,\theta,t)\Big)
\label{transform1}
\end{align}
with $\tilde{w}^\pm$ being screw-symmetric.
Thus
\begin{align}
\tilde{w}^\pm(r,\theta,t)=  u^\pm(r,\theta-t),
\label{transform2}
\end{align}
where $u^\pm:\R^+\times \R \rightarrow \mathbb{C}$ is $2\pi$-periodic in the second variable.
Denoting
$$
u^\pm(r,s) =  u^\pm(r, \theta-t)\quad \mbox{with } s:= \theta-t,
$$
we can see that
$$
\p_{r}w_1^\pm = e^{2 it}\p_{r} u^\pm(r,s), \quad \p^2_{{r}{r}}w_1^\pm = e^{2 it} \p^2_{r r}  u^\pm(r,s),
\, \quad
\p_\theta w_1^\pm = e^{2 it}\p_s u^\pm(r,s), \quad \p^2_{\theta \theta} w_1^\pm = e^{2 it} \p^2_{ss} u^\pm(r,s),
$$

$$
\p_t w_1^\pm  = [2 i u^\pm -\p_s u^\pm]e^{2 it}, \ \ \  \p^2_{tt} w_1^\pm =[\p^2_{ss} u^\pm-4 i\p_s u^\pm-4 u^\pm ]e^{2 it}.
$$
Recalling the expression of the Laplacian operator in cylindrical coordinates as below
$$
\p^2_{r r}+\frac{1}{r}\p_{r}+\frac{1}{{r}^2}\p^2_{\theta \theta}+\p^2_{tt},
$$
we conclude that $w_\epsilon$ is a solution of \eqref{original} if and only if $u=\big(u^+, u^-\big)$ is a solution of the following equations
\begin{align}
 \e^2\Big( \p^2_{r r}u^\pm +\frac{1}{r}\p_{r}u^\pm  &+\frac{1}{{r}^2}\p^2_{ss}u^\pm  +\p^2_{ss}u^\pm -4 i\p_su^\pm- 4  u^\pm\Big)
\\[2mm]
 &  +\Big[ A_\pm \big({t}^\pm-|u^\pm|^2\big)+B\big({t}^\mp- |u^\mp|^2\big)\Big]u^\pm=0\quad \text{ in } \R^*_+\times \R.
\label{transform2-1}
\end{align}

\medskip
We  will   work in the rescaled coordinates, that is, we define
\begin{align}
v^\pm(r,s):=u^\pm(\e  r, s),
\label{transform3}
\end{align}
and search for a solution to the system of equations
\begin{align}\label{GL2-dimensionalrescaled}
\p^2_{rr}v^\pm +\frac{1}{r}\p_rv^\pm +\frac{1}{r^2}\p^2_{ss}v^\pm  &+\e^2(\p^2_{ss}v^\pm -4 i\p_sv^\pm - 4 v^\pm )  \nonumber
\\[2mm]
 & +\Big[ A_\pm \big({t}^\pm-|v^\pm|^2\big)+B\big({t}^\mp- |v^\mp|^2\big)\Big]v^\pm=0\quad\text{ in } \R^*_+\times \R.
\end{align}
From now on we will work in the plane $\R^2$ and use the notation $z=x_1+ix_2=re^{is}$. We denote by $\Delta$ the Laplacian operator in $2$-dimensional space, i.e.,
\begin{equation*}
\Delta=\p^2_{x_1x_1}+\p^2_{x_2x_2}=\p^2_{rr}+\frac{1}{r}\p_r+\frac{1}{r^2}\p^2_{ss}.
\end{equation*}

\medskip
 The approximation of a real solution to \eqref{GL2-dimensionalrescaled} can be chosen in the form
 \begin{align}  \label{eq:Approx2-dimensionalrescaled}
 v_d (z) = \big(v_d^+(z), v_d^-(z)\big)
 \end{align}
 with
 \begin{align}  \label{eq:Approx2-dimensionalrescaledpm}
 v_d^\pm = ({t^\pm})^{-1} \prod_{j=1}^{2}  w^\pm\Big(z-\tilde{d} e^{i (j-1)\pi }\Big),
 \end{align}
where   $ \tilde{d} =\frac{\hat d} {\e|\ln\e|^{\frac12} }    $ for some $\hat d =O(1)$ to be determined later and  $ w=(w^+, w^-) $
  is the standard degree $(+1, +1)$ vortex solution of  equation  \eqref{2dgl}.
   For convenience, we use the notation
\begin{align}
e_j= e^{i (j-1)\pi }, \quad j =1, 2.
\label{ej}
\end{align}

\begin{remark}
{\em
Starting from \eqref{eq:Approx2-dimensionalrescaled} and following the transformations in \eqref{transform1}, \eqref{transform2} and \eqref{transform3},
 we  claim that  by some  tedious but straightforward calculations,  we can get
  \begin{align*}
e^{2 it}   \prod_{j=1}^{2}  w^\pm\Big(\frac z \e- \tilde{d}   e^{i\frac{2(j-1)\pi} {2}} \Big) =  \prod_{j=1}^{2}  w^\pm\Big(\frac{\tilde z}\e- \tilde d  e^{it} e^{i\frac{2(j-1)\pi} {2}}\Big),
  \end{align*}
with
$$
z = re^{i(\theta-t)},\quad \tilde z =   r e^{i\theta}.
$$
Equivalently,  we  get  an approximation  which has $2$   vortex helices with degree pair $(+1, +1)$  to the system \eqref{original}   with the form
\begin{align}
\Bigg( ({t^+})^{-1} \prod_{j=1}^{2}  w^+\Big(\frac{\tilde z}\e- \tilde d  e^{it}  e^{i\frac{2(j-1)\pi} {2}}\Big),\   ({t^-})^{-1} \prod_{j=1}^{2}  w^-\Big(\frac{\tilde z}\e- \tilde d  e^{it}  e^{i\frac{2(j-1)\pi} {2}}\Big)\Bigg).
\label{profile1}
\end{align}
For the general integer $k >2$,  the corresponding  approximation to the solution of system \eqref{GL2-dimensionalrescaled}  would be
\begin{align}
\Bigg(  ( t^+)^{1-k}\prod_{j=1}^{k}  w^+\Big(z-\tilde{d} e^{i\frac{2(j-1)\pi} {k}} \Big),\   ( t^-)^{1-k} \prod_{j=1}^{k}  w^- \Big(z-\tilde{d}  e^{i\frac{2(j-1)\pi} {k}} \Big)\Bigg).
\label{profile2}
\end{align}
The asymptotic property \eqref{uniform convengence for system} of our solutions holds  from the construction and Lemma \ref{lemmaofW}.
The functions in \eqref{profile1} and \eqref{profile2} will provide the profile of the solution $w_\e$ in Theorem \ref{theorem1}.
\qed


}
\end{remark}

\medskip
\subsection{Additive-multiplicative perturbation}
Let
\[
S(v) :=  \Big( S^+(v), S^-(v)\Big),  \quad \text{for}~v=(v^+, v^-),
\]
with
\[S^\pm(v) =\Delta v^\pm
+\e^2(\p^2_{ss}v^\pm -4i\p_sv^\pm -4v^\pm) +\Big[ A_\pm \big({t}^\pm-|v^\pm|^2\big)+B\big({t}^\mp- |v^\mp|^2\big)\Big]v^\pm.   \]
The operator $S$ can be decomposed as
\begin{align*}
S  = S_0 + S_1\ ,
\end{align*}
where
\begin{align*}
S_0(v) =  \Big(S_0^+(v), S_0^-(v) \Big), \quad
S_1(v) =
 \Big(S_1^+(v), S_1^-(v) \Big),
\end{align*}
with
\begin{align*}
S_0^\pm(v)   = \Delta v^\pm
+\Big[ A_\pm \big({t}^\pm-|v^\pm|^2\big)+B\big({t}^\mp- |v^\mp|^2\big)\Big]v^\pm,
\end{align*}
\begin{align*}
  S_1^\pm(v)    &=  \e^2(\p^2_{ss}v^\pm -4i\p_sv^\pm -4v^\pm).
  \end{align*}
Therefore, the equation  \eqref{GL2-dimensionalrescaled} can be written as
\begin{align}
\label{mainEq}
S(v)=0  .
\end{align}

\medskip
We first try to search for a solution of \eqref{mainEq}  with the form
$$v= v_d + \phi, \quad \text{for some$~\phi = (\phi^+, \phi^-)$ small}. $$
We see that
\begin{align*}
S_0(v_d+ \phi ) &= S_0(v_d) + \LLL_0(\phi) + \NN_0(\phi),
\\[2mm]
S_1(v_d + \phi) &= S_1(v_d)  + S_1(\phi).
\end{align*}
In the above formulae, we have denoted the linear operator in the form
\begin{align*}
\LLL_0(\phi) &= \Big(\LLL_0^+, \LLL_0^-\Big) (\phi),
\end{align*}
where
\begin{align*}
\LLL_0^\pm  (\phi)   &=  \Delta \phi^\pm + \Big[ A_\pm \big({t^\pm}^2-|v_d^\pm|^2 \big)+B\big({t^\mp}^2-|v_d^\mp|^2\big)\Big]\phi^\pm  \nonumber
\\[2mm]
&\quad\,-\,2A_\pm {\rm{Re}}\Big(v_d^\pm\overline{\phi^\pm}\Big)v_d^\pm\,-\,2B {\rm{Re}}\Big(v_d^\mp \overline{\phi^\mp}\Big)v_d^\pm,
\end{align*}
and the nonlinear operator as the following
\begin{align*}
\NN_0(\phi) &= \Big(\NN_0^+, \NN_0^-\Big)(\phi),
\end{align*}
where
\begin{align*}
\NN_0^\pm(\phi) & = -2A_\pm {\rm{Re}} (v_d^\pm \overline{\phi^\mp}) \phi^\pm - A_\pm  |\phi^\pm|^2 (v_d^\pm+\phi^\pm)  \nonumber
  \\[2mm]
  & \quad  \quad -\,2B {\rm{Re}}\Big(v_d^\mp \overline{\phi^\mp} \Big) \phi^\pm - B  |\phi^\mp|^2 (v_d^\pm+\phi^\pm).
\end{align*}
We note that, due to the strong  coupling effect in the following terms
\[ -\,2A_\pm {\rm{Re}}\Big(v_d^\pm\overline{\phi^\pm}\Big)v_d^\pm\,-\,2B {\rm{Re}}\Big(v_d^\mp \overline{\phi^\mp}\Big)v_d^\pm,  \]
it is very complicated to handle the linear operator $\LLL_0(\phi)$.

\medskip
So, we follow the idea in the  work of  \cite{delPinoKowalczykMusso2006}
 to look for a solution of \eqref{mainEq} in the following form
\begin{align}
\label{eq1}
v=(v^+, v^-),
\qquad
v^\pm= \eta v_d^\pm (1+i\psi^\pm)  + (1-\eta) v_d^\pm e^{i\psi^\pm},
\end{align}
where
$v_d=(v_d^+, v_d^-)$  is  the ansatz \eqref{eq:Approx2-dimensionalrescaled}
 and $\psi= (\psi^+, \psi^-) = (\psi^+_1+ i\psi^+_2, \psi^-_1+ i\psi^-_2) $ is the unknown perturbation.
The cut-off function $\eta$ is defined as
\begin{align*}
\eta(z) = \eta_1(|z-\tilde d|) + \eta_1(|z+\tilde d|) , \quad z\in \mathbb C\cong\R^ 2\ ,
\end{align*}
and  $\eta_1:\R \rightarrow [0,1]$ is a smooth cut-off function such that
\begin{align}
\label{eta1}
 \eta_1(t)=1 \text{ for } t\leq 1\text{ and }\eta_1(t)=0 \text{ for } t\geq 2.
\end{align}
  Our  goal  is to rewrite \eqref{mainEq} in the form
\begin{align} \label{equation-for-psi}
\mathcal L_\e\psi +\rR+ \mathcal N(\psi)=0
\end{align}
by identifying the linear operator $\mathcal L_\e = \Big(\mathcal L_\e^+, \mathcal L_\e^-\Big)$,
the error term $\rR=\Big(\rR^+, \rR^-\Big)$
and the nonlinear term  $\mathcal N(\psi) = \Big(\mathcal N^+(\psi), \mathcal N^-(\psi) \Big) $.
This will be done in the sequel and the result is given in \eqref{result}.

\medskip
Rewrite  $v$ in \eqref{eq1} as
\begin{align*}
v &= v_d + \phi  = \big(v_d^+, v_d^-\big) + \big(\phi^+, \phi^-\big)
\end{align*}
with
\begin{align*}
\phi^\pm  &= i v_d^\pm \psi^\pm + \gamma(\psi^\pm), \quad
 \gamma(\psi^\pm) = (1-\eta) v_d^\pm (e^{i\psi^\pm}-1-i\psi^\pm).
\end{align*}
Then
\begin{align*}
S_0(v) &= S_0(v_d) + \LLL_0\Big( \big(i v_d^+ \psi^+, i v_d^- \psi^-\big) \Big) + \LLL_0\Big( \big( \gamma(\psi^+), \gamma(\psi^-)   \big)\Big)
\\[2mm]
 & \quad
+ \NN_0\Big(  \big(i v_d^+ \psi^+, i v_d^- \psi^-\big)  + \big( \gamma(\psi^+), \gamma(\psi^-)   \big)\Big),
\\[3mm]
S_1(v) &= S_1(V_d) + S_1 \Big(  \big(i v_d^+ \psi^+, i v_d^- \psi^-\big) \Big) + S_1\Big( \big( \gamma(\psi^+), \gamma(\psi^-)   \big)\Big).
\end{align*}
Explicit  expressions of  some operators in the above will be derived one by one.

\smallskip
\noindent $\bullet$
By some computations, we start from $S_0(v)$ and consider
\begin{align*}
 \LLL_0\Big( \big(i v_d^+ \psi^+, i v_d^- \psi^-\big) \Big)  =  \Big(\LLL_0^+, \LLL_0^-\Big)  \Big( \big(i v_d^+ \psi^+, i v_d^- \psi^-\big) \Big)
 \end{align*}
where
 \begin{align*}
 & \LLL_0^\pm\Big( \big(i v_d^+ \psi^+, i v_d^- \psi^-\big) \Big)
  \nonumber
\\[2mm]
 &=\Delta  (i v_d^\pm \psi^\pm)  + \Big[ A_\pm \big({t^\pm}^2-|v_d^\pm|^2 \big)+B\big({t^\mp}^2-|v_d^\mp|^2\big)\Big]  (i v_d^\pm \psi^\pm)   \nonumber
\\[2mm]
& \qquad \qquad \,-\,2A_\pm {\rm{Re}}\Big(v_d^\pm\overline{  (i v_d^\pm \psi^\pm)}\Big)v_d^\pm\,-\,2B {\rm{Re}}\Big(v_d^\mp \overline{ (i v_d^\mp \psi^\mp) }\Big)v_d^\pm
\\[2mm]
 &= i  v_d^\pm \Bigg(  \frac{\Delta v_d^\pm}{v_d^\pm}     + \Big[ A_\pm \big({t^\pm}^2-|v_d^\pm|^2 \big)+B\big({t^\mp}^2-|v_d^\mp|^2\big)\Big]    \Bigg)    \psi^\pm
 \\[2mm]
 &  \quad  +i  v_d^\pm   \Bigg( \Delta \psi^\pm +\frac{\nabla v_d^\pm \nabla \psi^\pm } {v_d^\pm}  - 2i A_\pm |v_d^\pm|^2  {\rm{Im}}(\psi^\pm) -2 i B   |v_d^\mp|^2 \rm{Im}(\psi^\mp)  \Bigg)
  \\[2mm]
 &= i  v_d^\pm \Bigg(  \frac{\Delta v_d^\pm}{v_d^\pm}     + \Big[ A_\pm \big({t^\pm}^2-|v_d^\pm|^2 \big)+B\big({t^\mp}^2-|v_d^\mp|^2\big)\Big]    \Bigg)    \psi^\pm
 \\[2mm]
  &  \quad +i  v_d^\pm   \Bigg( \Delta \psi^\pm +\frac{\nabla v_d^\pm \nabla \psi^\pm } {v_d^\pm}  - 2i A_\pm |v_d^\pm|^2 {\rm{Im}}(\psi^\pm) -2 i B   |v_d^\mp|^2{ \rm{Im}}(\psi^\mp)  \Bigg)
    \\[2mm]
 &= i  v_d^\pm   \Bigg(   \frac{S_0^\pm(v_d)} {v_d^\pm}  \psi^\pm +  \Delta \psi^\pm +      \overbrace{ \textcolor{red}{\frac{\nabla v_d^\pm \nabla \psi^\pm } {v_d^\pm}  }}^{ \text{weak effect} }  - 2i A_\pm |v_d^\pm|^2 {\rm{Im}}(\psi^\pm) -2 i B   |v_d^\mp|^2 \rm{Im}(\psi^\mp)  \Bigg)
  \\[2mm]
 &:= i  v_d^\pm   \Bigg(   \frac{S_0^\pm (v_d)} {v_d^\pm}  \psi^\pm +  \tilde L^\pm_0(\psi )   \Bigg),
  \end{align*}
  with
  \begin{align}   \label{tilde L0pm}
   \tilde L^\pm_0(\psi )  = \Delta \psi^\pm +      \overbrace{ \textcolor{red}{\frac{\nabla v_d^\pm \nabla \psi^\pm } {v_d^\pm}  }}^{ \text{weak effect} }  - 2i A_\pm |v_d^\pm|^2 {\rm{Im}}(\psi^\pm) -2 i B   |v_d^\mp|^2 \rm{Im}(\psi^\mp).
  \end{align}
   Thus we have
  \begin{align*}
S^\pm_0(v) &= S^\pm_0(v_d) + \LLL^\pm_0\Big( \big(i v_d^+ \psi^+, i v_d^- \psi^-\big) \Big) + \LLL^\pm_0\Big( \big( \gamma(\psi^+), \gamma(\psi^-)   \big)\Big)
\\[2mm]
 & \quad
+ \NN^\pm_0\Big(  \big(i v_d^+ \psi^+, i v_d^- \psi^-\big)  + \big( \gamma(\psi^+), \gamma(\psi^-)   \big)\Big)
\\[2mm]
 &  = i  v_d^\pm   \Bigg(  -i \frac{S^\pm_0(v_d) } { v_d^\pm }   + \frac{S_0^\pm (v_d)} {v_d^\pm}  \psi^\pm +  \tilde L^\pm_0(\psi )     - i \frac{1}{v_d^\pm } \LLL^\pm_0\Big( \big( \gamma(\psi^+), \gamma(\psi^-)   \big)\Big)
\\[2mm]
 & \quad
    - i \frac{1}{v_d^\pm }  \NN^\pm_0\Big(  \big(i v_d^+ \psi^+, i v_d^- \psi^-\big)  + \big( \gamma(\psi^+), \gamma(\psi^-)   \big)\Big)  \Bigg).
\end{align*}
Specially,  when far away from the vortices,   $v= \big(v_d^+ e^{i\psi^+}, v_d^-e^{i\psi^-}\big)$. This implies that
 \begin{align}
  S_0(v) =  \Bigg(S_0^+\Big(\big(v_d^+ e^{i\psi^+}, v_d^-e^{i\psi^-}\big) \Big), S_0^-\Big(\big(v_d^+ e^{i\psi^+}, v_d^-e^{i\psi^-}\big)\Big) \Bigg),
\label{S0(v)far}
 \end{align}
and
 \begin{align*}
   &S_0^\pm\Big(\big(v_d^+ e^{i\psi^+}, v_d^-e^{i\psi^-}\big) \Big)    \nonumber
 \\[2mm]
 & =   \Delta (v_d^\pm e^{i\psi^\pm})
+\Big[ A_\pm \big({t^\pm}^2-| (v_d^\pm e^{i\psi^\pm}) |^2\big)+B\big({t^\mp}^2-| (v_d^\mp e^{i\psi^\mp}) |^2\big)\Big]  (v_d^\pm e^{i\psi^\pm})
 \nonumber
 \\[2mm]
 & =  i v_d^\pm e^{i\psi^\pm} \Big( - i \frac{S_0^\pm({v_d})}{{v_d^\pm}} +\tilde L^\pm_0(\psi )    +  \tilde N^\pm_0(\psi ) \Big),
 \end{align*}
 where
 \[ \tilde N^\pm_0(\psi )  = i (\nabla \psi^\pm)^2 + i A_\pm|v_d^\pm  |^2 (e^{-2\psi^\pm_2} - 1+ 2\psi^\pm_2)  + i B|v_d^\mp  |^2 (e^{-2\psi^\mp_2} - 1+ 2\psi^\mp_2).    \]


\medskip
\noindent $\bullet$
We consider
 \begin{align*}
  S_1 \Big(  \big(i v_d^+ \psi^+, i v_d^- \psi^-\big) \Big)
 = \Bigg(S_1^+\Big(  \big(i v_d^+ \psi^+, i v_d^- \psi^-\big) \Big) , S_1^-\Big(  \big(i v_d^+ \psi^+, i v_d^- \psi^-\big) \Big)  \Bigg),
  \end{align*}
 and it follows that
    \begin{align*}
   S_1^\pm \Big(  \big(i v_d^+ \psi^+, i v_d^- \psi^-\big) \Big) =   iv _d^\pm
\left(
\frac{S^\pm_1(v_d)}{v _d^\pm }\psi^\pm
+\tilde L_1^\pm (\psi )
\right)
    \end{align*}
    with
    \begin{align}\label{tilde L1pm}
\tilde L_1^\pm  (\psi) = \e^2 \Bigl( \partial_{ss}^2 \psi^\pm
+\frac{2\partial_s v_d^\pm  }{v_d^\pm }\partial_s \psi^\pm
- 4i \partial_s \psi^\pm   \Bigr) .
\end{align}
Specially, we have that, when far away from the vortices
\begin{align}
S_1(v) = \Bigg(S_1^+\Big(\big(v_d^+ e^{i\psi^+}, v_d^-e^{i\psi^-}\big) \Big), S_1^-\Big(\big(v_d^+ e^{i\psi^+}, v_d^-e^{i\psi^-}\big)\Big) \Bigg),
\label{S1(v)far}
\end{align}
with
 \begin{align*}
 S_1^\pm\Big(\big(v_d^+ e^{i\psi^+}, v_d^-e^{i\psi^-}\big) \Big) =
  i v^\pm_d e^{i\psi^\pm}
\left[
-i\frac{S_1^\pm(v_d)}{v^\pm_d}
+ \tilde L_1^\pm (\psi)
+\e^2 i (\partial_s\psi^\pm)^2
\right] .
 \end{align*}

\medskip
In order to draw a conclusion, a new cut-off function will be introduced as the following
\[
\tilde \eta(z ) = \eta_1(|z-\tilde d|-1) + \eta_1(|z+\tilde d|-1)
\]
with $\eta_1$ defined in  \eqref{eta1}.
Then the problem  \eqref{mainEq} can be rewritten as
\begin{align*}
0 & =\tilde\eta   i {v_d^\pm}
\Bigg[
-i\frac{S_0^\pm({v_d})}{{v_d^\pm}}
+ \tilde L_0^\pm (\psi)
+ \frac{S_0^\pm({v_d})}{{v_d^\pm}}\psi^\pm  -i \frac{S_1^\pm({v_d})}{{v_d^\pm}}
+ \tilde L_1^\pm (\psi)
+ \frac{S_1^\pm({v_d})}{{v_d^\pm}}\psi^\pm
\\[2mm]
& \qquad  - \frac{i}{{v_d^\pm}}  \LLL_0^\pm\Big( \big( \gamma(\psi^+), \gamma(\psi^-)   \big)\Big)
- \frac{i}{{v_d^\pm}}\NN_0^\pm \Big(  \big(i v_d^+ \psi^+, i v_d^- \psi^-\big)  + \big( \gamma(\psi^+), \gamma(\psi^-)   \big)\Big)
\\[2mm]
&  \qquad
- \frac{i}{{v_d^\pm}} S_1^\pm\Big( \big( \gamma(\psi^+), \gamma(\psi^-)   \big)\Big)
\Bigg]
\\[2mm]
& \quad
+ (1-\tilde\eta)
i {v_d^\pm} e^{i \psi^\pm}
\left[
- i \frac{S_0^\pm({v_d})}{{v_d^\pm}}
+ \tilde L_0^\pm (\psi)
+ \tilde N_0^\pm(\psi)
- i \frac{S_1^\pm({v_d^\pm})}{{v_d^\pm}}
 + \tilde L_1 (\psi^\pm) + \e^2 i (\partial_s \psi^\pm)^2
\right].
 \end{align*}
 By the definitions of $\mathcal{L}, \mathcal{R}$ and $\mathcal{N}$, we can have a simpler representation of the above equation for \eqref{mainEq}:
\begin{align}
\mathcal L_\e\psi +\rR+ \mathcal N(\psi)=0, \quad i.e.  \quad \mathcal L_\e^\pm\psi +\rR^\pm+ \mathcal N^\pm(\psi)=0,
\label{result}
\end{align}
 where
\begin{align}
\nonumber
\mathcal L_\e^\pm (\psi)
& :=
 (\tilde L_0^\pm + \tilde L_1^\pm)  (\psi )
+ \tilde\eta \frac{S^\pm(v_d)}{v_d^\pm}\psi^\pm ,
\qquad\quad
\rR^\pm := -i \frac{S(v_d^\pm)}{v_d^\pm} ,
\nonumber
\\[2mm]
\nonumber
\mathcal N^\pm(\psi) &:=
\tilde\eta\Bigl( \frac{1}{\tilde \eta + (1-\tilde\eta) e^{i\psi^\pm} }-1\Bigr) \frac{S^\pm(v_d)}{v_d^\pm}\psi^\pm
\\[2mm]
\nonumber
& \qquad
- \frac{i}{v_d^\pm} \frac{\tilde\eta}{\tilde\eta + (1-\tilde\eta) e^{i\psi^\pm} } \Bigg[  \LLL_0^\pm\Big( \big( \gamma(\psi^+), \gamma(\psi^-)   \big)\Big)    + S_1^\pm\Big( \big( \gamma(\psi^+), \gamma(\psi^-)   \big)\Big)
 \\[2mm]
\nonumber
& \qquad\qquad +  \NN_0^\pm \Big(  \big(i v_d^+ \psi^+, i v_d^- \psi^-\big)    + \big( \gamma(\psi^+), \gamma(\psi^-)   \big)\Big)\Bigg]
  \\[2mm]
\nonumber
& \qquad + \frac{(1-\tilde\eta)e^{i\psi^\pm} }{ \tilde\eta + (1-\tilde\eta) e^{i\psi^\pm} } \Bigl[\tilde N_0^\pm(\psi)
+\e^2 i (\partial_s\psi^\pm)^2 \Bigr].
\end{align}
From  \eqref{tilde L0pm} and \eqref{tilde L1pm},  we know that
\begin{align*}
\mathcal L_\e^\pm (\psi)   &=  \Delta \psi^\pm +      {\frac{\nabla v_d^\pm \nabla \psi^\pm } {v_d^\pm}  }   - 2i A_\pm |v_d^\pm|^2 {\rm{Im}}(\psi^\pm) -2 i B   |v_d^\mp|^2 \rm{Im}(\psi^\mp)
  \\[2mm]
\nonumber
& \quad
 + \e^2 \Bigl( \partial_{ss}^2 \psi^\pm
+\frac{2\partial_s v_d^\pm  }{v_d^\pm }\partial_s \psi^\pm
- 4i \partial_s \psi^\pm   \Bigr) + \tilde\eta \frac{S^\pm(v_d)}{v_d^\pm}\psi^\pm.
\end{align*}
  Note that when
  $|z\pm \tilde d|\geq 3$ the nonlinear terms take the form
\begin{align*}
\mathcal{N}^\pm(\psi) &=
\tilde N_0^\pm(\psi)
+\e^2 i (\partial_s\psi^\pm)^2
\\[2mm]
&=
  i (\nabla \psi^\pm)^2 + i A_\pm|v_d^\pm  |^2 (e^{-2\psi^\pm_2} - 1+ 2\psi^\pm_2)  + i B|v_d^\mp  |^2 (e^{-2\psi^\mp_2} - 1+ 2\psi^\mp_2)
 +\e^2 i (\partial_s\psi^\pm)^2.
\end{align*}

\medskip
\subsection{Another form of the equation near each vortex}  \label{section3.3}
Note that the expressions of the terms $S_0(v)$ and $S_1(v)$ in the outer region (far from the vortices) are given in \eqref{S0(v)far} and \eqref{S1(v)far}.
In this part, we will write the  operator $\mathcal L_\e\psi $   in another form  when it is near the vortices so that we can  know it better.

In order to do so, we first introduce some notation.
 We denote
$$
{z}_j = z- \tilde{d} e^{i\frac{2(j-1)\pi} {2}}\qquad  \mbox{for }j=1, 2,
$$
and  then let
 \[
 \Phi_j(z) =    \big( \Phi_j^+(z),  \Phi_j^-(z) \big), \qquad \Phi_j^\pm(z) =i w^\pm({z}_j)  \psi^\pm(z).
 \]
 Furthermore,  when near the vortex  center,  we also assume
\begin{align*}
  & \Phi(z) =  \big( \Phi^+(z),  \Phi^-(z) \big), \qquad \Phi^\pm(z) =i  v_d^\pm(z) \psi^\pm(z),
 \\[1mm]
&  \Omega_j(z)= \big(  \Omega^+_j(z),   \Omega^-_j(z) \big),  \qquad \Omega_j^\pm(z) =  \frac{ v_d^\pm(z)} {w^\pm({z}_j) },
 \\[1mm]
&   \mathbb{E}:= S(v_d), \quad \qquad\mathbb{E}^\pm= S^\pm(v_d).
 \end{align*}
By some calculations, the following identity holds
 \begin{align*}
 iw^\pm(z_j)\mathcal L_\e^\pm \psi (z)   &=  \frac{L^\pm_d(\Phi)} {\Omega_j^\pm}(z)  + (\eta_1-1)\frac{\mathbb{E}^\pm} {v_d^\pm}\Phi_j^\pm(z)    \\[2mm]
 \nonumber
&  =  \frac{L^\pm_d \Big( \big(\Phi^+_j \Omega^+_j,\ \Phi^-_j \Omega^-_j\big)\Big)} {\Omega_j^\pm}(z)  + (\eta_1-1)\frac{\mathbb{E}^\pm} {v_d^\pm}\Phi_j^\pm (z)\ ,
 \end{align*}
 where
\begin{align*}
 L^\pm_d(\Phi)  & =   \Delta \Phi^\pm + \Big[ A_\pm \big({t^\pm}^2-|v_d^\pm|^2 \big)+B\big({t^\mp}^2-|v_d^\mp|^2\big)\Big]\Phi^\pm  \nonumber
\\[2mm]
&\quad \,-\,2A_\pm {\rm{Re}}\Big(v_d^\pm\overline{\Phi^\pm}\Big)v_d^\pm\,-\,2B {\rm{Re}}\Big(v_d^\mp \overline{\Phi^\mp}\Big)v_d^\pm
 +
\e^2(\p^2_{ss}\Phi^\pm -4i\p_s\Phi^\pm -4\Phi^\pm).
\end{align*}
Recall that
 \[\Phi = \big(iv_d^+ \psi^+, i v_d^- \psi^-\big),  \]
and define
 \begin{align}\label{ldj}
 L_{d,j}(\Phi) = \Big(L^+_{d,j}, L^-_{d,j}\Big) (\Phi), \quad \text{with}\quad  L^\pm_{d,j}(\Phi) :=  iw^\pm(z_j)\mathcal L_\e^\pm (\psi).
 \end{align}

\medskip
 In order to describe the properties of ${\mathcal L}_\e^\pm (\psi)$ near the vortex, we  give the following lemma.
 \begin{lemma}  \label{Lepsilon}
 When near the vortices, we can see that the linear operator $L_{d,j}(\Phi) = \Big(L^+_{d,j}, L^-_{d,j}\Big) (\Phi) $ defined in \eqref{ldj}  is a small perturbation of $L^0(\Phi)$, where $ L^0$ is the linearized Ginzburg-Landau operator around $\Big( w^+(z_j),  w^-(z_j)\Big)$, i.e,
 \begin{equation}\nonumber
L^0(\phi):=  \big({L^0_+}(\phi),  {L^0_-}(\phi)\big) \quad\mbox{with } \phi=(\phi^+, \phi^-),
\end{equation}
and
\begin{align}
{L^0_\pm}  (\phi)  &=   \Delta \phi^\pm \,+\,\Big[\,A_\pm\big({t}^\pm-|{W^\pm(|z_j|)}|^2\Big)+B\Big({t^\mp}^2- |{W^\mp(|z_j|) }|^2 \big)\,\Big]\phi^\pm\nonumber
\\[2mm]
&\quad \,-\,2A_\pm {\rm{Re}}\Big(w^\pm\overline{\phi^\pm}\Big)w^\pm (z_j)\,-\,2B {\rm{Re}}\Big(w^\mp \overline{\phi^\mp}\Big)w^\pm (z_j).
\end{align}

 \end{lemma}

 \begin{proof} From  the definition, we  know that
 \begin{align*}
  L^\pm_{d,j}(\Phi)
&=   \frac{L^\pm_d \Big( \big(\Phi^+_j \Omega^+_j, \Phi^-_j \Omega^-_j\big)\Big)} {\Omega_j^\pm}  + (\eta_1-1)\frac{\mathbb{E}^\pm} {v_d^\pm}\Phi_j^\pm
  \nonumber
\\[2mm]
& =  \frac1{\Omega_j^\pm}  \Bigg\{ \Delta  (\Phi^\pm_j \Omega^\pm_j) + \Big[ A_\pm \big({t^\pm}^2-|v_d^\pm|^2 \big)+B\big({t^\mp}^2-|v_d^\mp|^2\big)\Big] (\Phi^\pm_j \Omega^\pm_j)
  \nonumber
\\[2mm]
&\qquad\quad \,-\,2A_\pm {\rm{Re}}\Big(v_d^\pm\overline{ (\Phi^\pm_j \Omega^\pm_j) }\Big)v_d^\pm\,-\,2B {\rm{Re}}\Big(v_d^\mp \overline{ (\Phi^\mp_j \Omega^\mp_j) }\Big)v_d^\pm
  \nonumber
\\[2mm]
&\qquad\quad
+
\e^2\Big[\p^2_{ss} (\Phi^\pm_j \Omega^\pm_j)-4i\p_s (\Phi^\pm_j \Omega^\pm_j) -4 (\Phi^\pm_j \Omega^\pm_j) \Big] \Bigg\}
\ +\
(\eta_1-1)\frac{\mathbb{E}^\pm} {v_d^\pm}\Phi_j^\pm
  \nonumber
\\[2mm]
&=  \frac1{\Omega_j^\pm}    \Big(  \Phi^\pm_j \Delta \Omega^\pm_j + 2 \nabla \Phi^\pm_j \nabla  \Omega^\pm_j +  \Omega^\pm_j \Delta  \Phi^\pm_j \Big) + (\eta_1-1)\frac{\mathbb{E}^\pm} {v_d^\pm}\Phi_j^\pm
  \nonumber
\\[2mm]  & \quad  +
 \Big[ A_\pm \big({t^\pm}^2-|v_d^\pm|^2 \big)+B\big({t^\mp}^2-|v_d^\mp|^2\big)\Big]\Phi^\pm_j
  \nonumber
\\[2mm]  & \quad   +   \frac1{\Omega_j^\pm}  \Big[\,-\,2A_\pm {\rm{Re}}\Big(v_d^\pm\overline{ (\Phi^\pm_j \Omega^\pm_j) }\Big)v_d^\pm\,-\,2B {\rm{Re}}\Big(v_d^\mp \overline{ (\Phi^\mp_j \Omega^\mp_j) }\Big)v_d^\pm       \Big]
  \nonumber
\\[2mm]  & \quad   +   \e^2 \frac1{\Omega_j^\pm}   \Big[ \Omega^\pm_j \p^2_{ss}  \Phi^\pm_j +\Phi^\pm_j\p^2_{ss}  \Omega^\pm_j + 2\p_s\Omega^\pm_j\p_s \Phi^\pm_j -4i (\p_s \Omega^\pm_j \Phi^\pm_j+ \p_s \Phi^\pm_j\Omega^\pm_j )-4 (\Phi^\pm_j \Omega^\pm_j) \Big]
   \nonumber
\\[2mm]  & : = \A_1+ \A_2+\A_3+ (\eta_1-1)\frac{\mathbb{E}^\pm} {v_d^\pm}\Phi_j^\pm,
  \end{align*}
  with
  \begin{align*}
   \A_1 & =\Delta  \Phi^\pm_j  +
 \Big[ A_\pm \big({t^\pm}^2-|v_d^\pm|^2 \big)+B\big({t^\mp}^2-|v_d^\mp|^2\big)\Big]\Phi^\pm_j
    \nonumber
\\[2mm]
& \quad  \,-\,2A_\pm {\rm{Re}}\Big(v_d^\pm\overline{ (\Phi^\pm_j \Omega^\pm_j) }\Big) w^\pm(z_j) \,-\,2B {\rm{Re}}\Big(v_d^\mp \overline{ (\Phi^\mp_j \Omega^\mp_j) }\Big)    w^\pm(z_j),
 \end{align*}
 \[\A_2= \e^2 \Big[  \p^2_{ss}  \Phi^\pm_j  -4i   \p_s \Phi^\pm_j  -4 \Phi^\pm_j  \Big]+\frac1{\Omega_j^\pm} \Big[ 2 \nabla \Phi^\pm_j \nabla  \Omega^\pm_j +2 \e^2 \p_s\Omega^\pm_j\p_s \Phi^\pm_j \Big],
 \]
 \[ \A_3 =\frac1{\Omega_j^\pm} \Big[ \Delta \Omega^\pm_j+\e^2\p^2_{ss}  \Omega^\pm_j  -4i \e^2\p_s \Omega^\pm_j \Big] \Phi^\pm_j. \]
For $\A_1 $, we have
\begin{align*}
   \A_1   & = {L^0_\pm}  (\Phi_j)  +\Big[ A_\pm \big( |w^\mp(z_j)|^2-|v_d^\pm|^2 \big) + B\big( |w^\mp(z_j)|^2-|v_d^\mp|^2 \big) \Big]   \Phi^\pm_j
\nonumber
\\[2mm]  & \quad   +2A_\pm {\rm{Re}}\Big(w^\pm(z_j)\overline{ \Phi^\pm_j  }\Big)  w^\pm(z_j)  +2B {\rm{Re}}\Big(w^\mp(z_j) \overline{ \Phi^\mp_j }\Big)   w^\pm(z_j)
\nonumber
\\[2mm]  & \quad
\,-\,2A_\pm {\rm{Re}}\Big(v_d^\pm\overline{ (\Phi^\pm_j \Omega^\pm_j) }\Big)  w^\pm(z_j)-\,2B {\rm{Re}}\Big(v_d^\mp \overline{ (\Phi^\mp_j \Omega^\mp_j) }\Big)   w^\pm(z_j)
\nonumber
\\[2mm]  & = {L^0_\pm}  (\Phi_j)  +\Big[ A_\pm \big( |w^\pm(z_j)|^2-|v_d^\pm|^2 \big) + B\big( |w^\mp(z_j)|^2-|v_d^\mp|^2 \big) \Big]   \Phi^\pm_j
\nonumber
\\[2mm]  & \quad  +2A_\pm \big(1-|\Omega^\pm_j|^2 \big) {\rm{Re}}\Big(w^\pm(z_j)\overline{ \Phi^\pm_j  }\Big)  w^\pm(z_j)  + 2B\big(1-|\Omega^\mp_j|^2 \big) {\rm{Re}}\Big(w^\mp(z_j) \overline{ \Phi^\mp_j }\Big)   w^\pm(z_j).
 \end{align*}
 Since
 \begin{align*}
 \mathbb{E}^\pm  &= S^\pm(v_d)
 \nonumber
\\[2mm]  & = \Delta  (w^\pm(z_j) \Omega^\pm_j)  +  \Big[ A_\pm \big(1-|w^\pm(z_j)|^2 |\Omega^\pm_j|^2 \big)+B\big(1-|w^\mp(z_j)|^2  |\Omega^\mp_j|^2\big)\Big] w^\pm(z_j) \Omega^\pm_j
\nonumber
\\[2mm]  & \quad  +
\e^2\Big[\p^2_{ss} (w^\pm(z_j) \Omega^\pm_j) -4i\p_s (w^\pm(z_j) \Omega^\pm_j) -4 w^\pm(z_j) \Omega^\pm_j \Big]
\nonumber
\\[2mm]  & = w^\pm(z_j)   \Big[ \Delta \Omega^\pm_j+\e^2\p^2_{ss}  \Omega^\pm_j  -4i \e^2\p_s \Omega^\pm_j \Big] -4\e^2 w^\pm(z_j) \Omega^\pm_j
\nonumber
\\[2mm]  & \quad   +   \Big[ A_\pm |w^\pm(z_j)|^2  \big(1-|\Omega^\pm_j |^2   \big)+B |w^\mp(z_j)|^2  \big(1-|\Omega^\mp_j |^2   \big)\Big]  w^\pm(z_j) \Omega^\pm_j
\nonumber
\\[2mm]  & \quad  +   2\nabla  w^\pm (z_j)\nabla \Omega^\pm_j   + 2\e^2 \p_s w^\pm(z_j) \p_s\Omega^\pm_j
  +   \e^2 \Big[ \p_{ss} w^\pm(z_j) - 4i \p_sw^\pm (z_j)\Big] \Omega^\pm_j,
 \end{align*}
 we therefore conclude that
 \begin{align*}
   & \Delta \Omega^\pm_j+\e^2\p^2_{ss}  \Omega^\pm_j  -4i \e^2\p_s \Omega^\pm_j
  \nonumber
\\[2mm]  &
   = \frac{ \mathbb{E}^\pm} {w^\pm(z_j)} + 4\e^2  \Omega^\pm_j  - \frac{ 2\nabla  w^\pm(z_j)\nabla \Omega^\pm_j   + 2\e^2 \p_s w^\pm(z_j) \p_s\Omega^\pm_j  } {w^\pm(z_j)}
             \nonumber
\\[2mm]  &  \quad
   -  \Big[ A_\pm |w^\pm(z_j)|^2  \big(1-|\Omega^\pm_j |^2   \big)+B |w^\mp(z_j)|^2  \big(1-|\Omega^\mp_j |^2   \big)\Big]   \Omega^\pm_j
   \nonumber
\\[2mm]  &  \quad   -\e^2 \frac{\Omega^\pm_j  \Big[ \p_{ss} w^\pm(z_j) - 4i \p_sw^\pm(z_j)\Big]  } {w^\pm(z_j)}.
 \end{align*}
  Then we can get
  \begin{align*}
  \A_3  &=  \frac1{\Omega_j^\pm} \Big[ \Delta \Omega^\pm_j+\e^2\p^2_{ss}  \Omega^\pm_j  -4i \e^2\p_s \Omega^\pm_j \Big] \Phi^\pm_j
     \nonumber
\\[2mm]  &  =   \Bigg\{ \frac{ \mathbb{E}^\pm} {v_d^\pm}  + 4 \e^2 -  \frac{2\nabla  w^\pm(z_j)\nabla \Omega^\pm_j   + 2\e^2 \p_s w^\pm(z_j) \p_s\Omega^\pm_j } {v_d^\pm}
     \nonumber
\\[2mm]  &  \qquad
-\e^2 \frac{\Big[ \p_{ss} w^\pm(z_j) - 4i \p_sw^\pm(z_j)\Big]  } {w^\pm(z_j)} \Bigg\}   \Phi_j^\pm
   \nonumber
\\[2mm]  &  \quad  -
 \Big[ A_\pm |w^\pm(z_j)|^2  \big(1-|\Omega^\pm_j |^2   \big)+B |w^\mp(z_j)|^2  \big(1-|\Omega^\mp_j |^2   \big)\Big]  \Phi^\pm_j \ .
  \end{align*}
  Combining  all the above calculations, we obtain
   \begin{align}  \label{Lpmdj}
   L^\pm_{d,j}(\Phi_j)
  & =   \A_1+ \A_2+ \A_3 +  (\eta_1-1)\frac{\mathbb{E}^\pm} {v_d^\pm}\Phi_j^\pm
  \nonumber
\\[2mm]  & = {L^0_\pm}  (\Phi_j)
   +2A_\pm \big(1-|\Omega^\pm_j|^2 \big) {\rm{Re}}\Big(w^\pm(z_j)\overline{ \Phi^\pm_j  }\Big)  w^\pm(z_j)  + 2B\big(1-|\Omega^\mp_j|^2 \big) {\rm{Re}}\Big(w^\mp(z_j) \overline{ \Phi^\mp_j }\Big)   w^\pm(z_j)
     \nonumber
\\[2mm]  & \quad+
   \e^2 \Big[  \p^2_{ss}  \Phi^\pm_j  -4i   \p_s \Phi^\pm_j  -4 \Phi^\pm_j  \Big]+\frac1{\Omega_j^\pm} \Big[ 2 \nabla \Phi^\pm_j \nabla  \Omega^\pm_j +2 \e^2 \p_s\Omega^\pm_j\p_s \Phi^\pm_j \Big]
      \nonumber
\\[2mm]  &  \quad +   \Bigg\{  4 \e^2 -  \frac{2\nabla  w^\pm(z_j)\nabla \Omega^\pm_j   + 2\e^2 \p_s w^\pm(z_j) \p_s\Omega^\pm_j } {v_d^\pm}
- \e^2\frac{\Big[ \p_{ss} w^\pm(z_j) - 4i \p_sw^\pm(z_j)\Big]  } {w^\pm(z_j)} \Bigg\}   \Phi_j^\pm
      \nonumber
\\[2mm]  &  \quad  +  \eta_1\frac{\mathbb{E}^\pm} {v_d^\pm}\Phi_j^\pm \ .
  \end{align}

We  point out that for $|{z}_j|<3,$ from Lemma \eqref{lemmaofW}, there holds
\begin{equation*}
 |\Omega^\pm_j({z}_j)|=1+O_\e(\e^2|\ln\e|), \ \ \ \nabla \Omega^\pm_j({z}_j)=O_ \e(\e\sqrt{|\ln\e|}), \ \ \  |\Delta \Omega^\pm_j({z}_j)|=O_\e(\e^2|\ln\e|).
\end{equation*}
Therefore,  the linear operator $  L^\pm_{d,j}$ is a small perturbation of $L^0_\pm$ when near the vortices.

 \end{proof}

  \subsection{Symmetry assumptions on the perturbation}
 At the last of this part, we point out that  we would make some  symmetry  assumptions  for the perturbation  in our problem.
In fact, from the expression of $v_d= (v_d^+, v_d^-) $ in \eqref{eq:Approx2-dimensionalrescaled},  we  know that, under the coordinates $z=x_1+ ix_2$,
 there hold
 \begin{align*}
 v_d^\pm (-x_1, x_2)  = \overline {v_d}^\pm (x_1, x_2), \quad  v_d^\pm (x_1, -x_2)  = \overline{v_d}^\pm (x_1, x_2).
 \end{align*}
 Note that these symmetries   are  compatible with the solution  operator   $S(v)$:
 \[ \mbox{if } S^\pm  (v)=0 \ \mbox{ and }\ u(z) = (u^+, u^-) (z) := \big({\overline v^+}, {\overline v^-}\big) (x_1, -x_2),\quad \mbox{then } S(u)=0,  \]
 and the same property holds for $u(z) = (u^+, u^-) (z) := \Big( v^+ (-x_1, x_2),  v^- (-x_1, x_2)  \Big).$
   Then we naturally assume that  the perturbation $\psi  $ satisfies
  \begin{align}  \label{symetry1}
  \psi (x_1, -x_2)  = - \overline{\psi^\pm} (x_1, x_2), \quad   \psi (-x_1, x_2 )  = - \overline{\psi^\pm} (x_1, x_2).
  \end{align}

%
\bigskip

\section{Error estimate  and Fourier decomposition}   \label{section Error estimate of Approximation}
\subsection{Error estimate}
Recall the approximation $v_d$ of $S(v)=0$ in  \eqref{eq:Approx2-dimensionalrescaled}-\eqref{eq:Approx2-dimensionalrescaledpm}.
Suppose
\[
z-e_j = \ell_j e^{i\theta_j} \quad \text{for}~j=1, 2,
\]
where  the $e_j$'s are given in \eqref{ej}.
  We can rewrite  $ v_d $ in the new coordinates $(\ell_j, \theta_j)$ as
 \begin{align}  \label{vdinelltheta}
  v_d   = (v_d^+, v_d^-), \quad  v_d^\pm(z) =  ({t^\pm})^{-1} W^\pm(\ell_1) W^\pm(\ell_2) e^{i(\theta_1+\theta_2)}.
 \end{align}
 In this section, we would  give an accurate    estimation  of $S(v_d)$.   Let
 \[
 S(v_d)  = \big(S^+(v_d), S^-(v_d)\big) := i (v_d^+ \rR^+, v_d^-\rR^-),
 \]
 and
 \[
 S^\pm(v_d) =  S^\pm_0(v_d) +   S^\pm_1(v_d) : = i  v_d^\pm \rR_0^\pm + i  v_d^\pm \rR_1^\pm,
 \]
 with
 \[ S^\pm_0(v_d) =  i  v_d^\pm \rR_0^\pm.
 \]

\medskip
 To  estimate the size of  error  $\rR=(\rR^+, \rR^-)$, we introduce the norm $\|\cdot\|_{**}$.
 For given $0<\alpha,\sigma<1$, $\rH=(\rH^+, \rH^-)$,  we  define the norm
 \begin{align}
  \|\rH\|_{**} = \|\rH^+\|_{**,1} \ +\  \|\rH^-\|_{**,1},  \label{def:norm_**0}
 \end{align}
with
 \begin{align}
\|\rH^\pm\|_{**,1}
&:=\sum_{j=1}^2 \| v_d^\pm \rH^\pm\|_{C^\alpha(\ell_j<3)}
\ +\ \sup_{\ell_1>2,\ell_2>2}
\Bigg[
\frac{| {\rm{Re}}(\rH^\pm)|}{\ell_1^{-2}+\ell_2^{-2}+ \e^{2} }
+ \frac{ |{\rm{Im}}(\rH^\pm)|}{\ell_1^{-2+\sigma}+\ell_2^{-2+\sigma}+ \e^{2-\sigma}}
\Bigg] \nonumber
\\[2mm]
& \quad
\ +\  \sup_{2<\ell_1<2\mathbb{R}_\e , \,2<\ell_2<2\mathbb{R}_\e}
\frac{ [ {\rm{Re}}(\rH^\pm)]_{\alpha,B_{|z|/2}(z)}}{ \ell_1^{-2-\alpha} +  \ell_2^{-2-\alpha} }
\ +\  \sup_{2<\ell_1<2\mathbb{R}_\e , \, 2<\ell_2<2\mathbb{R}_\e}
\frac{ [  {\rm{Im}}(\rH^\pm)]_{\alpha,B_{1}(z)}}{\ell_1^{-2+\sigma} +  \ell_2^{-2+\sigma}} , \label{def:norm_**}
\end{align}
where $\mathbb{R}_\e=\frac{\alpha_0}{\e |\ln\e|^\frac12}$ for some $\alpha_0$ not large such that $\R_\e \le \frac 12  {\tilde d}$, and $\| \cdot\|_{C^\alpha(D)}, [\cdot ]_{\alpha,D}, \|\cdot\|_{C^{k,\alpha}(D)}$ are defined as
\begin{align*}
\|f\|_{C^\alpha(D)} &= \|f\|_{C^{0,\alpha}(D)},
\qquad
[f]_{\alpha,D}
 = \sup_{x,y\in D , \, x\not=y}\frac{|f(x)-f(y)|}{|x-y|^\alpha},
\\[2mm]
\|f\|_{C^{k,\alpha}(D)}
&=
\sum_{j=0}^k \|D^jf\|_{L^\infty	(D)}
+ [D^k f]_{\alpha,D},\quad \text{for}~f: \R^2\rightarrow \R.
\end{align*}


\begin{proposition}\label{prop:sizeoferror}
Recall that
\[
\rR= (\rR^+, \rR^-)\quad \text{with}\ \ \rR^\pm=  - i \frac{S^\pm(v_d)} {v_d^\pm}.
\]
The following estimate for $\rR$ holds
\begin{equation}\label{R-star-norm}
\| \rR\|_{**}\leq \frac{C}{|\ln \e|}.
\end{equation}
\end{proposition}

\begin{proof}
Without loss of generality, we assume $x_1>0$.
  For  $S^\pm(v_d)$, we  recall that
\begin{align*}
S^\pm(v_d) = S^\pm_0(v_d) + S^\pm_1(v_d),
\end{align*}
with\begin{align*}
S_0^\pm(v_d)   = \Delta v_d^\pm
+\Big[ A_\pm \big({t^+}^2-|v_d^\pm|^2\big)+B\big({t^-}^2-|v_d^\mp|^2\big)\Big]v_d^\pm,
\end{align*}
and
\begin{align*}
  S_1^\pm(v_d)    &=  \e^2(\p^2_{ss}v_d^\pm -4i\p_sv_d^\pm -4v_d^\pm).
  \end{align*}
 Denote
 \[
w^\pm_a (z):=w^\pm(z-\tilde d)    , \quad  w^\pm_b (z):=w^\pm(z+\tilde d).
 \]
 Since $x_1>0$,  it follows from Lemma \ref{lemmaofW} that
 \[
 ({t^\pm})^{-1}|w^\pm_b (z)| = 1+O_\e(\e^2|\ln\e|).
 \]

\medskip
Substituting \eqref{vdinelltheta} into  the expressions of $S_0^\pm(v_d) $, we get
 \begin{align}  \label{S0vdpm}
 S_0^\pm(v_d)
   & = ({t^\pm})^{-1} \Bigg\{\, \Delta (w^\pm_a w^\pm_b)  + \Big[ A_\pm \big({t^\pm}^2- ({t^\pm})^{-2}|w^\pm_a|^2| w^\pm_b|^2\big)+B\big({t^\mp}^2-({t^\mp})^{-2}|w^\mp_a|^2| w^\mp_b|^2\big)\Big]w^\pm_a w^\pm_b\Bigg\}
  \nonumber
\\[2mm]
& =   ({t^\pm})^{-1} \Bigg\{\Big( w^\pm_b \Delta  w^\pm_a  + w^\pm_a \Delta  w^\pm_b+ 2 \nabla  w^\pm_a   \nabla w^\pm_b\Big)
  \nonumber
\\[2mm]
& \qquad\qquad
+  \Big[ A_\pm \big({t^\pm}^2- ({t^\pm})^{-2}|w^\pm_a|^2| w^\pm_b|^2\big)+B\big({t^\mp}^2-({t^\mp})^{-2}|w^\mp_a|^2| w^\mp_b|^2\big)\Big] w^\pm_a w^\pm_b
\Bigg\}
  \nonumber
\\[2mm]
& =  -  ({t^\pm})^{-1}  \Big[ A_\pm \big({t^\pm}^2-|w^\pm_a|^2\big)+B\big({t^\mp}^2-|w^\mp_a|^2\big)\Big] w^\pm_a w^\pm_b
  \nonumber
\\[2mm]
&\quad
 -  ({t^\pm})^{-1} \Big[ A_\pm \big({t^\pm}^2-|w^\pm_b|^2\big)+B\big({t^\mp}^2-|w^\mp_b|^2\big)\Big]w^\pm_a w^\pm_b
  \nonumber
\\[2mm]
&\quad
+   ({t^\pm})^{-1} \Big[ A_\pm \big({t^\pm}^2- ({t^\pm})^{-2}|w^\pm_a|^2| w^\pm_b|^2\big)+B\big({t^\mp}^2-({t^\mp})^{-2}|w^\mp_a|^2| w^\mp_b|^2\big)\Big] w^\pm_a w^\pm_b
  \nonumber
\\[2mm]
&\quad
+  ({t^\pm})^{-1} \Big({W^\pm}'(\ell_1) e^{i\theta_1}  \nabla{\ell_1} +  i {W^\pm}(\ell_1) e^{i\theta_1}  \nabla{\theta_1} \Big)\cdot\Big({W^\pm}'(\ell_2) e^{i\theta_2}  \nabla{\ell_2} +  i {W^\pm}(\ell_2) e^{i\theta_2}  \nabla{\theta_2} \Big)
  \nonumber
\\[2mm]
&   =   ({t^\pm})^{-1}  \Big[ A_\pm   \big(|w^\pm_a|^2 - {t^\pm}^2\big) \big(1-({t^\pm})^{-2}|w^\pm_b|^2\big)  + B  \big(|w^\mp_a|^2 -  {t^\mp}^2\big) \big(1-({t^\mp})^{-2}|w^\mp_b|^2\big) \Big]w^\pm_a w^\pm_b
\nonumber
\\[2mm]
&\quad    +   ({t^\pm})^{-1}   \Bigg[
{W^\pm}'(\ell_1)  {W^\pm}'(\ell_2) \big(\cos{\theta_1}\cos{\theta_2}+\sin{\theta_1}\sin{\theta_2}\big)
\nonumber
\\[2mm]
&\qquad\qquad\quad
- {W^\pm}(\ell_1) W^\pm(\ell_2) \frac{\cos{\theta_1}\cos{\theta_1}-\sin{\theta_1}\sin{\theta_2} } {\ell_1\ell_2}
\nonumber
\\[2mm]
&\qquad\qquad\quad + i {W^\pm}'(\ell_1) W^\pm(\ell_2) \frac{\cos{\theta_2}\sin{\theta_1}-\cos{\theta_1}\sin{\theta_2} } {\ell_2}
 \nonumber
\\[2mm]
&\qquad\qquad\quad +i{W^\pm}'(\ell_2) W^\pm(\ell_1) \frac{\cos{\theta_1}\sin{\theta_2}-\cos{\theta_2}\sin{\theta_1} } {\ell_1}  \Bigg] e^{i(\theta_1+\theta_2)}.
 \end{align}

  On the other hand, we consider the expression of $ S_1^\pm(v_d)  $. We recall that
  \begin{equation*}\nonumber
\ell_1= \sqrt{(r\cos s -\tilde{d})^2+r^2\sin^2s}, \ \ \ \ \ \ell_2= \sqrt{(r\cos s+\tilde{d})^2+r^2\sin^2s},
\end{equation*}
\begin{equation*}
e^{i\theta_1}= \frac{(r\cos s-\tilde{d})+ir\sin s}{\ell_1}, \ \ \ \ e^{i\theta_2}= \frac{(r\cos s+\tilde{d})+ir\sin s}{\ell_2}. \nonumber
\end{equation*}
Similar to the analysis   in the proof of Proposition $4.1$ in   \cite{DDMR}, we give  the  expressions of $ S_1^\pm(v_d)  $ in   $(\ell_j, \theta_j)$ as the following
 \begin{align}  \label{S1vd}
  & S_1^\pm(v_d)    \nonumber
  \\[2mm]
   &= \,\Bigg\{ \frac{\hat{d}^2}{|\ln \e|} \Bigg[ {W^\pm}'' (\ell_1) {W^\pm}(\ell_2) \sin^2{\theta_1}  +  +{W^\pm}'' (\ell_2) {W^\pm}(\ell_1) \sin^2 {\theta_2} -{W^\pm}' (\ell_1) {W^\pm}' (\ell_2)\sin {\theta_1} \sin {\theta_2}
   \nonumber
 \\[2mm]
   &\qquad\qquad  + \frac{\cos^2{\theta_1}}{\ell_1} {W^\pm}' (\ell_1) {W^\pm}(\ell_2) +\frac{\cos^2 {\theta_2}}{\ell_2}{W^\pm}' (\ell_2) {W^\pm}(\ell_1)-
\Big(\frac{\cos \theta_1}{\ell_1}-\frac{\cos \theta_2}{\ell_2} \Big)^2 {W^\pm}(\ell_1) {W^\pm}(\ell_2) \Bigg]  \nonumber
   \\[2mm]
   &\qquad
  + \frac{\e \hat{d}}{\sqrt{|\ln \e|}}\Bigl(\cos {\theta_1} {W^\pm}' (\ell_1)  {W^\pm} (\ell_2) -\cos {\theta_2} {W^\pm}' (\ell_2) {W^\pm}(\ell_1) \Bigr)
  \Bigg\} (t^{\pm})^{-1}e^{i({\theta_1}+{\theta_2})}
  \nonumber
  \\[2mm]
   &\quad   +i\Bigg[ \frac{\e \hat{d}} {\sqrt{|\ln \e|}}\Big( \frac{\sin {\theta_2}}{\ell_2} - \frac{\sin {\theta_1}}{\ell_1} \Big)
   {W^\pm}(\ell_1) {W^\pm}(\ell_2) -\frac{2\hat{d}^2}{|\ln \e|}\Big(\frac{\sin {\theta_1} \cos {\theta_1}}{\ell_1^2}+\frac{\sin {\theta_2}\cos {\theta_2}}{\ell_2^2} \Big){W^\pm}(\ell_1){W^\pm}(\ell_2)  \nonumber
     \\[2mm]
   &\qquad\quad   +\frac{2\hat{d}^2}{|\ln \e|}\Big(\frac{\cos {\theta_1}}{\ell_1}-\frac{\cos {\theta_2}}{\ell_2} \Big)\Big( \sin {\theta_1} {W^\pm}'(\ell_1) {W^\pm}(\ell_2)-\sin {\theta_2} {W^\pm}(\ell_2)' {W^\pm}(\ell_1) \Big)  \Bigg] (t^{\pm})^{-1}e^{i({\theta_1}+{\theta_2})}.
   \end{align}

Proceeding as      in the proof of Proposition $4.1$ in  \cite{DDMR}, we can show that \eqref{R-star-norm} holds.
\end{proof}

\medskip
From the expression of $ S_1^\pm(v_d)$ in  \eqref{S1vd}, we deduce the following:
\begin{lemma}  \label{lemma5.2}
    In the region $B_{\tilde{d}} (e_1)$, we have the decompositions
    \begin{align}     \label{decomposition-s1}
   S_1^\pm(v_d)=  \frac{\hat{d}^2}{|\ln \e|}w^\pm_{a,x_2 x_2} w_b^\pm +\frac{\hat{d}\e}{\sqrt{|\ln \e|}}w^\pm_{a,x_1} w_b^\pm + \G^\pm,
   \end{align}
where $\G^\pm$ are the terms given in \eqref{gammapm}.   Moreover, there hold
   \begin{equation}
\label{eq:integralestimate}
{\rm{Re} }\int_{B_{\tilde{d}} (e_1)} w^\pm_{a,x_2 x_2}  \overline{w}^\pm_{a,x_1}  =0\quad \mbox{ and }\quad
{\rm{Re}}   \int_{B_{\tilde{d}} (e_1)} \frac{\G^\pm}{w^\pm_b} \overline{w}^\pm_{a,x_1}=O_\e\left(\frac{\e}{\sqrt{|\ln \e|}}\right).
 \end{equation}
   \end{lemma}

    \begin{proof}
We can show easily that the first equality of \eqref{eq:integralestimate} holds by orthogonality.  From \eqref{S1vd}, we have

\begin{align*}
S_1^\pm(v_d)
   &= \, \frac{\hat{d}^2   {W^\pm}(\ell_2) e^{i({\theta_1}+{\theta_2})}}{|\ln \e|\, t^{\pm}}
   \Bigg[ {W^\pm}'' (\ell_1)\sin^2{\theta_1}   + \Big(\cos^2\theta_1 +2i \cos{\theta_1} \sin{\theta_1}\Big)     \Big(\frac{{W^\pm}' (\ell_1)} {\ell_1}-\frac{{W^\pm} (\ell_1)} {\ell_1^2}   \Big)     \Bigg]
   \nonumber
  \\[2mm]
   & \quad +   \frac{\hat{d} \e}{\sqrt{|\ln \e|\, t^{\pm}}} \Bigl[\cos {\theta_1} {W^\pm}' (\ell_1)  {W^\pm} (\ell_2)  - i {W^\pm} (\ell_1)  {W^\pm} (\ell_2) \frac{\sin {\theta_1}} {\ell_2}  \Bigr] + \G^\pm,
       \end{align*}
 where
 \begin{align}
 \G^\pm  &  = \,\Bigg\{ \frac{\hat{d}^2}{|\ln \e|} \Bigg[  {W^\pm}'' (\ell_2) {W^\pm}(\ell_1) \sin^2 {\theta_2} -{W^\pm}' (\ell_1) {W^\pm}' (\ell_2)\sin {\theta_1} \sin {\theta_2}  \nonumber
 \\[2mm]
   &\qquad\qquad\qquad  +\frac{\cos^2 {\theta_2}}{\ell_2}{W^\pm}' (\ell_2) {W^\pm}(\ell_1)-
\left(\frac{-2\cos \theta_1 \cos \theta_2}{\ell_1\ell_2}+\frac{\cos^2 \theta_2}{\ell_2^2} \right) {W^\pm}(\ell_1) {W^\pm}(\ell_2) \Bigg]  \nonumber
   \\[2mm]
   &\qquad
  + \frac{\e \hat{d}}{\sqrt{|\ln \e|}}\Bigl( -\cos {\theta_2} {W^\pm}' (\ell_2) {W^\pm}(\ell_1) \Bigr)
  \Bigg\} (t^\pm)^{-1}e^{i({\theta_1}+{\theta_2})}
  \nonumber
  \\[2mm]
   &\quad   +i\Bigg[ \frac{\e \hat{d}} {\sqrt{|\ln \e|}} \frac{\sin {\theta_2}}{\ell_2}
   {W^\pm}(\ell_1) {W^\pm}(\ell_2) -\frac{2\hat{d}^2}{|\ln \e|}\left( \frac{\sin {\theta_2}\cos {\theta_2}}{\ell_2^2} \right){W^\pm}(\ell_1){W^\pm}(\ell_2)  \nonumber
     \\[2mm]
   &\qquad   -\frac{2\hat{d}^2}{|\ln \e|}  \frac{\cos {\theta_2}}{\ell_2}  \Big( \sin {\theta_1} {W^\pm}'(\ell_1) {W^\pm}(\ell_2)-\sin {\theta_2} {W^\pm}'(\ell_2){W^\pm}(\ell_1) \Big)
    \nonumber
     \\[2mm]
   &\qquad
    -\frac{2\hat{d}^2}{|\ln \e|} \frac{\cos {\theta_1}}{\ell_1} \sin {\theta_2} {W^\pm}'(\ell_2) {W^\pm}(\ell_1)
    \Bigg] (t^\pm)^{-1}e^{i({\theta_1}+{\theta_2})}.
    \label{gammapm}
   \end{align}
From the decompositions above, we know that \eqref{decomposition-s1} and \eqref{eq:integralestimate} hold.
\end{proof}

In Lemma \ref{lemma5.2},  we decompose $S^{+}_1(v_d)$ into two parts: one is $\frac{\hat{d}^2}{|\ln \e|}w^+_{a,x_2 x_2} w^{+}_b$ which is in big size but has a symmetry that makes it orthogonal to the kernel ${w}^+_{a,x_1}$ locally, the other one is  $\frac{\hat{d}\e}{\sqrt{|\ln \e|}}w^+_{a,x_1} w^{+}_b + \G^+ $ which has smaller size.
Similar decomposition also holds for $S^{-}_1(v_d)$.
Lemma \ref{lemma5.2}   will be used in the reduction process in Section \ref{section77}.

\subsection{Fourier  series  for the error}
\label{section3.2}
Next, we shall  make a more precise decomposition for the error $\rR$.
Recall the relations
\[
z-e_j  = \ell_j e^{i\theta_j} \quad \text{for}~j=1, 2,
\]
where $e_j$'s are given in \eqref{ej}.
For given  $\rH =(\rH^+, \rH^-)\in \mathbb{C}^2$ satisfying
$$
\rH(\overline{z})= - \overline{\rH}(z),\quad
\text{i.e,}   \quad \rH^\pm(\overline{z})= - \overline{\rH^\pm}(z),
$$
we have the following local decomposition for $\rH^\pm$  when near $e_j$,  $j=1,2$,
   \begin{align}
& \rH^\pm (z)  = \sum_{k=0}^\infty \rH^{\pm}_{k,j} (\ell_j, \theta_j),  \label{def:h_Fourier}
\end{align}
and
\begin{align}
\rH^{\pm}_{k,j}(\ell_j,\theta_j) = \rH_{k,j}^{1, \pm}(\ell_j) \sin (k\theta_j)
+ i \rH_{k,j}^{2, \pm}(\ell_j) \cos(k\theta_j), \quad
\rH_{k,j}^{1, \pm}(\ell_j), \rH_{k,j}^{2, \pm}(\ell_j) \in \R.
\label{HKJ}
\end{align}
Then define
\begin{align}\label{Hevenpm}
\rH^{\pm}_{e,j} = \sum_{k ~\text{is even}} \rH^{\pm}_{k,j},
\qquad
\rH^{\pm}_{o,j} = \sum_{k ~\text{is odd}} \rH^{\pm}_{k,j},
\end{align}
 \begin{align}\label{Heven}
 \rH_{e,j} =( \rH^{+}_{e,j}, \rH^{-}_{e,j}),  \qquad  \rH_{o,j} =( \rH^{+}_{o,j}, \rH^{-}_{o,j}).
 \end{align}
Let $  \mathscr{R}_j z$ denote  the reflection  across the real line ${{\rm{Re}}(z)}=(-1)^{j+1} \tilde{d}$, then
  \begin{align} \label{mathscrR}
  \mathscr{R}_j z= 2e_j -{{\rm{Re}} (z)} + i {{\rm{Im}}(z)}, \quad\text{or}\quad \mathscr{R}_j z = \ell_j e^{i(\pi -\theta_j)}+e_j .
  \end{align}
 Combining  \eqref{mathscrR}  and \eqref{HKJ}-\eqref{Heven}, we get
\begin{align*}
& \rH^{\pm}_{e,j}\big(\mathscr{R}_j z \big) =  - \overline{\rH^{\pm}_{e,j}}(z),
 \qquad
 \rH^{\pm}_{o,j}\big(\mathscr{R}_j z \big) = \overline{\rH^{\pm}_{o,j}}(z),
\\[3mm]
&   \rH_{e,j}\big(\mathscr{R}_j z \big) =  - \overline{\rH_{e,j}}(z),
\qquad
\rH_{o,j}\big(\mathscr{R}_j z \big) =   \overline{\rH_{o,j}}(z),
\end{align*}
 and thus we can
  define equivalently
\begin{align*}
& \rH_{o,j}^\pm(z) = \frac{1}{2}\Big[ \rH^\pm(z) + \overline{\rH^\pm(\mathscr{R}_j z )}\Big] ,
\qquad
\rH_{e,j}^\pm(z) = \frac{1}{2}\Big[ \rH^\pm(z) - \overline{\rH^\pm(\mathscr{R}_j z )}\Big],
\\[3mm]
 & \rH_{o,j}(z) = \frac{1}{2}\Big[ \rH(z) + \overline{\rH(\mathscr{R}_j z )}\Big] ,
 \qquad
\quad\, \, \, \rH_{e,j}(z) = \frac{1}{2}\Big[ \rH(z) - \overline{\rH(\mathscr{R}_j z )}\Big].
\end{align*}
Note that, in the coordinates $(\ell_j, \theta_j)$,  the functions $\rH_{o,j}$  introduced  in  \eqref{HKJ}-\eqref{Heven}   are  the odd parts of  $\rH$ in local domains $B_{\tilde{d}} (e_j), j=1,2$.
Furthermore,
we  would define  functions $ \rH^\pm_o, \rH_{o}$    which stand for the odd part of $\rH^\pm, \rH$   globally.
Define  the cut-off functions $\eta_{j, R} $ for $j=1, 2$  as the following
\begin{align}
\eta_{j, R} (z) = \eta_j\Big(\frac{|z-e_j |} {2}\Big),
\label{cutoff}
\end{align}
   where $\eta_1$ is the cut-off functions defined in \eqref{eta1} and $R$ is a constant to be chosen later.
For any given $\rH = (\rH^+, \rH^-)\in \mathbb{C}^2$, we define
   \begin{align}  \label{def-ho}
   \rH_o=\big(   \rH_o^+,    \rH_o^- \big), \quad\text{with}\quad    \rH_o^\pm = \eta_{1, \mathbb{R}_\e}  \rH_{o,1}^\pm+   \eta_{2, \mathbb{R}_\e}  \rH_{o,2}^\pm,  \quad
   \end{align}
where $\mathbb{R}_\e= \frac{c_0} {\e|\ln\e|^{\frac 12} } $ for some $c_0$ not large such that $\mathbb{R}_\e \le \frac 12 \tilde{d}.$
On the other hand,    define the global functions
 \[ \rH_e= \rH-\rH_o, \quad\text{with}\quad \rH_e^\pm= \rH^\pm-\rH_o^\pm.   \]

 We note that, using the decompositions introduced above, we could decompose   the error function
$$
\rR=(\rR^+, \rR^-)=\Bigg(\frac{S^+(v_d)}{iv_d^+}, \frac{S^-(v_d)}{iv_d^-}\Bigg)
$$
in odd Fourier modes and even Fourier modes, see Proposition \ref{errorProp2}.  The odd Fourier modes of the error have much smaller sizes than even ones. However, there exists one difficulty: some terms of  odd Fourier modes have a good size but   decay  slowly.
Using the idea in \cite{DDMR}, we introduce  the  semi-norm $|\cdot|_{\sharp\sharp} $ for these slow decay terms.
For any $\rH =(\rH^+, \rH^-) =(\rH^+_1+i\rH^+_2 , \rH^-_1+i\rH^-_2)\in \mathbb{C}^2 $,   the  semi-norm $|\rH|_{\sharp\sharp} $ is defined as
   \begin{align} \label{sharpsharpnorm}
   |\rH|_{\sharp\sharp} = |\rH^+|_{\sharp\sharp_1} + |\rH^-|_{\sharp\sharp_1},
   \end{align}
where
  \begin{align}
|\rH^\pm|_{\sharp\sharp_1}
= \sum_{j=1}^2 \|  v_d^\pm \rH^\pm  \|_{C^{0,\alpha}(\ell_j<4)}
\ +\  \sup_{2<\ell_1< \mathbb{R}_\e, \, 2<\ell_2 <\mathbb{R}_\e}
\Biggl[
\frac{| \rH^\pm_1|}{\ell_1^{-1} + \ell_2^{-1}}
+\frac{ | \rH^\pm_2 | }{\ell_1^{-1+\sigma} + \ell_2^{-1+\sigma}}
\Biggr].
\label{def:semi_norm_sharpsharp}
\end{align}

Now we  give a precise decomposition for the error $\rR$.
\begin{proposition}\label{errorProp2}
Recall that $v_d$ is given by \eqref{eq:Approx2-dimensionalrescaled} and
$$
\mathbb{E}=S(v_d)=\big(iv_d^+\rR^+,iv_d^-\rR^-\big),
\qquad
\rR=(\rR^+, \rR^-).
$$
Then we can write
$$
\rR=\rR_o+\rR_e= \big(\rR_o^+, \rR_o^-\big)+\big(\rR_e^+, \rR_e^-\big),
\qquad
\rR_o=\rR_o^\alpha + \rR_o^\beta= \big(\rR_o^{+, \alpha}, \rR_o^{-, \alpha}\big)+\big(\rR_o^{+, \beta}, \rR_o^{-, \beta}\big),
$$
i.e,
$$
\rR^\pm = \rR^\pm_o+  \rR^\pm_e, \qquad \rR^\pm_o = \rR_o^{\pm, \alpha} +\rR_o^{\pm, \beta},
$$
with $\rR_o, \rR_o^\pm$ defined analogously to \eqref{def-ho}.
Moreover, the symmetry
\(\rR_o(\mathscr{R}_jz) =\overline{\rR_o (z)} \) in \(B_{\mathbb{R}_\e}(e_1)\cup B_{\mathbb{R}_\e}(e_2)\) is true and the following estimates are valid
\begin{align*}
\|\rR_e\|_{**}+\|\rR_o\|_{**}\leq \frac{C}{|\ln \e|},
\qquad
|\rR_o^\alpha|_{\sharp\sharp} & \leq C \frac{\e}{\sqrt{|\ln \e|}},
\qquad
\| \rR_o^\beta\|_{**} \leq C\e\sqrt{|\ln \e|}.
\end{align*}
\end{proposition}

 \begin{proof}
  Following Proposition \ref{prop:sizeoferror}, we immediately get that
  \begin{align*}
  \|\rR_e\|_{**}+\|\rR_o\|_{**}\leq \frac{C}{|\ln \e|}.
  \end{align*}
  For every $j\in \{1,2\}$, we define
 \begin{align*}
r_{o,1}^\pm  &:=
-i\frac{{\tilde d}^2}{|\ln \e| (t^\pm)}\Bigg( \frac{2\cos \theta_1 \cos \theta_2}{\ell_1 \ell_2}+\frac{\cos^2\theta_2}{\ell_2^2} \Bigg)
+\Bigg(\frac{{\tilde d} \e}{\sqrt{|\ln \e|} (t^\pm)} \frac{-\sin \theta_1}{\ell_1}-\frac{2{\tilde d}^2}{|\ln \e| (t^\pm)}\frac{\sin \theta_2 \cos \theta_2}{\ell_2^2} \Bigg),
 \\[4mm]
r_{o,2}^\pm &:=
-i\frac{{\tilde d}^2}{|\ln \e| (t^\pm)}\Bigg( \frac{2\cos \theta_1 \cos \theta_2}{\ell_1 \ell_2}+\frac{\cos^2\theta_1}{\ell_1^2} \Bigg)
+\Bigg(\frac{{\tilde d} \e}{\sqrt{|\ln \e|}(t^\pm)} \frac{+\sin \theta_2}{\ell_2}-\frac{2{\tilde d}^2}{|\ln \e| (t^\pm)}\frac{\sin \theta_1 \cos \theta_1}{\ell_1^2} \Bigg),
\end{align*}
\begin{align} \label{rjo}
R_{o, j}^{\pm, \alpha}:= \frac{1}{2}\Big[r_{o,j}^\pm (z)+\overline{r_{o,j}^\pm \left(\mathscr{R}_jz \right)}\Big],
\end{align}
and
\[
R_o^{\alpha}: = (R_o^{+, \alpha}, R_o^{-, \alpha}), \quad \text{with}~
R_o^{\pm, \alpha}:=\eta_{1,\R_\e}  R_{o, 1}^{\pm, \alpha}+\eta_{2,\R_\e} R_{o, 2}^{\pm, \alpha}.
\]
Note that the terms in $r_{o,j}^\pm$ are all slow decay terms.

\medskip
We  claim  that  the  components of functions
$ R_{o, j}^{\pm, \alpha}$   are all  odd mode terms when they lie  locally in domain $B_{\mathbb{R}_\e} (e_j)$.
 In fact, for example,   taking $j=1$,   when  ${|z-e_1 |} \le \R_\e$,  by  rewriting   $r_{o,1}^\pm$ in Fourier series in  $\theta_1$, there hold
\begin{align*}
\frac{2\cos \theta_1 \cos \theta_2}{\ell_1 \ell_2} &=  \frac{2\cos \theta_1 (x_1 + \tilde d) }{\ell_1 \ell_2^2} = \frac{2\cos \theta_1 \cos \theta_1   }{\ell_2^2} + \frac{4\cos \theta_1   \tilde d }{\ell_1 \ell_2^2} \nonumber
\\[1mm] &= \frac1 {\ell_1^2+{ \tilde d}^2 + 4 \tilde d \ell_1 \cos \theta_1} \Big(2\cos^2 \theta_1+ \frac{4\cos \theta_1   \tilde d }{\ell_1} \Big)
\nonumber
\\[1mm] &
=  \sum_{m=0}^{\infty}  a_m (\ell_1) \cos^m {\theta_1},
\end{align*}
and
\begin{align*}
\frac{\cos^2\theta_2}{\ell_2^2} &= \frac {(x_1-\tilde d + 2 \tilde d)^2} {\ell_2^4}
\nonumber
\\[1mm] &
=
 \Big(\frac1 {\ell_1^2+{ \tilde d}^2 + 4 \tilde d \ell_1 \cos \theta_1} \Big)^2 ( \ell_1^2 \cos^2\theta_1 + 2 \tilde d \ell_1 \cos\theta_1+ 4 \tilde d^2)
 \nonumber
\\[1mm] &
=  \sum_{m=0}^{\infty}  b_m (\ell_1) \cos^m {\theta_1},
\end{align*}
\begin{align*}
 \frac{\sin \theta_2 \cos \theta_2}{\ell_2^2} &=\frac{(x_1- \tilde d + 2 \tilde d ) x_2}{\ell_4^2}  =\frac{(x_1- \tilde d ) x_2 + 2 \tilde d x_2}{\ell_2^4}
 \nonumber
\\[1mm] &
  =   \Big(\frac1 {\ell_1^2+{ \tilde d}^2 + 4 \tilde d \ell_1 \cos \theta_1} \Big)^2 ( \ell_1^2 \sin \theta_1 \cos\theta_1  + 2 \tilde d \ell_1  \sin \theta_1)
   \nonumber
\\[1mm] &
  =   \sin \theta_1 \sum_{m=0}^{\infty}  c_m (\rho_1) \cos^m {\theta_1},
 \end{align*}
 where  $a_m, b_m, c_m$ are bounded smooth functions.  Therefore, we  can rewrite $ r_{o,1}^\pm$ as
 \begin{align}  \label{mr01}
 r_{o,1}^\pm (z)  &=
-i\frac{{\tilde d}^2}{|\ln \e| (t^\pm)}\Bigg( \sum_{m=0}^{\infty}  \Big[a_m (\ell_1) \cos^m {\theta_1}+b_m (\ell_1) \cos^m {\theta_1}\Big] \Bigg)
\\[2mm]
 &\quad
+\Bigg(\frac{{\tilde d} \e}{\sqrt{|\ln \e|} (t^\pm)} \frac{-\sin \theta_1}{\ell_1}-\frac{2{\tilde d}^2}{|\ln \e| (t^\pm)}\sin \theta_1 \sum_{m=0}^{\infty}  c_m (\rho_1) \cos^m {\theta_1} \Bigg).
 \end{align}
  By the analysis above and recalling
 \[\mathscr{R}_1 z = \ell_1 e^{i(\pi -\theta_1)}+e_1,  \]  we can easily  show that
  \begin{align}  \label{mr02}
 \overline{r_{o,1}^\pm \left(\mathscr{R}_1z \right)}    &=
-i\frac{{\tilde d}^2}{|\ln \e| (t^\pm)}\Bigg( \sum_{m=0}^{\infty} (-1)^{m+1} \Big[a_m (\ell_1) \cos^m {\theta_1}+b_m (\ell_1) \cos^m {\theta_1}\Big] \Bigg) \nonumber
\\[2mm]
 &\quad
+\Bigg(\frac{{\tilde d} \e}{\sqrt{|\ln \e|} (t^\pm)} \frac{-\sin \theta_1}{\ell_1}-\frac{2{\tilde d}^2}{|\ln \e| (t^\pm)}\sin \theta_1 \sum_{m=0}^{\infty}(-1)^m  c_m (\rho_1) \cos^m {\theta_1} \Bigg).
 \end{align}
 Therefore, from \eqref{rjo},  \eqref{mr01}, and \eqref{mr02}, we know that the claim above is true.
Furthermore, we can  also check that \(R_o^\alpha\) and \(R_o^\beta:=R_o-R_o^\alpha\) satisfy the desired properties.
\end{proof}

\bigskip
 \section{The resolution  of  projected linear problem}\label{section44}
For convenience,
we first introduce some symmetries for the functions $\rH, \psi \in \mathbb{C}^2:$
 \begin{align} \label{s1}
 \rH(\overline{z}) = - \overline{\rH}(z),   \tag{S1}
 \end{align}
 \begin{align}
\label{s2}
\rH(\mathscr{R}_j z ) = - \overline{\rH(z)}  ,\quad
|z-e_j | < 2 \mathbb{R}_\e , \quad j=1,2,\tag{S2}
\end{align}
\begin{align}
\label{s3}
\rH(\mathscr{R}_j z ) =  \overline{\rH(z)}  ,\quad
|z-e_j | < 2 \mathbb{R}_\e , \quad j=1,2,\tag{S3}
\end{align}
  \begin{align} \label{s4}
  \psi(x_1,-x_2)=-\overline{\psi}(x_1,x_2), \ \ \ \ \ \psi(-x_1,x_2)=-\overline{\psi}(x_1,x_2). \tag{S4}
 \end{align}

 From the analysis in  Lemma \ref{Lepsilon}, we know that  the linear operator   $ L^\pm_{d,j}(\Phi) :=  iw^\pm(z- e_j )\mathcal L_\e^\pm (\psi)$    has one
 nontrivial kernel  for  $\psi$ which satisfies the symmetry \eqref{s4} in the vortex region.
 Therefore, in order to solve  the problem \eqref{result}, we should  consider   the  nonlinear projected problem
 for $\psi$ satisfies the symmetry \eqref{s4}

 \begin{align}
\label{eq:linear5.1}
\left\{
\begin{aligned}
& \LLL^\pm_\e(\psi)= \rR^\pm-\mathcal N^\pm(\psi) +  c \sum_{j=1}^2\frac{\chi_j}{iw^\pm(z-e_j )} (-1)^jw^\pm_{x_1}(z-e_j )
\quad\text{in }\R^2,
\\[2mm]
&{\rm{Re}} \int_{B(0,4)} \chi \Big[ \overline{\phi_j^+} w_{x_1}^++ \overline{\phi_j^-} w_{x_1}^- \Big]=0, \text{ with }\phi^\pm_j(z)=iw^\pm(z)\psi^\pm(z+e_j ), \quad j=1, 2,
\\[2mm]
&\psi \text{ satisfies the symmetry }  \eqref{s4},
\end{aligned}
\right.
\end{align}
where
\begin{equation}
\nonumber
\chi(z):=\eta_1 \Big(\frac{|z|}{2} \Big), \qquad \chi_j(z):=\eta_1 \Big(\frac{\ell_j}{2} \Big)=\eta_1 \Big(\frac{|z-e_j |}{2} \Big),
\end{equation}
and $\eta_1$ is a smooth cut-off function defined in \eqref{eta1}.
The resolution theory  of a linear projected  problem will be first
provided for solving \eqref{eq:linear5.1}.
Then an application of the Contraction Mapping Principle will give the existence of solutions.

\medskip
  \subsection{The linear resolution theory}
 The main objective of this part is to set up the resolution  theory of a linear projected  problem.
  For any $\rH$ satisfies \eqref{s4}, we  first  consider  the projected problem
  \begin{align}
\label{eq:linear}
\left\{
\begin{aligned}
& \LLL^\pm_\e(\psi)=\rH^\pm+  c \sum_{j=1}^2\frac{\chi_j}{iw^\pm(z-e_j )} (-1)^jw^\pm_{x_1}(z-e_j )
\quad\text{in }\R^2,
\\[2mm]
&{\rm{Re}} \int_{B(0,4)} \chi \Big[ \overline{\phi_j^+} w_{x_1}^++ \overline{\phi_j^-} w_{x_1}^- \Big]=0, \text{ with }\phi^\pm_j(z)=iw^\pm(z)\psi^\pm(z+e_j),\quad j=1, 2,
\\[2mm]
&\psi \text{ satisfies the symmetry }  \eqref{s4}.
\end{aligned}
\right.
\end{align}

%
   For solving \eqref{eq:linear}, at first we shall get a priori estimates  expressed in  suitable norms.
   Define for $\psi=(\psi^+, \psi^-):\mathbb R^2\to{\mathbb C}^2$ the norms for fixed small $\alpha>0, \sigma>0$,
   \begin{align}\label{norm*1}
\|\psi\|_* =   \|\psi^+\|_{*, 1}   + \|\psi^-\|_{*, 1},
\end{align}
 \begin{align}\label{normstar1}
\|\psi^\pm\|_{*,1} = \sum_{j=1}^2 \| v_d^\pm \psi^\pm \|_{C^{2,\alpha}(\ell_j<3)} + \| {\rm{Re}}(\psi^\pm) \|_{*, 1, {Re}} + \| {\rm{Im}}(\psi^\pm) \|_{*, 1, {Im}},
\end{align}
where, by the relations
$$
{\rm{Re}}(\psi^\pm)=\psi^\pm_1,\quad {\rm{Im}}(\psi^\pm)=\psi^\pm_2,
$$
we have denoted
\begin{align}
\|\psi_1^\pm\|_{*, 1, Re}
&=
\sup_{\ell_1>2,\ell_2>2}
|\psi_1^\pm|
\ +\ \sup_{2<\ell_1<\frac{2}{\e} , \, 2<\ell_2<\frac{2}{\e}}
\frac{|\nabla \psi_1^\pm| }{\ell_1^{-1}+\ell_2^{-1}}
\ +\ \sup_{r>\frac{1}{\e}}
\Biggl[
\frac{1}{\e}
|\partial_r \psi_1^\pm|
\,+\, |\partial_s \psi_1^\pm|
\Biggr]
\nonumber
\\[2mm]
& \qquad
\ +\ \sup_{2<\ell_1<\mathbb{R}_\e , \,
2<\ell_2< \mathbb{R}_\e }
\frac{|D^2 \psi_1^\pm| }{\ell_1^{-2}+\ell_2^{-2}}
\ +\ \sup_{2<\ell_1<\mathbb{R}_\e , \,
2<\ell_2<\mathbb{R}_\e}
\frac{[D^2\psi_1^\pm]_{\alpha,B_{|z|/2}(z)}}{ \ell_1^{-2-\alpha} + \ell_2^{-2-\alpha}},
\\[4mm]
\|\psi_2^\pm\|_{*, 1, Im}
&=
\sup_{\ell_1>2,\ell_2>2}
\frac{|\psi_2^\pm|}{\ell_1^{-2+\sigma}+\ell_2^{-2+\sigma}+\e^{\sigma-2}}
\ +\ \sup_{2<\ell_1<\frac{2}{\e} , \, 2<\ell_2<\frac{2}{\e}}
\frac{|\nabla \psi_2^\pm| }{\ell_1^{-2+\sigma}+\ell_2^{-2+\sigma}}  \nonumber
\\[2mm]
& \qquad
\ +\ \sup_{r>\frac{1}{\e}}
\Big[
\e^{\sigma-2}|\partial_r \psi_2^\pm|
+\e^{\sigma-1}|\partial_s \psi_2^\pm|
\Big] \nonumber
\\
& \qquad
\ +\ \sup_{2<\ell_1<\mathbb{R}_\e ,
 \, 2<\ell_2< \mathbb{R}_\e }
\frac{|D^2 \psi_2^\pm| }{\ell_1^{-2+\sigma}+\ell_2^{-2+\sigma}}
\ +\ \sup_{2<\ell_1<\mathbb{R}_\e , \,
2<\ell_2<\mathbb{R}_\e}
\frac{[D^2\psi_2^\pm]_{\alpha,B_{1}(z)}}{  \ell_1^{-2+\sigma} +  \ell_2^{-2+\sigma}}.
\label{norm*4}
\end{align}

   We next give the   solvability of   linear projected problem \eqref{eq:linear}.

   \begin{proposition}\label{prop:linearfull}
   There exists a constant $C>0$ depending only on  $\alpha, \sigma \in(0,1)$, such that  the following hold:
if $\rH$ satisfies \eqref{s4} and
$$
\| \rH\|_{**}<+\infty,
$$
then for $\e>0$ sufficiently small there exists a unique solution $ T_\e(\rH) = \big(T_\e^1(\rH), T_\e^2(\rH) \big)= (\psi_\epsilon,  c_\e) $ to \eqref{eq:linear}.
Furthermore, there  holds
\begin{equation*}
\|\psi_\e\|_*\leq C\|\rH\|_{**}.
\end{equation*}
\end{proposition}
\begin{proof}
We delay the proof of this proposition to \S\ \ref{section55.1}.
\end{proof}

\medskip
        \subsection{More precise estimates  and decompositions  for $\psi$}
\label{section4.2}
  The resolution theory  of \eqref{eq:linear} provided  in Proposition \ref{prop:linearfull}
implies that   we can find   a solution of \eqref{eq:linear5.1} in the region
  \begin{equation*}\begin{split}
\mathcal{A}:=\Bigg\{& \psi:\; \psi \text{ satisfies } \eqref{s4}, \ {\rm{Re}} \int_{B(0,4)} \chi \big[ \overline{\phi_j^+} w_{x_1}^++ \overline{\phi_j^-} w_{x_1}^- \big]=0, \;\;j=1,2, \; \text{and} \;\|\psi\|_*\leq\frac{C}{|\ln \e|} \Bigg\}.
\end{split}\end{equation*}
In fact,  from  Proposition \ref{prop:linearfull},   the problem \eqref{eq:linear5.1}  is equivalent to the following fixed point problem
 \begin{equation}\nonumber
\psi=T_\e\Big(\rR+\mathcal N(\psi) \Big)=: \mathbb{G}_\e(\psi).
\end{equation}
The existence of $\psi$ to \eqref{eq:linear5.1} follows by  the Contraction Mapping  Principle, see Proposition \ref{proposition61} in Section \ref{section6}  for more details.

\medskip
Next, to find a real solution of \eqref{equation-for-psi}, we need to  solve the reduced problem by finding a suitable $\hat d$ such that the  multiplier $c$ in \eqref{eq:linear5.1} is identical zero  when  $\e$ is sufficiently  small.
However,   since  the estimates for $\psi$ are not   delicate enough,  we can not carry out reduction procedure such that $c$ is zero.
In order to get around the technical difficulty,  we  are led to using the idea  doing more precise  decompositions  and estimates for  the perturbation $\psi$.
This   was first   introduced by J. D\'avila, M. del Pino, M. Medina and R. Rodiac in     \cite{DDMR} and \cite{DDMR1}
 to construct vortex helical filaments of classical (single complex component) Ginzburg-Landau equation   and Gross-Pitaevskii equation.

%
 For $\psi = \big(\psi^+, \psi^- \big)\in \mathbb{C}^2 $ satisfying \eqref{s4}, we
 decompose \(\psi\) in Fourier series in $\theta_j$ as in \eqref{def:h_Fourier}
and define
\begin{align*}
\psi_{e,j}^\pm= \sum_{k \text{ even}} \psi^\pm_{k,j},
\qquad
\psi_{o,j}^\pm = \sum_{k \text{ odd}} \psi_{k,j}^\pm, \ \ j=1, 2.
\end{align*}
Similar to the definitions of  $\rH_o$ and $\rH_e$ in Section \ref{section Error estimate of Approximation}, we define
\begin{align}
&\psi_o = (\psi_o^+, \psi_o^-),
\quad \text{with}\ \
\psi^\pm_o : =    \eta_{1,\frac{1}{2}\mathbb{R}_\e} \psi_{o, 1}^\pm
+ \eta_{2,\frac{1}{2}\mathbb{R}_\e}  \psi_{o, 2}^\pm,
\label{def-psio}
\\[4mm]
& \psi_e =  (\psi_e^+, \psi_e^-), \quad \text{with}\quad  \psi^\pm_e : = \psi^\pm - \psi^\pm_o,
\label{def-psie}
\end{align}
where $\eta_{1,\frac{1}{2}\mathbb{R}_\e}$ and $\eta_{2,\frac{1}{2}\mathbb{R}_\e}$ are two cut-off functions given by the form in \eqref{cutoff}.

\medskip
We recall some facts about the error in the form
$$
\rR=(\rR^+, \rR^-)=\Bigg(\frac{S^+(v_d)}{iv_d^+}, \frac{S^-(v_d)}{iv_d^-}\Bigg).
$$
There are some odd Fourier modes  terms which  have   good size but decay slowly.
Therefore, we have introduced the semi-norm $|\cdot|_{\sharp\sharp} $  in \eqref{sharpsharpnorm} for  these slow decay terms.
As we have stated, by taking   $\rH^\pm =-\rR^\pm-\mathcal N^\pm(\psi)$  in  Proposition \ref{prop:linearfull} we can solve the nonlinear projected problem \eqref{eq:linear5.1} in Proposition \ref{proposition61}.  And then, in order to  carry out  the reduction process,  we will establish  Proposition \ref{prop:sharp2b}.
To prove   Proposition \ref{prop:sharp2b},  we would have to    solve  some problems like $\Delta \psi_1^\pm \approx  O\Big( \frac 1{\ell_1}\Big)$ and thus the functions $\psi_1^\pm$ would grow logarithmically up to a certain distance, see      Lemma  \ref{prioriestimate2} and Lemma \ref{lem:estimates-3b} for more  explanations.
As a   result of these, we  need   the following semi-norm   to capture  the  behaviors of $\psi =\big(\psi^+, \psi^- \big)= \big(\psi^+_1 +i \psi^+_2,  \psi^-_1 +i \psi^-_2 \big) $:
 \begin{align}\label{sharpnorm}
| \psi |_\sharp &=  | \psi^+|_{\sharp, 1}   + | \psi^-|_{\sharp, 1}
\end{align}
with
 \begin{align}\label{sharp_1}
| \psi^\pm |_{\sharp, 1} &=
\sum_{j=1}^2
|\log\e|^{-1}
\| V_d \psi^\pm \|_{C^{2,\alpha}(\ell_j<3)}
+ |\psi_1^\pm|_{\sharp,1, Re}+ |\psi_2^\pm|_{\sharp, 1, Im} ,
\end{align}
where, by the relations
$$
{\rm{Re}}(\psi^\pm)=\psi^\pm_1,\quad {\rm{Im}}(\psi^\pm)=\psi^\pm_2,
$$
we have denoted
\begin{align}
\label{normSharp1}
|\psi_1^\pm|_{\sharp, 1, Re}
& =
\sup_{2<\ell_j < \mathbb{R}_\e, j=1,2}
\Biggl[
\frac{|\psi_1^\pm|}{\ell_1 \log(2\mathbb{R}_\e/\ell_1)+ \ell_2 \log(2\mathbb{R}_\e/\ell_2)}
\ +\
\frac{|\nabla \psi_1^\pm|}{ \log(2\mathbb{R}_\e/\ell_1)+  \log(2\mathbb{R}_\e/\ell_2)}
\Biggr],
\end{align}
and
\begin{align}
\label{normSharp2}
|\psi_2^\pm|_{\sharp, 1, Im}
&=
\sup_{2<\ell_j < \mathbb{R}_\e, j=1,2}
\Biggl[
\frac{|\psi_2^\pm|+|\nabla \psi_2^\pm|}{\ell_1^{-1+\sigma} + \ell_2^{-1+\sigma}+\ell_1^{-1} \log(2\mathbb{R}_\e/\ell_1) + \ell_2^{-1} \log(2\mathbb{R}_\e/\ell_2)}
\Biggr].
\end{align}

We recall the semi-norm $|\cdot|_{\sharp \sharp} $ in \eqref{sharpsharpnorm}
and the reflection mapping   $\mathscr{R}_j z$ in \eqref{mathscrR}.
We finally give  a crucial proposition in this paper. In some sense, it is the ``heart" of the present paper.

\begin{proposition}
\label{prop:sharp2b}
Suppose that $\rH$ satisfies the symmetries \eqref{s4} and  $\|\rH \|_{**}<\infty$.
$\rH_o$ is the function  defined  in \eqref{def-ho}  and it can be  decomposed as
$$
\rH_o =\rH_o^\alpha + \rH_o^\beta,
$$
where $|  \rH_o^\alpha |_{\sharp\sharp}<\infty$, and the terms $\rH_o^\alpha$, $\rH_o^\beta$ satisfy \eqref{s3}, i.e,
\begin{align}
\nonumber
\rH_o^\kappa (\mathscr{R}_j z ) =  \overline{\rH_o^\kappa(z)}  ,\quad
|z-e_j | <  \mathbb{R}_\e, \quad j=1,2 \quad \text{and} \quad  \kappa =\alpha,\beta,
\end{align}
and have supports in $B_{2\mathbb{R}_\e}(e_1) \cup B_{2\mathbb{R}_\e}(e_2) $.
Then for any solution $\psi = \psi_e + \psi_o$ of \eqref{eq:linear}   with  $\psi_o$ defined by \eqref{def-psio}, $\psi_o$ can be decomposed as
$$
\psi_o = \psi_o^\alpha + \psi_o^\beta,
$$
where   $\psi_o^\kappa$   are supported in $B_{\mathbb{R}_\e}(e _1) \cup B_{\mathbb{R}_\e}(e _2)$ with $\kappa =\alpha, \beta$.
Furthermore,  the following estimates     hold
\begin{align*}
|\psi^\alpha_o|_\sharp
& \lesssim
|\rH^\alpha_o |_{\sharp\sharp}
+ \e | \log\e|^\frac{1}{2} ( \|\rH^\alpha_o\|_{**}
+
\|  \rH - \rH_o\|_{**} ),
\\[2mm]
\| \psi^\beta_o \|_*
&\lesssim
 \|\rH_o^\beta\|_{**},
 \\[2mm]
 \|\psi^\alpha_o\|_*+\|\psi^\beta_o\|_* & \lesssim \|\rH\|_{**}+\|\rH^\alpha_o\|_{**}+\|\rH^\beta_o\|_{**},\nonumber
\end{align*}
and
\begin{align}
\nonumber
\psi_o^\kappa(\mathcal{R}_j z ) =  \overline{\psi_o^\kappa (z)}  ,\quad
|z-e_j | <  \mathbb{R}_\e, \quad j=1,2,\ \kappa=\alpha,\beta.
\end{align}
\end{proposition}
\begin{proof}
The proof of Proposition~\ref{prop:sharp2b}  will be provided   in  \S\ \ref{sec:prop2}.
\end{proof}

\bigskip

  \section{proof of linear theory}\label{section55}
  \subsection{Proof of Proposition \ref{prop:linearfull} }\label{section55.1}
 We now consider the following: for given $\rH$ satisfying \eqref{s4}, find $\psi$ such that
%
%
%
 \begin{align}
\label{eq:linearhomogeneous}
\left\{
\begin{aligned}
& \LLL^\pm_\e(\psi)=\rH^\pm
\quad\text{in }\R^2,
\\[2mm]
&{\rm{Re}} \int_{B(0,4)} \chi \big[ \overline{\phi_j^+} w_{x_1}^++ \overline{\phi_j^-} w_{x_1}^- \big]=0, \text{ with }\phi^\pm_j(z)=iw^\pm(z)\psi^\pm(z+e_j ),
\\[2mm]
&\psi \text{ satisfies the symmetry }  \eqref{s4}.
\end{aligned}
\right.
\end{align}
Recall the definitions of norms $\|\cdot\|_*$ in \eqref{norm*1}-\eqref{norm*4} and $\|\cdot\|_{**}$ in \eqref{def:norm_**0}-\eqref{def:norm_**}.
We shall first give a priori estimate  for the problem \eqref{eq:linearhomogeneous}.
  \begin{lemma}\label{FirstEstimate}
There exists a constant $C>0$, depending only on $\alpha\in (0,1), \sigma\in (0,1)$,   such that for all $\e$ sufficiently small and any solution $\psi$ of \eqref{eq:linearhomogeneous} with $\|\psi\|_*<\infty$ one has
\begin{align}
\label{est0a}
\|\psi\|_*\leq C\|\rH\|_{**}.
\end{align}
\end{lemma}

\begin{proof}
 First, we introduce the norms
\begin{align*}
\|\psi\|_{\star}  =   \|\psi^+\|_{\star, 0}   + \|\psi^-\|_{\star, 0}\, ,
\end{align*}
and
\begin{align*}
  \|\rH\|_{\star\star} = \|\rH^+\|_{\star\star, 0}+ \|\rH^-\|_{\star\star, 0}\, ,
 \end{align*}
where we have set the notation by the following
 \begin{align*}
\|\psi^\pm\|_{\star, 0} = \sum_{j=1}^2 \| v_d^\pm \psi^\pm \|_{C^{0}(\ell_j<3)} + \|  \psi^\pm_1 \|_{\star, 0, Re} + \| \psi^\pm_2\|_{\star, 0, Im},
\end{align*}
\begin{align}
\|\psi_1^\pm\|_{\star, 0, Re}
&=
\sup_{\ell_1>2,\ell_2>2}
|\psi_1^\pm|
+\sup_{2<\ell_1<\frac{2}{\e} , \, 2<\ell_2<\frac{2}{\e}}
\frac{|\nabla \psi_1^\pm| }{\ell_1^{-1}+\ell_2^{-1}}
+ \sup_{r>\frac{1}{\e}}
\Bigl[
\frac{1}{\e}
|\partial_r \psi_1^\pm|
+ |\partial_s \psi_1|
\Bigr],
\nonumber
\\[4mm]
\|\psi_2^\pm\|_{\star, 0, Im}
&=
\sup_{\ell_1>2,\ell_2>2}
\frac{|\psi_2^\pm|}{\ell_1^{-2+\sigma}+\ell_2^{-2+\sigma}+\e^{\sigma-2}}
+\sup_{2<\ell_1<\frac{2}{\e} , \, 2<\ell_2<\frac{2}{\e}}
\frac{|\nabla \psi_2^\pm| }{\ell_1^{-2+\sigma}+\ell_2^{-2+\sigma}}
\nonumber
\\[2mm]
& \qquad
+\sup_{r>\frac{1}{\e}}
\left[
\e^{\sigma-2}|\partial_r \psi_2|
+\e^{\sigma-1}|\partial_s \psi_2|
\right], \nonumber
\end{align}
and
 \begin{align*}
\|\rH^\pm\|_{\star\star, 0}
&:=\sum_{j=1}^2 \| v_d^\pm \rH^\pm\|_{C^0(\ell_j<3)}
+ \sup_{\ell_1>2,\ell_2>2}
\Biggl[
\frac{| {\rm{Re}}(\rH^\pm)|}{\ell_1^{-2}+\ell_2^{-2} }
+ \frac{ |{\rm{Im}}(\rH^\pm)|}{\ell_1^{-2+\sigma}+\ell_2^{-2+\sigma}}
\Biggr].
\end{align*}
We claim that
\begin{align}
 \|\psi\|_{\star}\leq C\|\rH\|_{\star\star}. \label{formula5.2}
\end{align}
By using Schauder theory, we then can get the full estimate \eqref{est0a}.

\medskip
In  \cite{DDMR}, the authors gave a proof of \eqref{formula5.2} for the complex-valued scalar case ($B=0$).
We adopt  the methods and technology  from  \cite{DDMR}.  In  the present case, we   need to  deal with   some new  problems caused by the  coupled terms due to that $B\ne0$.
For convenience, we only highlight the difference.

 We argue  by contradiction. Assume that there exist sequences of the parameter epsilon $\{\e_n\}$ approaching $0^+$  and solutions $\{\psi^{n}\}$  in such a way that \eqref{eq:linearhomogeneous} is valid with $\rH= \rH^{(n)} $  and  the following estimates are true
\begin{equation}
\|\psi^{n} \|_{\star}=1, \qquad \| \rH^{(n)} \|_{\star\star}=o_n(1).
\label{ContraAssumption}
\end{equation}

Since the linearized operator $\LLL_\e = (\LLL^+_\e, \LLL^-_\e) $ has different  asymptotic behaviors  in  different domains, therefore,
in order to obtain a contradiction, we will divide the analysis into two parts.

\medskip
\noindent $\spadesuit $
In the region near the vortices,  by a similar method  to the proof of Lemma $5.1$ in  \cite{DDMR},   we can show that
\begin{align}\label{converencelocal}
\phi_j^n = \big(iw^+ (\psi^{n})^+,  iw^- (\psi^{n})^- \big) (z+e_j ) \longrightarrow 0, \quad \text{in}~C_{loc}^2(\R^2, \mathbb{C}^2).  \end{align}
Here we point out that  we have used the non-degeneracy results  in Lemma \ref{lem:ellipticestimatesL0-b} in the proof of \eqref{converencelocal}.
The linearized operator $L_0(\phi)$  does have one kernel element $w_{x_1} $ when the perturbations satisfy some decays  and  symmetry  constraints in that lemma.
Then the orthogonality in \eqref{eq:linearhomogeneous} will imply the validity of \eqref{converencelocal}.

\medskip
\noindent $\spadesuit $
Next  we consider the case far away from the vortices. Here we assume $ x_1 >0$ and $\ell_1>R_0$.
For $\rH=\big( \rH^+, \rH^-\big)$,
we  express $  \LLL_\e^\pm(\psi)=\rH^\pm  $ as the following
 \begin{align}  \label{psi1equation}
 \Delta \psi^\pm +    \frac{\nabla v_d^\pm \nabla \psi^\pm } {v_d^\pm}   &  - 2i A_\pm |v_d^\pm|^2 \psi^\pm_2
   -2 i B   |v_d^\mp|^2\psi^\mp_2 \nonumber
  \\[2mm]
&  \quad
  + \e^2 \Bigl( \partial_{ss}^2 \psi^\pm
+\frac{2\partial_s v_d^\pm  }{v_d^\pm }\partial_s \psi^\pm
- 4i \partial_s \psi^\pm   \Bigr)= \rH^\pm.
 \end{align}
Noting that
$$
\psi^\pm  = \psi^\pm_1+ i \psi^\pm_2,
\qquad
\rH^\pm = \rH_1^\pm+ i\rH_2^\pm ,
$$
%
we now try   to simplify the above system \eqref{psi1equation} for $(\psi^+, \psi^-)$.
   We rewrite    the system  \eqref{psi1equation} to the following one in $\R^2$
  \begin{align} \label{system for psi2}
\left\{
\begin{aligned}
&  \Delta \psi_1^+    + \e^2 \partial_{ss}^2 \psi_1^+   = \tilde{\mathcal H}_1^+,
\\[2mm]
&  \Delta \psi_1^-    + \e^2 \partial_{ss}^2 \psi_1^-   = \tilde{\mathcal H}_1^-,
\\[2mm]
&  \Delta \psi_2^+    + \e^2 \partial_{ss}^2 \psi_2^+ - 2A_+ (t^+)^2 \psi^+_2
   -2  B   (t^-)^2 \psi^-_2   = \tilde{\mathcal H}_2^+,
\\[2mm]
&\Delta \psi_2^-    + \e^2 \partial_{ss}^2 \psi_2^- - 2A_- (t^-)^2\psi^-_2
   -2  B   (t^+)^2 \psi^+_2   = \tilde{\mathcal H}_2^-,
\end{aligned}
\right.
\end{align}
with
\begin{align*}
 \tilde{\mathcal H}_1^\pm   &=  {\mathcal H}_1^\pm - \left(\frac{\nabla  W^\pm(\ell_1) }{{W^\pm(\ell_1)} }+\frac{\nabla {W^\pm(\ell_2)} }{{W^\pm(\ell_2)} } \right) \nabla \psi_1^\pm - \nabla (\theta_1+\theta_2) \nabla \psi_2^\pm  \nonumber
 \\[2mm]    & \quad
   - 2\e^2 \left[ \left(\frac{\p_s{W^\pm(\ell_1)} }{{W^\pm(\ell_1)} }+\frac{\p_s {W^\pm(\ell_2)} }{{W^\pm(\ell_2)} } \right) \partial_s \psi_1^\pm -\p_s( \theta_1+\theta_2) \p_s \psi_2^\pm  \right]
-4 \e^2 \partial_s \psi_2^\pm,
 \end{align*}
and
\begin{align*}
 \tilde{\mathcal H}_2^\pm   &=  {\mathcal H}_2^\pm - \left(\frac{\nabla  W^\pm(\ell_1) }{{W^\pm(\ell_1)} }+\frac{\nabla {W^\pm(\ell_2)} }{{W^\pm(\ell_2)} } \right) \nabla \psi_2^\pm - \nabla (\theta_1+\theta_2) \nabla \psi_1^\pm
 \nonumber
 \\[2mm]
 &\quad   -  2\e^2 \left[ \left(\frac{\p_s{W^\pm(\ell_1)} }{{W^\pm(\ell_1)} }-\frac{\p_s {W^\pm(\ell_2)} }{{W^\pm(\ell_2)} } \right) \partial_s \psi_2^\pm -\p_s( \theta_1+\theta_2) \p_s \psi_1^\pm  \right]
+4 \e^2 \partial_s \psi_1^\pm
\nonumber
 \\[2mm]
   & \quad
   + 2A_\pm\big( |v_d^\pm|^2-{t^\pm}^2\big) \psi^\pm_2
   \ +\ 2  B\big(   |v_d^\mp|^2-{t^\mp}^2\big)\psi^\mp_2.
\end{align*}
Since the  matrix
   $\mathbb{M}=\begin{pmatrix}
  A_+ |t^+|^2 &B |t^-|^2
  \\
  B |t^+|^2 &A_- |t^-|^2
  \end{pmatrix} $ is positive definite,
there then exists an invertible matrix ${\bf C}    $ such that
\begin{align}
{\bf C}^T {\bf C} = {\bf I}, \quad  2{\bf C} \, \mathbb{M}\,  {\bf C}^T =  \text{diag}(     \lambda^+,      \lambda^-) \quad \text{with}~\lambda^+,\, \lambda^->0.
\label{diagonal}
\end{align}
%
   Now we define the new vector functions
   \begin{align}
  & \tilde \psi_1 =(\tilde\psi_1^+, \tilde \psi_1^-)^T,  \label{fortildepsi1}
   \\[2mm]
&\tilde \psi_2 =(\tilde\psi_2^+, \tilde \psi_2^-)^T: ={\bf C}^T     \, ( \psi_2^+,   \psi_2^-)^T,   \label{fortildepsi2}
  \end{align}
    and \begin{align*}
      &  \mathbb{H}_2 =( \mathbb{H}_2^+,  \mathbb{H}_2^-) : =  {\bf C}^T   \, ( \tilde{\mathcal H}_2^+,    \tilde{\mathcal H}_2^-)^T,
      \\[2mm]  &
         \mathbb{H}_1 =( \mathbb{H}_1^+,  \mathbb{H}_1^-) : =    ( \tilde{\mathcal H}_1^+,    \tilde{\mathcal H}_1^-)^T.
      \end{align*}
Then  if $|z-e_j |>R_0, $   the system \eqref{system for psi2} can be rewritten as
\begin{align}
    &  ( \Delta + \e^2\partial_{ss}^2)  \tilde \psi_1   =   \mathbb{H}_1, \label{modified-sys1}
       \\[2mm]
 &  ( \Delta + \e^2\partial_{ss}^2)  \tilde \psi_2  -  \text{diag}(     \lambda^+,      \lambda^-)   \tilde \psi_2  =   \mathbb{H}_2. \label{modified-sys2}
 \end{align}
For the system \eqref{modified-sys1}-\eqref{modified-sys2}, we will use the barrier arguments based on the maximum principle provided in Lemmas \ref{lem:comparison_principle}-\ref{lem:comparison_principle_Neumann}.
We obtain that, by some calculations,
\begin{align*}
|\mathbb{H}_2| \le C\big(\| \rH\|_{\star\star,0} + R_0^{-\sigma} +\e^\sigma\big) \big( \ell_1^{\sigma-2} +\e^{2-\sigma} \big).
\end{align*}
By the comparison principle, elliptic estimates,  and  choosing  the barrier as
\begin{align*}
\mathcal B_2 =  \hat C  \Big( \|\rH\|_{\star\star}+R_0^{-\sigma}+\e^{\sigma}+\| \tilde \psi_2^\pm\|_{L^\infty(B_{R_0}(\tilde{d}))}  \Big) \Big( \ell_1^{\sigma-2} +\e^{2-\sigma} \Big),
\end{align*}
for some fixed  large constant  $\hat C>0$, we   obtain
\begin{align*}
\| \tilde \psi_2^\pm\|_{\star,0, Im} \le C \Big( \|\rH\|_{\star\star}+R_0^{-\sigma}+\e^{\sigma}+\|\tilde \psi_2^\pm\|_{L^\infty(B_{R_0}(\tilde{d}))}  \Big).
\end{align*}
On the other hand, by a similar argument as in \cite{DDMR}, we have
\begin{align*}
\| \tilde \psi_1^\pm \|_{\star, 0, Re}
\leq  C
\Big( \|\rH\|_{\star\star} + R_0^{\sigma-1}+\e^{1-\sigma} + \|\psi_1^\pm\|_{L^\infty(B_{R_0}(\tilde d))}\Big).
\end{align*}
Meanwhile, using \eqref{converencelocal} and $\| \rH^{(n)} \|_{\star\star,0}=o_n(1)$, we can get
 \begin{align*}
 \| \tilde \psi\|_{\star} \le C \| \rH\|_{\star\star}, \quad \text{with}~\tilde \psi =  (\tilde \psi_1^++i \tilde \psi_2^+,\, \tilde \psi_1^-+i \tilde \psi_2^-).
 \end{align*}
  Since ${\bf C}^T$ is invertible, we directly get
\begin{align*}
\|\psi\|_{\star}\leq C\|\rH\|_{\star\star}.
\end{align*}

Therefore, by combining the above results we derive an estimate contradicting to \eqref{ContraAssumption}.
%
%
%
\end{proof}

   We next give the proof of Proposition \ref{prop:linearfull}.

\begin{proof}[Proof of Proposition~ \ref{prop:linearfull}]  For given $M> 100\tilde d$, we first consider the problem \eqref{eq:linear} locally as the following
 \begin{align}
\label{eq:linear-local}
\left\{
\begin{aligned}
& \LLL^\pm_\e(\psi)=\rH^\pm+  c \sum_{j=1}^2\frac{\chi_j}{iw^\pm(z-e_j )} (-1)^jw^\pm_{x_1}(z-e_j )
\quad\text{in } B_M(0),
\\[2mm]
&{\rm{Re}} \int_{B(0,4)} \chi \big[ \overline{\phi_j^+} w_{x_1}^++ \overline{\phi_j^-} w_{x_1}^- \big]=0, \text{ with }\phi^\pm_j(z)=iw^\pm(z)\psi^\pm(z+e_j ),
\\[2mm]
&\psi \text{ satisfies the symmetry }  \eqref{s4}~\text{and }\psi =0 ~ \text{on}~\p {B_M(0)}.
\end{aligned}
\right.
\end{align}

Let  \begin{align*}
 \B :=& \Bigg\{\phi= (iv_d^+\psi^+,  iv_d^-\psi^-) \in H^1_0\big(B_M(0),\mathbb{C}^2\big)\, :\,  \;\;\psi \text{ satisfies \eqref{s4}},
 \\[3mm]
    &\qquad \qquad  {\rm{Re}} \int_{B(0,4)} \chi \Big[ \overline{\phi_j^+} w_{x_1}^++ \overline{\phi_j^-} w_{x_1}^- \Big]=0, \ j=1,2,\Bigg\}.
\end{align*}
We endow  the space $\B$ with   the inner product
\begin{equation} \label{inner product}
[\phi,\varphi]_ \B:=   {\rm{Re}}  \int_{B_M(0)}\Bigg[ \big( \,\nabla \phi^+ \overline{\nabla\varphi^+} +\e^2\p_s \phi^+ \overline{\p_s \varphi^+} \,\big)
\ +\
\big( \,\nabla \phi^- \overline{\nabla\varphi^-} +\e^2\p_s \phi^- \overline{\p_s \varphi^-} \,\big) \Bigg],
\end{equation}
for any $\phi, \varphi \in\B. $
 Due to the Poincar\'e inequality, we    know  that  the space  equipped with  the  topology \eqref{inner product}   is  a subspace of
$H^1_0\big(B_M(0),\mathbb{C}^2\big)$.
By some calculations, the following identities hold
\begin{align*}
\Delta \Phi^\pm &+ \Big[ A_\pm \big({t^\pm}^2-|v_d^\pm|^2 \big)+B\big({t^\mp}^2-|v_d^\mp|^2\big)\Big]\Phi^\pm   + (\eta_1-1)  \frac{\mathbb{E}^\pm} {v_d^\pm}\Phi_j^\pm (z) \nonumber
\\[2mm]
&-\,2A_\pm {\rm{Re}}\Big(v_d^\pm\overline{\Phi^\pm}\Big)v_d^\pm\,-\,2B {\rm{Re}}\Big(v_d^\mp \overline{\Phi^\mp}\Big)v_d^\pm
 +
\e^2(\p^2_{ss}\Phi^\pm -4i\p_s\Phi^\pm -4\Phi^\pm)
\nonumber
\\[2mm]
& =   i v_d^\pm  \rH^\pm +  iv_d^\pm c  \sum_{j=1}^2\frac{\chi_j}{iw^\pm(z-e_j )} (-1)^jw^\pm_{x_1}(z-e_j ), \quad \text{in}~B_M(0),
\end{align*}
with
 \[\Phi^\pm=  iv_d^\pm \psi^\pm.  \]
For any test function $\varphi = (\varphi^+, \varphi^-) $,  we rewrite these equations in the sense of  distribution
 \begin{align}  \label{variational form}
  &  -{\rm{Re}}  \int_{B_M(0)}\Big[ \big( \nabla \Phi^+ \overline{\nabla\varphi^+} +\e^2\p_s \Phi^+ \overline{\p_s \varphi^+} \big) +\big( \nabla \Phi^- \overline{\nabla\varphi^-} +\e^2\p_s \Phi^- \overline{\p_s \varphi^-} \big)   \Big]
  \nonumber
\\[2mm]
&\qquad\qquad
  -\e^2  {\rm{Re}}  \int_{B_M(0)}\Big[  \big(4i\p_s\Phi^+ -4\Phi^+\big)\, \overline{\varphi^+} + \big(4i\p_s\Phi^- -4\Phi^-\big)\, \overline{\varphi^-}\Big]
    \nonumber
\\[2mm]
&\qquad\qquad
  -   {\rm{Re}}  \int_{B_M(0)}  2A_\pm {\rm{Re}}\Big(v_d^+\overline{\Phi^+}\Big)v_d^+ \overline{\varphi^+}\,+\,2B {\rm{Re}}\Big(v_d^- \overline{\Phi^-}\Big)v_d^+ \overline{\varphi^+}
      \nonumber
\\[2mm]
&\qquad\qquad
  -   {\rm{Re}}  \int_{B_M(0)}  2A_\pm {\rm{Re}}\Big(v_d^-\overline{\Phi^-}\Big)v_d^- \overline{\varphi^-}\,+\,2B {\rm{Re}}\Big(v_d^+ \overline{\Phi^+}\Big)v_d^- \overline{\varphi^-}
        \nonumber
\\[2mm]
&\qquad\qquad
  +   {\rm{Re}}  \int_{B_M(0)}  \Bigg(  \Big[ A_+ \big({t^+}^2-|v_d^+|^2 \big)+B\big({t^-}^2-|v_d^-|^2\big)\Big]    + (\eta_1-1)  \frac{\mathbb{E}^+} {v_d^+}  \Bigg) \Phi^+\overline{\varphi^+}
        \nonumber
\\[2mm]
&\qquad\qquad
  +   {\rm{Re}}  \int_{B_M(0)}  \Bigg(  \Big[ A_- \big({t^-}^2-|v_d^-|^2 \big)+B\big({t^+}^2-|v_d^+|^2\big)\Big]    + (\eta_1-1)  \frac{\mathbb{E}^-} {v_d^-}  \Bigg) \Phi^-\overline{\varphi^-}
          \nonumber
\\[2mm]
& =   {\rm{Re}}  \int_{B_M(0)}   \Bigg( i v_d^+  \rH^- +  iv_d^+ c  \sum_{j=1}^2\frac{\chi_j}{iw^+(z-e_j )} (-1)^jw^+_{x_1}(z-e_j )\Bigg)  \overline{\varphi^+}
        \nonumber
\\[2mm]
&\qquad\qquad  + {\rm{Re}}  \int_{B_M(0)}   \Bigg( i v_d^-  \rH^- +  iv_d^- c  \sum_{j=1}^2\frac{\chi_j}{iw^-(z-e_j )} (-1)^jw^-_{x_1}(z-e_j )\Bigg)  \overline{\varphi^-}.
  \end{align}
Denote by  linear operator $\big \langle \mathcal{K}(\Phi), \cdot   \big\rangle_{\B}$  on $\B$  as
 \begin{align*}
  &
  -\e^2  {\rm{Re}}  \int_{B_M(0)}\Big[  \big(4i\p_s\Phi^+ -4\Phi^+\big)\, \overline{\varphi^+} + \big(4i\p_s\Phi^- -4\Phi^-\big)\, \overline{\varphi^-}\Big]
    \nonumber
\\[2mm]
&\qquad\qquad
  -   {\rm{Re}}  \int_{B_M(0)}  2A_\pm {\rm{Re}}\Big(v_d^+\overline{\Phi^+}\Big)v_d^+ \overline{\varphi^+}\,+\,2B {\rm{Re}}\Big(v_d^- \overline{\Phi^-}\Big)v_d^+ \overline{\varphi^+}
      \nonumber
\\[2mm]
&\qquad\qquad
  -   {\rm{Re}}  \int_{B_M(0)}  2A_\pm {\rm{Re}}\Big(v_d^-\overline{\Phi^-}\Big)v_d^- \overline{\varphi^-}\,+\,2B {\rm{Re}}\Big(v_d^+ \overline{\Phi^+}\Big)v_d^- \overline{\varphi^-}
        \nonumber
\\[2mm]
&\qquad\qquad
  +   {\rm{Re}}  \int_{B_M(0)}  \Bigg(  \Big[ A_+ \big({t^+}^2-|v_d^+|^2 \big)+B\big({t^-}^2-|v_d^-|^2\big)\Big]    + (\eta_1-1) \frac{\mathbb{E}^+} {v_d^+}\Bigg) \Phi^+\overline{\varphi^+}
        \nonumber
\\[2mm]
&\qquad\qquad
  +   {\rm{Re}}  \int_{B_M(0)}  \Bigg(  \Big[ A_- \big({t^-}^2-|v_d^-|^2 \big)+B\big({t^+}^2-|v_d^+|^2\big)\Big]    + (\eta_1-1) \frac{\mathbb{E}^+} {v_d^+}  \Bigg) \Phi^-\overline{\varphi^-}
          \nonumber
\\[2mm]
& =   :   \big \langle \mathcal{K}\Phi, \varphi     \big\rangle_{\B}.
  \end{align*}
 Also,  we   denote by  the linear operator $\big \langle \mathcal{S}(\Phi), \cdot   \big\rangle_{\B}$  on $\B$ as
\begin{align*}
   & {\rm{Re}}  \int_{B_M(0)}   \Bigg( i v_d^+  \rH^- +  iv_d^+ c  \sum_{j=1}^2\frac{\chi_j}{iw^+(z-e_j )} (-1)^jw^+_{x_1}(z-e_j )\Bigg)  \overline{\varphi^+}
        \nonumber
\\[2mm]
&  \quad + {\rm{Re}}  \int_{B_M(0)}   \Bigg( i v_d^-  \rH^- +  iv_d^- c  \sum_{j=1}^2\frac{\chi_j}{iw^-(z-e_j )} (-1)^jw^-_{x_1}(z-e_j )\Bigg)  \overline{\varphi^-}  =   :   \big \langle  \mathcal{S},     \varphi\big\rangle_{\B}.
\end{align*}
 Then  \eqref{variational form} is rewritten into the following form
 \begin{align*}
-  [\phi,\varphi]_\B + \big \langle \mathcal{K}\Phi, \varphi     \big\rangle_{\B} =  \big \langle  \mathcal{S},     \varphi\big\rangle_{\B}, \quad\text{for any}~ \varphi \in  C_0^\infty\big(B_M(0),\mathbb{C}^2\big).
 \end{align*}
Riesz representation theory  and compact Sobolev embedding  imply  that   there exists a  compact bounded linear operator $\mathbb{K}$ on $\B$, and $\mathbb{S} \in \B$  such that
\begin{align}\label{fredholm}
\Phi -  \mathbb{K} (\Phi) = \mathbb{S}.
\end{align}

In order to get the existence of a solution  to  \eqref{fredholm},
we need to show the homogeneous equation has only trivial solutions via the  Fredholm's theory.
Therefore, we shall first establish the a priori estimate of  multiplier $c$.
We consider  the equivalent form of \eqref{eq:linear-local} with translated variables
  \begin{align*}
    L^\pm_{d,j}(\Phi_j)    &  = i w^\pm(z)   \rH^\pm(z+e_j )  + c \chi_j  w^\pm_{x_1}(z),
    \nonumber
\\[2mm]
&  : =  \rH^\pm_j +  c \chi_j  w^\pm_{x_1}(z) \quad \text{in}~ B_{\tilde{d}} (0),
  \end{align*}
with
\[
\Phi_j =    \big( \Phi_j^+,  \Phi_j^- \big), \qquad \Phi_j^\pm(z) =i w^\pm(z)  \psi^\pm(z+e_j ),
\]
  and where the operator  $    L^\pm_{d,j}(\Phi_j)  $  is defined  in  \eqref{Lpmdj}.
 We can test  the equations  against $w^\pm_{x_1}$ respectively  to find
\begin{align*}
c=-\frac{1}{c_*} \Bigg[ { \rm{Re} }  \int_{B_{\tilde{d}} (0)}    \rH^\pm_j w^\pm_{x_1} -{ \rm{Re} } \int_{B_{\tilde{d}} (0)}     L^\pm_{d,j}(\Phi_j)   w^\pm_{x_1}\Bigg],
 \end{align*}
 for some   $c^* \sim C$. Using the expressions of $L^\pm_{d,j}(\Phi_j)$ in \eqref{Lpmdj},  by some tedious but straightforward analysis,   we can get the  expression for $c$
 \begin{equation}\label{expressionforc}
c=-\frac{1}{c_*} { \rm{Re} }  \int_{B_{\tilde{d}} (0)}   \rH^\pm_j w^\pm_{x_1} +O_\e(\e\sqrt{|\ln \e|})\frac{\|\psi\|_*}{c_*}.
\end{equation}
When $ \rH^\pm_j =0 $,     the expression for $c$ in \eqref{expressionforc} and the estimate in  Lemma \ref{FirstEstimate} imply that
\begin{align*}
\|\psi\|_*\leq C\| c\sum_{j=1}^2\chi_j(z)(-1)^j\frac{w^\pm_{x_1}(z-e_j )}{iw^\pm(z-e_j )} \|_{**}\leq C \e \sqrt{|\ln \e|} \|\psi\|_*,
\end{align*}   and thus there exists only trivial solution when  $ \rH^\pm_j =0 $.
It is worth noting that we now apply the norms $\|\cdot\|_{*}$ and $\|\cdot\|_{**}$ restricted on the bounded domains.
 Then the existence of \eqref{eq:linear-local} holds  and from Lemma \ref{FirstEstimate}  there exists a constant $C$ independent of $M$ such that the following estimate holds
\begin{equation*}
\|\psi_M\|_*\leq C \|\rH\|_{**}.
\end{equation*}
The real  $\psi$ solving \eqref{eq:linear} can be obtain by extracting  a subsequence such that $\psi_M\rightharpoonup \psi$ in $H^1_{\text{loc}}(\R^2)$.
  And the estimate $\|\psi\|_*\leq C\|\rH\|_{**}$ is a direct result of Lemma \ref{FirstEstimate}.
\end{proof}

\medskip
\subsection{Proof of Proposition \ref{prop:sharp2b}} \label{sec:prop2}
Consider the linear projected problem \eqref{eq:linearhomogeneous}
and recall the semi-norms $|\cdot|_{\sharp\sharp}$, $|\cdot|_{\sharp} $ in \eqref{sharpsharpnorm} and \eqref{sharpnorm}.
We first give the following a priori estimate.

 \begin{lemma}  \label{prioriestimate2}
 Suppose $\alpha\in (0,1), \sigma\in (0,1)$.
 There exists a constant $C$ depending only on $\alpha, \sigma$, such that for all $\e$ small enough, and any solution of \eqref{eq:linearhomogeneous},
 we have the estimate
\begin{equation}
\label{claimSharp}
|\psi|_\sharp \leq C(  |\rH|_{\sharp\sharp} + \e | \log\e|^\frac{1}{2} \|\rH\|_{**} )  .
\end{equation}

 \end{lemma}
\begin{proof}  We give  some technical essentials and the outline of the proof.
The reader can  refer to the proof of Lemma 5.2 in \cite{DDMR} for more details.

%

 \medskip
\noindent{\bf Step 1:} Similar to the proof of Lemma \ref{FirstEstimate}, we make a transformation   \[ (\tilde\psi_2^+, \tilde \psi_2^-)^T: ={\bf C}^T      \, ( \psi_2^+,  \psi_2^-)^T\]
such that  the main terms in the   linear operator  of  $\tilde \psi_2=(\tilde\psi_2^+, \tilde \psi_2^-)$ can be decoupled in the region far away from the vortices, see \eqref{modified-sys2}.

 \medskip
\noindent{\bf Step 2:} When  near the vortices,   we just proceed as in the proof of Lemma $5.2$ in \cite{DDMR}. We point out that in this step we use  Lemma \ref{Lepsilon} and the  non-degeneracy results in Lemma \ref{lem:ellipticestimatesL0-b} .
On the other hand,  when far away from the vortices,  we handle the system \eqref{modified-sys1}-\eqref{modified-sys2}  by the maximum principle,
see Lemmas \ref{lem:comparison_principle}-\ref{lem:comparison_principle_Neumann} in the Appendix.
In this step, we establish the estimate
\begin{equation*}
|\tilde \psi|_\sharp \leq C(  |\rH|_{\sharp\sharp} + \e | \log\e|^\frac{1}{2} \|\rH\|_{**} ).
\end{equation*}

  \medskip
\noindent{\bf Step 3:}    Finally,  since ${\bf C}^T $ is invertible, we  get that \eqref{claimSharp} holds.
\end{proof}

\medskip
In order to prove  Proposition   \ref{prop:sharp2b}, we also need  the following lemma.
Here we  recall the reflection $\mathscr R_j(z)$ defined in \eqref{mathscrR} and suppose that the function $\rH$ enjoys the  local  symmetry  \eqref{s4}
and the following
\begin{align}
  \rH(\mathscr{R}_j z ) = - \overline{\rH(z)}  ,\quad
|z-e_j | < 2 \mathbb{R}_\e, \quad \text{for}~j=1,2.   \tag{S2}
\end{align}
We note that the function $\rH$  satisfies  the  local  symmetry \eqref{s2} consisting in even Fourier modes.

\begin{lemma} \label{decompose and priori}
Suppose that $\rH$  enjoys the symmetry   \eqref{s2} and \eqref{s4} and
\[\| \rH\|_{**}\le +\infty. \]
For any solution $\psi$ of \eqref{eq:linearhomogeneous} with $\|\psi\|_*\le +\infty$,    there exist $\psi_s, \psi_* $ such that
\begin{align*}
\psi= \psi_s+ \psi_* ,
\end{align*}
where $\psi_s$  satisfies  the following symmetry
\begin{align*}
\psi_s(\mathcal{R}_j z ) &= - \overline{\psi_s(z)} ,\quad
|z-e_j | <  \R_\e .
\end{align*}
Moreover,  there exists a constant $C$ such that for all $\e$ small, the following estimates hold
\begin{align*}
\| \psi_s \|_* + \| \psi_*  \|_* \leq C  \| \rH \|_{**},
\end{align*}
\begin{align*}
| \psi_*  |_\sharp \leq C  \e |\log\e|^{\frac{1}{2}}   \| \rH \|_{**} .
\end{align*}
\end{lemma}
\begin{proof}
We  decompose  the operator $\LLL_\e = \Big( \LLL^+_\e, \LLL^-_\e\Big)$  as
\begin{align*}
\LLL_\e^\pm = \LLL_{\e,s, j}^\pm + \LLL_{\e,r, j}^\pm,
\end{align*}
with $\LLL_{\e,s,j}$ preserving  the symmetry \eqref{s2}.  More specifically, we have the expressions of $\LLL_{\e,s, j},  \LLL_{\e,r, j} $ in the coordinates $(\ell_j, \theta_j)$
\begin{align*}
\LLL_{\e,s,1}^\pm( \psi )  &=  \Delta \psi^\pm
+ 2\frac{\nabla w_a \nabla \psi^\pm}{w_a}
-2i A_\pm |w_a^\pm|^2 {\rm{Im}}(\psi^\pm) -2i  B |w_a^\mp|^2 {\rm{Im}}(\psi^\mp)
\nonumber
\\[2mm]
&   \quad + \e^2 \Biggl[ \tilde d^2 \partial_{\ell_1 \ell_1}^2 \psi \sin^2\theta_1
+\frac{\tilde d^2}{\ell_1}
\partial_{\ell_1 \theta_1}^2 \psi^\pm   \sin\theta_1   \cos\theta_1
+\partial_{\theta_1 \theta_1}^2 \psi^\pm \Bigl(1 + \frac{\tilde d^2}{\ell_1^2} \cos^2\theta_1 \Bigr)
\\[2mm]
&
\qquad\qquad
+ \partial_{\ell_1} \psi^\pm  \frac{\tilde d^2}{\ell_1} \cos^2\theta_1
-2 \partial_{\theta_1} \psi^\pm
\frac{\tilde d^2}{\ell_1^2} \sin\theta_1\cos \theta_1
\Biggr]
\\[2mm]
&
\quad
+   \frac{\hat d^2\eta_1}{|\ln \e|}    {W^\pm}'(\ell_1) \sin^2 \theta_1   \psi^\pm,
\end{align*}
and
\begin{align*}
\LLL_{\e,r,1}^\pm( \psi )  &=
 2\frac{\nabla w_b \nabla \psi^\pm}{w_b}
-2i A_\pm (|v_d^\pm|^2- |w_a^\pm|^2) {\rm{Im}}(\psi^\pm) -2i  B (|v_d^\mp|^2-|w_a^\mp|^2) {\rm{Im}}(\psi^\mp)
\nonumber
\\[2mm]
&\quad
+ \e^2 \Biggl[ 2 \tilde d
\partial_{\ell_1 \theta_1}^2 \psi
 \sin\theta_1
+2 \partial_{ \theta_1\theta_1}^2 \psi^\pm  \frac{\tilde d}{\ell_1} \cos\theta_1
+ \partial_{\ell_1} \psi^\pm  \tilde d \cos\theta_1
-  \partial_{\theta_1} \psi^\pm
\frac{\tilde d}{\ell_1} \sin\theta_1 \Biggr]
\nonumber
\\[2mm]
&\quad
+
\e^2 \Bigl( \frac{2\partial_s v_d^\pm }{v_d^\pm}
- 4i  \Bigr)
\Biggl[  \partial_{\ell_1} \psi^\pm \tilde d \sin \theta_1
+  \Big(1 + \frac{\tilde d}{\ell_1} \cos \theta_1 \partial_{\theta_1}\Big)\psi  \Biggr]
\nonumber
\\[2mm]
&\quad
+ \eta_1 \Bigg(  \frac{S^\pm(v_d)} {v_d^\pm} -  \frac{\hat d^2}{|\ln \e|}    {W^\pm}'(\ell_1) \sin^2 \theta_1  \Bigg)   \psi^\pm,
\end{align*}
with
$$
w^\pm_a(z) =W^\pm(\ell_1) e^{i \theta_1},\qquad  w^\pm_b(z) =W^\pm(\ell_2) e^{i \theta_2} .
$$
 The expressions of $\LLL_{\e,s,2}^\pm( \psi )$ and $\LLL_{\e,r,2}^\pm( \psi )    $  are similar.

 Next, we consider the following  problem
 \begin{align}
\nonumber
\left\{
\begin{aligned}
&  \LLL_{\e,s,j}^\pm (\psi_{s,j})= \rH^\pm  \eta_{j,2\R_\e}
\quad\text{in }\R^2,
\\[2mm]
&{\rm{Re}} \int_{B(0,4)} \chi \big[ \overline{\phi_{s,j}^+} w_{x_1}^++ \overline{\phi_{s,j}^-} w_{x_1}^- \big]=0,
\text{ with }\phi_{s,j}^\pm(z)=iw^\pm(z)\psi_{s,j}^\pm(z+e_j ),
\\[2mm]
&\psi_{s,j}^\pm \text{ satisfies } \psi_{s,j}^\pm(\bar z)  = - \psi_{s,j}^\pm(z).
\end{aligned}
\right.
\end{align}
Note that,  due to the local symmetry \eqref{s2},  $\rH$ is orthogonal to the kernel locally,
  thus the multiplier $c$ vanishes.
  Recall that $ \LLL_{\e,s,j}= \big(\LLL_{\e,s,j}^+, \LLL_{\e,s,j}^-\big)$ preserves  symmetry \eqref{s2} and $\rH$ satisfies \eqref{s2}.    Then for $\psi_{s,j}$, the symmetry \eqref{s2}   and the following estimate both  hold
  \begin{align*}
  \| \psi_{s,j} \|_*
\leq C \| \rH\|_{**}.
  \end{align*}
  Let
  \begin{align*}
\psi_s = \eta_{1,\frac{1}{2}\R_\e} \psi_{s,1}
+ \eta_{2,\frac{1}{2}\R_\e} \psi_{s,2} ,
\end{align*}
then  \begin{align*}
  \| \psi_{s} \|_*
\leq C \| \rH\|_{**}.
  \end{align*}
Define   $\hat \rH= (\hat \rH^\pm, \hat \rH^\pm) $ with
\begin{align*}
 \hat \rH &:= \rH -  \LLL_{\e,r,1}(\eta_{1,\frac{1}{2}R_\e } \psi_{s,1} ) -   \LLL_{\e,s,1}(\eta_{1,\frac{1}{2}R_\e } \psi_{s,1} )
- \LLL_{\e,r,2}(\eta_{2,\frac{1}{2}R_\e } \psi_{s,2} ) -   \LLL_{\e,s,2}(\eta_{2,\frac{1}{2}R_\e } \psi_{s,2} ),
\end{align*}
and consider the following
\begin{align} \label{prop:linearfullforhatpsi}
\left\{
\begin{aligned}
& \LLL_\e^\pm ( \hat \psi^\pm)=\hat \rH^\pm   + c \sum_{j=1}^2\frac{\chi_j}{iw^\pm(z-e_j )} (-1)^jw^\pm_{x_1}(z-e_j )
\quad\text{in }\R^2,
\\[2mm]
&{\rm{Re}} \int_{B(0,4)} \chi \big[ \overline{\hat \phi_j^+} w_{x_1}^++ \overline{\hat \phi_j^-} w_{x_1}^- \big]=0,
\text{ with }\hat \phi^\pm_j(z)=iw^\pm(z)\hat \psi^\pm(z+e_j ),
\\[2mm]
&\hat \psi^\pm \text{ satisfies the symmetry }  \eqref{s2}.
\end{aligned}
\right.
\end{align}
By the solvability  in  Proposition \ref{prop:linearfull} and the a priori estimate in Lemma \ref{prioriestimate2}, we can get  a solution $\hat \psi$  of \eqref{prop:linearfullforhatpsi} with the estimates
 \begin{align*}
\| \hat \psi \|_* &\leq C   \| \hat \rH\|_{**},
\\[2mm]
| \hat \psi |_\sharp &\leq C
( | \hat \rH  |_{\sharp\sharp} +\e |\log\e |^{\frac{1}{2}} \|\rH \|_{**} ),
\\[2mm]
 | \hat \rH  |_{\sharp\sharp}   & \le C \e |\log\e |^{\frac{1}{2}} \|\rH \|_{**}.
\end{align*}
 In fact, we have
  \[ \psi = \psi_s + \hat \psi, \]
  so we take \[  \psi_* : = \hat \psi.   \]
  From the analysis above,  $\psi_s$ and $\psi_*$ satisfy  the estimates in  this Lemma.
\end{proof}

 \medskip
Finally, we give the proof of Proposition~\ref{prop:sharp2b}.

 \begin{proof}[Proof of Proposition~\ref{prop:sharp2b}]
 By  combining the a priori estimates  in Lemma \ref{prioriestimate2} and Lemma \ref{decompose and priori}, and  the decompositions in Lemma \ref{decompose and priori}, we  can prove  Proposition \ref{prop:sharp2b}.  More  technical details can be found  in \cite{DDMR}.
 \end{proof}

\bigskip
\section{Solving the projected nonlinear problem} \label{section6}
In this section, we would  deal with the   projected problem \eqref{eq:linear5.1}.
The resolution theory is  provided in the following:
\begin{proposition}\label{proposition61}
There exists  a number  $D>0$, depending only on $\alpha, \sigma \in (0,1)$,  such that for any $\epsilon$  sufficiently small, problem \eqref{eq:linear5.1} has a solution $(\psi_\e, c)$  which satisfies
\begin{align*}
  \|\psi_\e\|_* \le  D \frac 1{|\ln \e|},
\end{align*}
and $\psi_\e$ is a continuous function of the parameter $\hat d$.
Furthermore, we have the decomposition
$$
\psi_\e = \psi_{\e o} + \psi_{\e e}= \psi_{\e o}^\alpha + \psi_{\e o}^\beta+ \psi_{\e e}
$$  and the estimates
\begin{align*}
| \psi^{\beta}_{\e o}|_\sharp +\|\psi^{\beta}_{\e o}\|_*   \le  D \e \sqrt{|\ln \e|},
 \qquad
| \psi_{\e o}^\alpha|_\sharp \le D \frac \e{ \sqrt{|\ln \e|}}.
\end{align*}
\end{proposition}

\begin{proof}
We consider the following closed, bounded  domain for functions with the form $\psi_\e = \psi^{\alpha}_{\e o}+\psi^{\beta}_{\e o} + \psi_{\e e}$:
\begin{align*}
 \mathcal{\tilde A}:=&\Bigg\{ \psi:\; \psi \text{ satisfies } \eqref{s4},
\quad
{\rm{Re}} \int_{B(0,4)} \chi  \big[ \overline{\phi_j^+} w_{x_1}^++ \overline{\phi_j^-} w_{x_1}^- \big]=0,    \;\;j=1,2, \;
 \\[2mm]
& \qquad   \;\|\psi\|_*\leq\frac{C}{|\ln \e|},
\quad
| \psi^{\alpha}_{\e o}|_\sharp   \le \frac{D \e} {\sqrt{|\ln \e|}},
\quad
| \psi^{\alpha}_{\e o}|_\sharp +\|\psi^{\beta}_{\e o}\|_*   \le  D \e \sqrt{|\ln \e|} \Bigg\}.
\end{align*}
From  Proposition \ref{prop:linearfull},
     the problem \eqref{eq:linear5.1} is equivalent to finding a fixed point of
\begin{equation}\nonumber
\psi=T_\e\Big(\rR+\mathcal N(\psi) \Big)=: \mathbb{G}_\e(\psi)
\end{equation}
 in the domain $\mathcal {\tilde A}$.
Now what we shall show is  that, for $\e$ small enough,   the map $\mathbb{G}_\e$ is a contraction from $\mathcal {\tilde A}$ to itself.
 We divide  this proof into   several steps.

\medskip
\noindent{\bf Step 1:}  We  claim  that \[ \|\mathbb{G}_\e(\psi)\|_* \le  \frac{D}{|\ln \e|}, \quad \text{for any}~\psi \in \mathcal {\tilde A}.    \]
The readers  can refer to \cite{DDMR} for the proof of this step.

\medskip
\noindent{\bf Step 2:}
We claim  that, for any $\psi \in \mathcal {\tilde A}$,  we have the decompositions  and estimates
\begin{align} \label{mathbbG6.1}
\mathbb{G}_\e(\psi) =  \mathbb{G}_{\e o}(\psi) +  \mathbb{G}_{\e e}(\psi) =  \mathbb{G}_{\e o}^\alpha(\psi) + \mathbb{G}_{\e o}^\beta(\psi) +  \mathbb{G}_{\e e}(\psi),
\end{align}
\begin{align} \label{mathbbG6.2}
|  \mathbb{G}_{\e o}^\alpha(\psi)|_{\sharp \sharp} \le  \frac{D \e} {\sqrt{|\ln \e|}}, \quad \| \mathbb{G}_{\e o}^\beta(\psi)\|_* \le D \e \sqrt{|\ln \e|}.
\end{align}
The proofs of \eqref{mathbbG6.1}-\eqref{mathbbG6.2} rely on the following two facts.

\smallskip
\medskip
\noindent$\clubsuit$
From   Proposition \ref{errorProp2}, we have
\begin{align*}
\rR=\rR_o+\rR_e =  \rR_o^\alpha + \rR_o^\beta +  \rR_e,
\end{align*}
with
\begin{align*}
|\rR_o^\alpha|_{\sharp \sharp} \le  \frac{D \e} {\sqrt{|\ln \e|}}, \quad \|  \rR_o^\beta \|_* \le D \e \sqrt{|\ln \e|}.
\end{align*}

\smallskip
\medskip
\noindent$\clubsuit$
 Next, we claim  that  we can decompose $ \mathcal N(\psi)   $   into the following three parts
 \begin{align}
 \mathcal N(\psi)    =  \big(\mathcal N(\psi)\big)_o^\alpha + \big( \mathcal N(\psi) \big)_o^\beta  +   \big(\mathcal N(\psi)\big)_e,
 \label{decompN}
 \end{align}
 with the estimates
 \begin{align}
 | \big(\mathcal N(\psi)\big)_o^\alpha|_{\sharp \sharp} \le  \frac{D \e} {\sqrt{|\ln \e|}},
\qquad
\|  \big( \mathcal N(\psi) \big)_o^\beta \|_{**} \le D \e \sqrt{|\ln \e|},\quad \forall\, \psi\in \mathcal {\tilde A}.
\label{decompNEstim}
 \end{align}
These can be verified in the following way. Recall that
 \[  \mathcal N (\psi) =  \mathcal N_1 (\psi)  + i\mathcal N_2(\psi),  \quad \text{with}\ \mathcal N_j (\psi)  = \big(\mathcal N_j^+(\psi), \mathcal N_j^- (\psi)  \big)\quad j=1,2. \]
  Without loss of generality, express  $ \mathcal N^{\pm}_1,  \mathcal N^{\pm}_2$  in the coordinates $(\ell_1, \theta_1) = (\ell, \theta)$
when  $|z- \tilde d| >3$ and  ${\rm{Re} (z) } >0$
 \begin{align}  \label{expressionNpm1}
 \mathcal N^{\pm}_1 &=  \,2(\p_\ell \psi^\pm_1)(\p_\ell \psi^\pm_2)+2(\p_\theta \psi^\pm_1)(\p_\theta \psi^\pm_2) \Big(\e^2+\frac{1}{\ell^2}  \Big)
 \nonumber
  \\[2mm]
  & \quad
  +\e^2\tilde d  \Big( \sin \theta\, \big[ \p_\ell \psi^\pm_1 \p_\theta \psi^\pm_2+\p_\theta \psi^\pm_1 \p_\ell \psi^\pm_2 \big]+\frac{2\cos \theta}{\ell}\p_\theta \psi^\pm_1 \p_\theta \psi^\pm_2 \Big)
  \nonumber
  \\[2mm]
  & \quad  +
  \e^2\tilde d^2\left( 2\sin^2\theta \p_\ell \psi^\pm_1 \p_\ell\psi^\pm_2
  +\frac{4\sin \theta \cos \theta}{\ell}\big[\p_\ell \psi^\pm_1\p_\theta \psi_2+\p_\theta\psi^\pm_1\p_\ell \psi^\pm_2 \big]\right.
  +
    \frac{2\cos^2\theta}{\ell^2}
    \p_\theta\psi^\pm_1 \p_\theta \psi^\pm_2 \Big),
    \end{align}
and
\begin{align}  \label{expressionNpm2}
 \mathcal N^{\pm}_2&= -(\p_\ell \psi^\pm_1)^2+(\p_\ell \psi^\pm_2)^2-\big(\e^2+\frac{1}{\ell^2} \big) \Big((\p_\theta \psi^\pm_2)^2-(\p_\theta \psi^\pm_1)^2\Big)    \nonumber
  \\[2mm]
  & \quad   + \e^2\tilde d\left(\sin \theta \big(\p_\ell \psi^\pm_1\p_\theta \psi^\pm_1+\p_\ell\psi^\pm_2\p_\theta \psi^\pm_2 \big)
   +\frac{\cos \theta}{\ell}\big[ (\p_\theta \psi^\pm_2)^2-(\p_\theta \psi^\pm_1)^2 \big] \right)  \nonumber
  \\[2mm]
  & \quad +\e^2\tilde d^2 \left(\sin^2\theta \big((\p_\ell \psi^\pm_2)^2-(\p_\ell \psi^\pm_1)^2\big)
   +\frac{4\cos \theta \sin \theta}{\ell}\big(\p_\ell \psi^\pm_1\p_\theta \psi^\pm_1+\p_\ell \psi^\pm_2\p_\theta \psi^\pm_2\big)\right.
  \nonumber
  \\[2mm]
  & \qquad\qquad
  \left.+\frac{\cos^2\theta}{\ell^2}\big[ (\p_\theta \psi^\pm_2)^2-(\p_\theta \psi^\pm_1)^2 \big] \right)
    \nonumber\\[2mm]
  & \quad+  A_\pm|v_d^\pm  |^2 \big(e^{-2\psi^\pm_2} - 1+ 2\psi^\pm_2\big)  +  B|v_d^\mp  |^2 \big(e^{-2\psi^\mp_2} - 1+ 2\psi^\mp_2\big).
\end{align}
We will use the approach  introduced in \cite{DDMR}.
Suppose that the real  scalar functions $f, g$ satisfy the following decompositions
\[
f = f_o + f_e = f_o^\alpha + f_o^\beta + f_e,  \qquad  g= g_o + g_e = g_o^\alpha + g_o^\beta + g_e,
\]
with estimates
\[
| f_o^\alpha |_{\sharp, 1}   + | g_o^\alpha |_{\sharp, 1}   \le  \frac{D \e} {\sqrt{|\ln \e|}},
\qquad
\|  f_o^\beta  \|_{*, 1}  +\|  g_o^\beta  \|_{*, 1}  \le D \e \sqrt{|\ln \e|},
\qquad
\|  f_e  \|_{*, 1}  + \|  g_e  \|_{*, 1}  \le  \frac D {|\ln \e|},
\]
where $ | \cdot |_{\sharp, 1} $ is the semi-norm defined in \eqref{sharp_1}-\eqref{normSharp2} for real scalar functions  and $\| \cdot\|_{*, 1}$  is the norm defined in \eqref{normstar1} for real scalar functions.
Now we express the product of two real scalar  functions $f, g$ as the following
  \begin{align}
  fg  & =( f_o^\alpha + f_o^\beta + f_e)(g_o^\alpha + g_o^\beta + g_e)
  \nonumber
  \\[2mm]
  &   =  f_e g_e +  f_o g_o +  f_o^\alpha g_e + f_o^\beta g_e + f_e g_o^\alpha + f_e g_o^\beta
\nonumber\\[2mm]
& =: (fg)_e + (fg)^\alpha_{o}+(fg)^\beta_o ,
\label{productdecom}
  \end{align}
where
  \begin{align}  \label{expressionoffg}
    (fg)_e = f_e g_e, \qquad (fg)_o^\alpha = f_o g_o +  f_o^\alpha g_e +  f_e g_o^\alpha, \qquad   (fg)_o^\beta = f_o^\beta g_e+  f_e g_o^\beta.
  \end{align}
Using the expressions  of  $ \mathcal N^{\pm}_1,  \mathcal N^{\pm}_2$ in \eqref{expressionNpm1}-\eqref{expressionNpm2} and the decompositions in \eqref{productdecom}-\eqref{expressionoffg},
when $|z-\tilde d| >3$, we have
  \begin{align*}
   \mathcal N^{\pm}_j =    \big(   \mathcal N^{\pm}_j \big)_e +    \big(   \mathcal N^{\pm}_j\big)_o^\alpha +   \big(   \mathcal N^{\pm}_j \big)_o^\beta, \quad  j =1,2,
  \end{align*}
   and
   \begin{align*}
   \frac{\big(   \mathcal N^{\pm}_1\big)_o^\alpha} {\ell_1^{-1}}  \le C \Big(  | \ln \e|\,  \| \psi\|_* \,  | \psi_o^\alpha |_\sharp+| \psi_o|^2_\sharp +\frac{D\e}{\sqrt{|\ln \e|}}   \| \psi \|_*^2\Big),
   \end{align*}
   \begin{align*}
  \frac{ \big(   \mathcal N^{\pm}_1\big)_o^\beta} {\ell_1^{-1}}  \le C  \Big(\| \psi_o^\beta\|_* \| \psi\|_* +| \psi_o|^2_\sharp +\frac{D\e}{\sqrt{|\ln \e|}}   \| \psi \|_*^2\Big),
   \end{align*}
    \begin{align*}
   \frac{\big(   \mathcal N^{\pm}_2\big)_o^\alpha} {\ell_1^{-1+\sigma}}  \le C \Big(  | \ln \e|\,  \| \psi\|_* \,  | \psi_o^\alpha |_\sharp+| \psi_o|^2_\sharp +\frac{D\e}{\sqrt{|\ln \e|}}   \| \psi \|_*^2\Big),
   \end{align*}
   \begin{align*}
  \frac{ \big(   \mathcal N^{\pm}_2\big)_o^\beta} {\ell_1^{-1+\sigma}}  \le C  \Big(\| \psi_o^\beta\|_* \| \psi\|_*  +| \psi_o|^2_\sharp +\frac{D\e}{\sqrt{|\ln \e|}}   \| \psi \|_*^2\Big).
   \end{align*}
On the  other hand, we proceed in a similar way to estimate the nonlinear term $ \mathcal N (\psi)$    when   $|z-\tilde d| <2$ and get
\begin{align*}
  \big| i v_d^\pm   \big(   \mathcal N^{\pm}_i\big)_o^\alpha \big| \le  C \Big(  | \ln \e|\,  \| \psi\|_* \,  | \psi_o^\alpha |_\sharp+| \psi_o|^2_\sharp +\frac{D\e}{\sqrt{|\ln \e|}}   \| \psi \|_*^2\Big),
  \end{align*}
\begin{align*}
  \big| i v_d^\pm   \big(   \mathcal N^{\pm}_i\big)_o^\beta \big| \le  C  \Big(\| \psi_o^\beta\|_* \| \psi\|_*   +| \psi_o|^2_\sharp +\frac{D\e}{\sqrt{|\ln \e|}}   \| \psi \|_*^2\Big).
  \end{align*}
Combining  the analysis  and estimates above, we conclude that
  \begin{align*}
\big| \mathcal{N}(\psi)_{o}^\kappa\big|_{\sharp\sharp} &\leq C  \Big(\| \psi_o^\beta\|_* \| \psi\|_* +  | \ln \e|\,  \| \psi\|_* \,  | \psi_o^\alpha |_\sharp+| \psi_o|^2_\sharp +\frac{D\e}{\sqrt{|\ln \e|}}   \| \psi \|_*^2\Big)
\nonumber
  \\[2mm]
  &  \le \frac{D\e}{\sqrt{|\ln \e|}}
 \quad \text{for}~ \kappa=\alpha, \beta.
\end{align*}
Thus, the verifications of \eqref{decompN}-\eqref{decompNEstim} are completed.

\medskip
Taking
 \[ \rH_o^\kappa = \rR_o^\kappa+  \big(\mathcal{N}(\psi)\big)_{o}^k, \quad \text{for}~\kappa =\alpha, \beta.  \]
 in Proposition \ref{prop:sharp2b}, we then conclude that \eqref{mathbbG6.1} and \eqref{mathbbG6.2} hold.
%
%

   \medskip
\noindent{\bf Step 3:}   At last, we will show that the map $\mathbb{G}_\e$  is a contraction when $\e$ small enough.
   This step is standard and we omit it here.
%
  Readers can refer to \cite{DDMR} for more details.

   \end{proof}

 \medskip
 \section{Solve the reduced problem}    \label{section77}
  Note that the parameter ${\hat d} $ will determine the locations  of vortices.
  As a standard step in the reduction method to make the  Lagrange multiplier $c$ vanish in \eqref{eq:linear5.1},
  we will set up  an equation involving $\hat d$ in such a way that $c=0$.

\medskip
 Without loss of generality,
 we consider the following equivalent problem of \eqref{eq:linear5.1}
  \begin{align} \label{equation8.1}
 i w^\pm(z) \Big[  \LLL^\pm_\e(\psi) +\rR^\pm+\mathcal N^\pm(\psi) \Big](z+{\tilde d}_j)  =   c\chi w^\pm_{x_1}, \quad \text{for}~z\in   B_{\mathbb{R}_\e}(0).
\end{align}
Integrating  the equations  \eqref{equation8.1}  against $w^\pm_{x_1}$ respectively, we get
 \begin{align*}
 &\sum_{\kappa \in\{+,-\}}  {\rm Re}\int_{B_{\mathbb{R}_\e}(0)}   \Bigg( i\overline {w^\kappa_{x_1}} (z) w^\kappa (z)\Big[  \LLL^\kappa_\e(\psi) +\rR^\kappa+\mathcal N^\kappa(\psi) \Big] (z+{\tilde d})   \Bigg)   \nonumber
\\[3mm]
&
 \qquad\qquad=   c  \int_{B_{\mathbb{R}_\e}(0)} \chi \big( |w^+_{x_1}|^2+ |w^-_{x_1}|^2\big)
 =   cc^*.
\end{align*}
 Since $c^* \sim C$,  we know easily that  $c=0$ is equivalent to the following relation
  \begin{align}  \label{balancecondition}
 &\sum_{\kappa \in\{+,-\}}  {\rm Re}\int_{B_{\mathbb{R}_\e}(0)}   \Bigg( i\overline {w^\kappa_{x_1}} (z) w^\kappa (z)\Big[  \LLL^\kappa_\e(\psi) +\rR^\kappa+\mathcal N^\kappa(\psi) \Big] (z+{\tilde d})   \Bigg)   =0.
 \end{align}
The computations of all terms in \eqref{balancecondition} will be carried out in the sequel.

\medskip
 We recall the following results introduced  in  Lemma \ref{Lepsilon}
 \begin{align*}
 &  i w^\pm(z)  \LLL^\pm_\e(\psi) (z+{\tilde d})
:= {L_{0,\pm}}  (\Phi_1)   + \mathscr F_{1,\pm}(z)
    \end{align*}
where     ${L_0}  = (L_{0,+}, L_{0,-})$  is the linearization of Ginzburg-Landau equation \eqref{2dgl} around $w$, and
\begin{align}  \label{mathscrF1}
 \mathscr F_{1,\pm}(z)    =& 2A_\pm \big(1-|\hat\Omega^\pm_1|^2 \big) {\rm{Re}}\Big(w^\pm(z)\overline{ \hat\Phi^\pm_1}\Big)  w^\pm(z)
 \ +\
 2B\big(1-|\hat\Omega^\mp_1|^2 \big) {\rm{Re}}\Big(w^\mp(z) \overline{\hat \Phi^\mp_1}\Big)   w^\pm(z)
     \nonumber
\\[2mm]
& +\e^2 \Big[  \p^2_{ss}  \hat\Phi^\pm_j  -4i   \p_s \hat \Phi^\pm_1  -4 \hat\Phi^\pm_1  \Big]
   \ +\
   \frac1{\Omega_1^\pm} \Big[ 2 \nabla \Phi^\pm_1 \nabla  \hat\Omega^\pm_1 +2 \e^2 \p_s\hat\Omega^\pm_1\p_s \hat\Phi^\pm_1 \Big]
      \nonumber
\\[2mm]
&+   \Bigg\{  4 \e^2 -  \frac{2\nabla  w^\pm(z)\nabla \hat\Omega^\pm_1   + 2\e^2 \p_s w^\pm(z) \p_s\hat\Omega^\pm_1 } {v_d^\pm}
- \e^2\frac{\Big[ \p_{ss} w^\pm(z) - 4i \p_sw^\pm(z)\Big]  } {w^\pm(z)} \Bigg\}  \hat \Phi_1^\pm
\nonumber
\\[2mm]
& +  \eta_1\frac{\mathbb{E}^\pm} {v_d^\pm}\hat\Phi_1^\pm,
 \end{align}
 and
 \begin{align*}
  \Phi_1^\pm (z) = iw^\pm(z) \psi^\pm (z+\tilde d),  \quad   \hat  \Omega_1^\pm   = \Omega_1^\pm (z+2\tilde d) = w^\pm(z+2\tilde d), \quad \text{for}~ z\in B_{\mathbb{R}_\e}(0).
 \end{align*}
 We first consider
\begin{align*}
    &\sum_{\kappa \in\{+,-\}} {\rm Re}\int_{B_{\mathbb{R}_\e}(0)}    i \overline {w^\kappa_{x_1}} (z) w^\kappa(z)  \LLL^\kappa_\e(\psi) (z+\tilde d)
     \nonumber
\\[2mm]
  &\qquad\qquad = \sum_{\kappa \in\{+,-\}} {\rm Re}\int_{B_{\mathbb{R}_\e}(0)}      \overline {w^\kappa_{x_1}} (z)   \Big[ {L_{0, \kappa} }  (\Phi_1)   + iw^\kappa(z) \mathscr F_{1,\kappa}(z)
\Big]
     \nonumber
\\[2mm]
& \qquad\qquad= \sum_{\kappa \in\{+,-\}} {\rm Re}\int_{B_{\mathbb{R}_\e}(0)}    \overline {\Phi^\kappa_1} (z)    L_{0, \kappa}    (w_{x_1})
\ +\
\sum_{\kappa \in\{+,-\}} {\rm Re}\int_{\p{B_{\mathbb{R}_\e}(0)}}  \Bigg[\frac{\p \Phi^\kappa_1} {\p\nu}   \overline{w^\kappa_{x_1}} -  \overline{\Phi^\kappa_1} \frac{\p{w^\kappa_{x_1}}} {\p\nu} \Bigg]
   \nonumber
\\[2mm]
  &  \qquad\qquad\qquad + \sum_{\kappa \in\{+,-\}}  {\rm Re}\int_{B_{\mathbb{R}_\e}(0)}  i   \overline {w^\kappa_{x_1}} (z)  w^\kappa(z) \mathscr F_{1,\kappa}(z)
     \nonumber
\\[2mm]
&\qquad\qquad: = \mathscr I_1+ \mathscr I_2 + \mathscr I_3.
   \end{align*}
   We easily get $\mathscr I_1=0$ since $L_0(w_{x_1}) =0$.
For $\mathscr I_2$,  using the decay  properties of $W^\pm$ in Lemma \ref{lemmaofW} and the definition of $\|\cdot\|_{*}$, we  get
\begin{align*}
\mathscr I_2  \le C \e |\ln\e|^{\frac12}  \| \psi\|_* =  O\Big(\frac\e{|\ln\e|^{\frac12}}  \Big).
\end{align*}
Using the expression  of $ \mathscr F_{1,\pm}(z)$ in \eqref{mathscrF1},  by some tedious but straightforward   calculations, we  can get
\begin{align*}
\mathscr I_3 =  O\Big(\frac\e{|\ln\e|^{\frac12}}  \Big).
\end{align*}
Therefore, we have
\begin{align}
    \sum_{\kappa \in\{+,-\}} {\rm Re}\int_{B_{\mathbb{R}_\e}(0)}    i \overline {w^\kappa_{x_1}} (z) w^\kappa(z)  \LLL^\kappa_\e(\psi) (z+\tilde d)
    = O\Big(\frac\e{|\ln\e|^{\frac12}}  \Big).
\label{reduced1}
\end{align}

 Next, we consider the third term of \eqref{balancecondition}.  We decompose locally
 \begin{align*}
  \mathcal N^\pm(\psi)  & = \mathcal N_e^{\pm}(\psi)+\mathcal N_o^{\pm}(\psi)
  \nonumber
\\[2mm]
  &  =    \big(\mathcal N^\pm_1\big)_o  + i  \big(\mathcal N^\pm_2\big)_o  + \big(\mathcal N^\pm_1\big)_e  + i  \big(\mathcal N^\pm_2\big)_e
  \end{align*}
where $\big(\mathcal N^\pm_j\big)_o$ and  $\big(\mathcal N^\pm_j\big)_e$ for $j=1,2$  are  defined in Section \ref{section6}.
Then
 \begin{align*}
 &\sum_{\kappa \in\{+,-\}}  {\rm Re}\int_{B_{\mathbb{R}_\e}(0)}   \Bigg( i\overline {w^\kappa_{x_1}} (z) w^\kappa (z) \mathcal N^\kappa(\psi)   (z+{\tilde d})   \Bigg)
   \nonumber
\\[2mm]
  &  = \sum_{\kappa \in\{+,-\}}  {\rm Re}\int_{B_{\mathbb{R}_\e}(0)}   \Bigg( i\overline {w^\kappa_{x_1}} (z) w^\kappa (z) \mathcal N_o^{\kappa}(\psi)   (z+{\tilde d})   \Bigg)
   \nonumber
\\[2mm]
  &  = \sum_{\kappa \in\{+,-\}}  {\rm Re}\int_{B_{\mathbb{R}_\e}(0)}   i W^\kappa  \Bigg({W^\kappa}'\cos \theta+ i\frac{\sin\theta} r W^\kappa  \Bigg)      \mathcal N_o^{\kappa}(\psi)   (z+{\tilde d})
   \nonumber
\\[2mm]
  &  = \sum_{\kappa \in\{+,-\}}  \int_{B_{\mathbb{R}_\e}(0)}    W^\kappa \Bigg( -\frac{\sin\theta} r W^\kappa   \big( \mathcal N_1^{\kappa}(\psi) \big)_o  (z+{\tilde d}) - {W^\kappa}'\cos \theta \big(\mathcal N_2^{\kappa}(\psi) \big)_o  (z+{\tilde d})\Bigg)
 \end{align*}
 where  $z=re^{i\theta}$ and the second equality hold due to orthogonality.
For the estimates of the terms $\big(\mathcal N^{\pm}_1\big)_o$ and $\big(\mathcal N^{\pm}_2\big)_o$,
 from \eqref{expressionNpm1}-\eqref{expressionoffg},  we have
 \begin{align*}
|  \big(\mathcal N^{\pm}_1\big)_o   (z+{\tilde d}) |  &\le  \frac C{1+r^{2-\sigma}}    \Big( \| \psi_o^\beta\|_* \| \psi\|_* +  | \ln \e|\,  \| \psi\|_* \,  | \psi_o^\alpha |_\sharp+| \psi_o|^2_\sharp \Big)
\nonumber
\\[2mm]    & \le   \frac {C \e } { (1+r^{2-\sigma})  |\ln \e|^{\frac 12}},
 \end{align*}
 and
  \begin{align*}
| \big(\mathcal N^{\pm}_2\big)_o  (z+{\tilde d}) |  &\le   C  \Big( \| \psi_o^\beta\|_* \| \psi\|_* +  | \ln \e|\,  \| \psi\|_* \,  | \psi_o^\alpha |_\sharp+| \psi_o|^2_\sharp \Big)
  \le   \frac {C \e } {   |\ln \e|^{\frac 12}}.
 \end{align*}
 Therefore, we have
   \begin{align}
 &\sum_{\kappa \in\{+,-\}}  {\rm Re}\int_{B_{\mathbb{R}_\e}(0)}   \Bigg( i\overline {w^\kappa_{x_1}} (z) w^\kappa (z) \mathcal N^\kappa(\psi)   (z+{\tilde d})   \Bigg)   =  O\Big(\frac\e{|\ln\e|^{\frac12}}  \Big).
 \label{reduced2}
 \end{align}

\medskip
 At last,  we  consider
    \begin{align*}
 &\sum_{\kappa \in\{+,-\}}  {\rm Re}\int_{B_{\mathbb{R}_\e}(0)}   \Bigg( i\overline {w^\kappa_{x_1}} (z) w^\kappa (z)\rR^\kappa  (z+{\tilde d})   \Bigg)
    \nonumber
\\[2mm]
  &  =   \sum_{\kappa \in\{+,-\}}  \sum_{j\in\{0,1\}} {\rm Re}\int_{B_{\mathbb{R}_\e}(0)}   \Bigg( i\overline {w^\kappa_{x_1}} (z) w^\kappa (z)\rR^\kappa_j  (z+{\tilde d})   \Bigg)
      \nonumber
\\[2mm]
  &  :=  \mathscr T_0 + \mathscr T_1,
 \end{align*}
 where
\[
\mathscr T_j =\sum_{\kappa \in\{+,-\}}    {\rm Re}\int_{B_{\mathbb{R}_\e}(0)}   \Bigg( i\overline {w^\kappa_{x_1}} (z) w^\kappa (z)\rR^\kappa_j  (z+{\tilde d})   \Bigg),
\qquad
S^\pm_j(v_d) =  i  v_d^\pm \rR_j^\pm, ~\text{for}~j =0,1.
\]
For $\mathscr T_1 $,  by  Lemma  \ref{lemma5.2}, we have
\begin{align*}
 \mathscr T_1   &= \frac{\e \hat d} {|\ln \e|^{\frac12}} \sum_{\kappa \in\{+,-\}}   (t^\kappa)^{-1} \int_{B_{\mathbb{R}_\e}(0)}  |w^\kappa_{x_1}|^2
\ +\
O\Big(\frac\e{|\ln\e|^{\frac12}}  \Big)
    \nonumber
\\[2mm]
  &   = \hat d \e |\ln \e|^{\frac12}  \frac 1{|\ln \e|} \sum_{\kappa \in\{+,-\}}   (t^\kappa)^{-1}  \int_0^{2\pi} \int_0^{\frac1\e} \frac{(W^\pm)^2\sin^2\theta} {r}  \, {\mathrm d} r{\mathrm d}\theta
\ +\
O\Big(\frac\e{|\ln\e|^{\frac12}}  \Big)
      \nonumber
\\[2mm]
  &   = (t^+ + t^-)  \hat d \e |\ln \e|^{\frac12}  \pi \big( 1+o_\e(1)\big).
 \end{align*}
Here we have used the facts
$$
W^\pm = t^\pm+O\Big(\frac{1} {r^2}\Big),\quad \mbox{as } r\rightarrow +\infty.
$$
On the other hand,  for  $\mathscr T_0 $,  by the expression of $S_0^\pm(v_d)$  in \eqref{S0vdpm},
we have
\begin{align*}
  \mathscr T_0  & =2\sum_{\kappa \in\{+,-\}}   (t^\kappa)^{-1}    {\rm Re}\int_{\{\ell_1 \le\mathbb{R}_\e\}}
  \Bigg[ \frac{    \nabla  w^\kappa_a   \nabla w^\kappa_b } {w^\kappa_b } (z+{\tilde d})    \Bigg]\overline{w^\kappa_{x_1}}
\ +\
O\Big(\frac\e{|\ln\e|^{\frac12}}  \Big)
     \nonumber
\\[2mm]
  &   =   2\sum_{\kappa \in\{+,-\}}  (t^\kappa)^{-1}   {\rm Re} \int_{\{\ell_1 \le\mathbb{R}_\e\}}
  \Bigg[\frac{\p_{x_1} w^\kappa_a \p_{x_1} w^\kappa_b+ \p_{x_2} w^\kappa_a \p_{x_2} w^\kappa_b} {w^\kappa_b}  \Bigg] \overline{w^\kappa_{a, x_1}}
\ +\
O\Big(\frac\e{|\ln\e|^{\frac12}}  \Big).
  \end{align*}
  Here we have used the notation   \[  w^\pm(z-\tilde d)     = w^\pm_a (z), \quad w^\pm(z+\tilde d)     = w^\pm_b (z).   \]
  Similar to the calculations in \cite{DDMR},  by  the properties of $w^\pm$  in  Lemma \ref{lemmaofW},  we can easily get  that
  \begin{align*}
  2\sum_{\kappa \in\{+,-\}}   {\rm Re}\int_{\{\ell_1 \le\mathbb{R}_\e\} }   \frac{\p_{x_1} w^\kappa_a \p_{x_1} w^\kappa_b} {w^\kappa_b} \overline{w^\kappa_{x_1}}   = O\Big(\e^2 |\ln \e|\Big).
  \end{align*}
Finally  we consider the term
 \begin{align*}
  &2\sum_{\kappa \in\{+,-\}}   {\rm Re}\int_{\{\ell_1 \le\mathbb{R}_\e\} }   \frac{\p_{x_2} w^\kappa_a \p_{x_2} w^\kappa_b} {w^\kappa_b} \overline{w^\kappa_{x_1}}
      \nonumber
\\[2mm]
  &   = -  2\sum_{\kappa \in\{+,-\}}   \int_{\{\ell_1 \le\mathbb{R}_\e\} }   \frac{W^\pm(\ell_1)\, {W^\pm}'(\ell_1)} {\ell_2} \cos{\theta_2} \, {\mathrm d} \ell_1{\mathrm d}\theta_1 + O\Big(\e^2 |\ln \e|\Big).
 \end{align*}
Note that
\begin{align*}
\frac {\cos\theta_2} {\ell_2}  = \frac {1} {2\tilde d} \big( 1+ o_\e(1)\big),
\end{align*}
and
\begin{align*}
 2\int_{\{\ell_1 \le\mathbb{R}_\e\} }  W^\pm(\ell_1)\, {W^\pm}'(\ell_1)\, {\mathrm d} \ell_1{\mathrm d}\theta_1
 & = 2\pi\times\frac12 \Big(|W^\pm(\mathbb{R}_\e)|^2-|W^\pm(0)|^2\Big)
   \nonumber
\\[2mm]
  &   =  \pi\big({t^\pm}^2+o_\e(1)\big).
\end{align*}
The above estimates will give
\begin{align*}
\mathscr T_0
= - (t^+ + t^-) \frac {\e |\ln \e|^\frac12\pi} {\hat d} \big( 1+ o_\e(1)\big).
\end{align*}
Therefore, by collecting the estimates of $\mathscr T_0$ and $\mathscr T_1$, we obtain
\begin{align}
\sum_{\kappa \in\{+,-\}}  {\rm Re}\int_{B_{\mathbb{R}_\e}(0)}   \Bigg( i\overline {w^\kappa_{x_1}} (z) w^\kappa (z)\rR^\kappa  (z+{\tilde d})   \Bigg)
=(t^+ + t^-)\Big[\, \pi \hat d \e |\ln \e|^{\frac12} -\frac {\e |\ln \e|^\frac12\pi} {\hat d} \, \Big]\big( 1+ o_\e(1)\big).
\label{reduced3}
\end{align}

\medskip
In conclusion, the substitutions of \eqref{reduced1}-\eqref{reduced3} into \eqref{balancecondition} will imply that: $c=0$ if and only if
\begin{align}  \label{reduced for d}
  \Big[\, \pi \hat d \e |\ln \e|^{\frac12} -\frac {\e |\ln \e|^\frac12\pi} {\hat d} \, \Big]\big( 1+ o_\e(1)\big)    =0,
\end{align}
where  the functions in $ o_\e(1)$ are continuous with respect to the parameter $\hat d$ since $\psi_\e$ is continuous  on $\hat d$.
Therefore we  obtain  that there exists $\hat d = 1+ o(1) $ such that \eqref{reduced for d} holds.
By combining the arguments in Sections \ref{section2}-\ref{section77}, we complete  the proof of Theorem \ref{theorem1}.


%
%
\bigskip
\section{Appendix}
\subsection{The linear operator of Ginzburg-Landau system in $\R^2$ around the standard vortex $w$}

\medskip
The following non-degeneracy results  will be used in the proof of Lemma \ref{FirstEstimate} and Lemma \ref{prioriestimate2}. We first recall the operator $L_0 (\phi) = \big(L_0^+, L_0^-\big)(\phi) $, with
\begin{align*}
L_0^\pm  (\phi)   &=  \Delta \phi^\pm + \Big[ A_\pm \big({t^\pm}^2-{W^\pm}^2 \big)+B\big({t^\mp}^2- {W^\mp}^2\big)\Big]\phi^\pm  \nonumber
\\[2mm]
&\quad-2A_\pm {\rm{Re}}\Big(w^\pm\overline{\phi^\pm}\Big)w^\pm\,-\,2B {\rm{Re}}\Big(w^\mp \overline{\phi^\mp}\Big)w^\pm.
\end{align*}

\begin{lemma}\label{lem:ellipticestimatesL0}
 For any  $\phi = (\phi^+, \phi^-) \in L^\infty(\R^2, \mathbb{C}^2)$ satisfying the symmetry  \(\phi(\bar{z})=\overline{\phi(z)}\),
suppose  that
\begin{eqnarray} \label{estiamteforpsi1}
|\psi_1|+(1+|z|)|\nabla \psi_1| \leq  C, \qquad |\psi_2|+|\nabla \psi_2| \leq  \frac{C}{1+|z|} ,\quad \text{for}\ \ |z|>1,
\end{eqnarray}
with $\phi^\pm=iw^\pm\psi^\pm$ and $\psi^\pm=\psi_1^\pm+i\psi_2^\pm$ with $\psi_1^\pm,\psi_2^\pm\in \R$.
Then $L_0(\phi)=0$   if and only if
\begin{equation}\label{no-degeneracy}
\phi=c_1 w_{x_1}
\end{equation}
for some real constant \(c_1\).
\end{lemma}
\begin{proof}

 Taking $\phi = (iw^+\psi^+,  iw^-\psi^-) $ into  $ L_0 (\phi) =0$, we can get
\begin{align*}
\Delta \psi^\pm+2\frac{\nabla w^\pm}{w^\pm} \nabla \psi^\pm-2i A_\pm|w^\pm|^2 \psi^\pm_2 -2i B |w^\mp|^2\psi^\mp_2 =0, \quad ~\text{ in } B(0,1)^c.
\end{align*}
This reads in the coordinates $(r, \theta)$
\begin{align}
0 & = \Delta \psi_1^\pm  + \frac{2{W^\pm}'}{W^\pm}\p_r \psi_1^\pm -\frac2{r^2} \p_\theta\psi_2^\pm \quad \text{ in } B(0,1)^c,  \label{psi1}
\\
0 & = \Delta \psi_2^\pm+\frac{2{W^\pm}'}{W^\pm}\p_r \psi_2^\pm  +\frac2{r^2} \p_\theta\psi_1^\pm-2A_\pm (W^\pm)^2 \psi^\pm_2 -2 B (W^\mp)^2 \psi^\mp_2
\quad \text{ in } B(0,1)^c. \label{psi2}
\end{align}
  Rewrite \eqref{psi2} as
  \begin{align*}
  0 & = \Delta \psi_2^\pm+\frac{2{W^\pm}'}{W^\pm}\p_r \psi_2^\pm  +\frac2{r^2} \p_\theta\psi_1^\pm-2A_\pm {t^\pm}^2 \psi^\pm_2 -2 B {t^\mp}^2\psi^\mp_2     \nonumber
\\[2mm]
& \quad  -2A_\pm\big(  (W^\pm)^2-{t^\pm}^2\big)  \psi^\pm_2 -2 B\big( (W^\mp)^2-{t^\mp}^2\big) \psi^\mp_2
\quad   \text{ in } B(0,1)^c.
  \end{align*}
Recall \eqref{diagonal}
\[
{\bf C}^T {\bf C} = {\bf I}, \quad  2{\bf C} \, \mathbb{M}\,  {\bf C}^T =  \text{diag}(     \lambda^+,      \lambda^-) \quad \text{with}~\lambda^\pm>0,
\]
    and  define new  vector functions
    \begin{equation*}
\tilde \psi_j=(\tilde\psi_j^+, \tilde \psi_j^-)^T: ={\bf C}^T      \, ( \psi_j^+,  \psi_j^-)^T,   \quad \text{for}~j=1,2,
  \end{equation*}
  \begin{align*}
  \mathscr W_1  = (  \mathscr W_1^+,   \mathscr W_1^-)  =    {\bf C}^T  \Big(\frac{2{W^+}'}{W^+}\p_r \psi_1^+,  \frac{2{W^-}'}{W^-}\p_r \psi_1^-\Big),
  \end{align*}
  \begin{align*}
  \mathscr W_2  = (  \mathscr W_2^+,   \mathscr W_2^-)  =    {\bf C}^T  \Big(\frac{2{W^+}'}{W^+}\p_r \psi_2^+,  \frac{2{W^-}'}{W^-}\p_r \psi_2^-\Big),
  \end{align*}
  \begin{align*}
  \mathscr A= (  \mathscr A^+,  \mathscr A^-) : = -2{\bf C}^T   \Big(A_+\big(  (W^+)^2-{t^+}^2\big)  \psi^+_2,  A_-\big(  (W^-)^2-{t^-}^2\big)  \psi^-_2\Big),
  \end{align*}
and
   \begin{align*}
  \mathscr B= (  \mathscr B^+,  \mathscr B^-) : = -2{\bf C}^T   \Big( B\big((W^-)^2-{t^-}^2\big)  \psi^-_2,  B\big(  (W^+)^2-{t^+}^2\big)  \psi^+_2\Big).
  \end{align*}
Then the system \eqref{psi1}-\eqref{psi2} can be transformed to
     \begin{align}
0 & = \Delta \tilde \psi_1  +    \mathscr W_1 +\frac2{r^2} \p_\theta \tilde \psi_2 \quad \text{ in } B(0,1)^c,  \label{psi11}
\\
0 & = \Delta \tilde \psi_2 +\frac2{r^2} \p_\theta\tilde \psi_1  -\text{diag}(     \lambda^+,      \lambda^-)  \tilde \psi_2+   \mathscr W_2+ \mathscr A+ \mathscr B   \quad \text{ in } B(0,1)^c. \label{psi22}
\end{align}
 Using the estimate \eqref{estiamteforpsi1}  and the asymptotic for $W^\pm$, we have
\begin{align*}
  \Big|   -\frac2{r^2} \p_\theta\tilde \psi_1   +     \mathscr W_2+ \mathscr A+ \mathscr B \Big| \le  \frac C{1+r^2}  \quad \text{ in } B(0,1)^c.
\end{align*}
     Then by a barrier argument and elliptic estimates on  \eqref{psi22}, we can get
\begin{align}\label{eq:est_psi_2_decay}
 |\tilde \psi_2|+|\nabla \tilde \psi_2| \le \frac C{1+r^2},
\end{align}
which yields that
\begin{align*}
      \Big|  \mathscr W_1 -\frac2{r^2} \p_\theta \tilde \psi_2\Big|  \le \frac C{1+r^3}.
\end{align*}
On the other hand, considering  the equation \eqref{psi11},  together with the barrier argument and elliptic estimates, we obtain
     \begin{equation*}
|\tilde \psi_1|+(1+|z|)|\nabla \tilde \psi_1| \leq \frac{C}{(1+|z|)^\alpha},
\end{equation*}
for any $\alpha \in (0,1)$.  Since ${\bf C}$ is invertible, we can establish the estimates  for $\psi_1, \psi_2$
    \begin{equation} \label{decay-psi12}
| \psi_1|+(1+|z|)|\nabla  \psi_1| \leq \frac{C}{(1+|z|)^\alpha},\quad  | \psi_2|+|\nabla  \psi_2| \le \frac C{1+r^2}.
\end{equation}
By  \eqref{decay-psi12} and the decay of $W^\pm, {W^\pm}'$,  we conclude
\begin{align*}
   & \int_{\R^2}  (    |\nn \phi^+|^2 +  |\nn \phi^-|^2)
   \ +\
\int_{\R^2}  \Big[ A_+ \big({t^+}^2-{W^+}^2 \big)-B\big({t^-}^2- {W^-}^2\big)\Big]|\phi^+|^2
  \nonumber
\\[2mm]
& \quad
\ +\    \int_{\R^2}  \Big[ A_+ \big({t^-}^2-{W^-}^2 \big)-B\big({t^+}^2- {W^+}^2\big)\Big] |\phi^-|^2  < + \infty.
   \end{align*}
   Then Lemma \ref{nodegeneracy-theorem1.1}  implies that
   \[ \phi = c_1 w_{x_1}+ c_2 w_{x_2}, \quad \text{for some constant} ~c_1, c_2.  \]
   Using the symmetry $ \phi(\overline{z}) = \overline {\phi(z)} $, we conclude that \eqref{no-degeneracy} holds.

 \end{proof}

\begin{lemma}
\label{lem:ellipticestimatesL0-b}
For any $\phi \in L_{\text{loc}}^\infty(\R^2)$ satisfying  the symmetry  \(\phi(\bar{z})=\overline{\phi(z)}\),
suppose  that
\begin{align*}
|\psi_1|+(1+|z|)|\nabla \psi_1| \leq  C(1+|z|)^\alpha, \qquad |\psi_2|+|\nabla \psi_2| \leq  \frac{C}{1+|z|} ,
\quad\text{for}\  |z|>1,
\end{align*}
for some $\alpha<3$, where $\phi=(iw^+\psi^+, iw^-\psi^-)$ and $\psi^\pm=\psi_1^\pm+i\psi_2^\pm$ with $\psi_1^\pm,\psi_2^\pm\in \R$.
Then $L_0(\phi)=0$    if and only if
\begin{equation*}
\phi=c_1 w_{x_1}
\end{equation*}
for some real constant \(c_1\).
\end{lemma}
\begin{proof}
Using Lemma \ref{lem:ellipticestimatesL0}, some Fourier analysis  and ODE theory, we can prove this lemma.    The proof is similar as that in Lemma 7.3  in \cite{DDMR}.    We omit it here for conciseness.

 \end{proof}

\medskip
 \subsection{Elliptic estimates used in the linear theory}

Recall the polar coordinate notation $z = r e^{i s}$, $r>0$, $s\in\R$.
For the convenience of the reader, we   provide some useful lemmas   from \cite{DDMR}.

\begin{lemma}\label{lem:comparison_principle}
(\cite{DDMR}).
Let $u:\R \times \R^*_+ \rightarrow \R$ be a bounded function which is in $C^2(\R \times \R^*_+)\cap C^0(\overline{\R \times \R^*_+})$ and satisfy
\begin{equation*}
\left\{
\begin{array}{rcll}
\Delta u+\e^2 \p_{ss}^2 u & \geq& 0 \text{ in } \R \times \R^*_+, \\
u&\leq & 0 \text{ on } \R \times \{0 \} .
\end{array}
\right.
\end{equation*}
Then $u\leq 0$ in $\R \times \R^*_+$.
\qed
\end{lemma}

\medskip
\begin{lemma}\label{lem:comparison_principle_Neumann}
 (\cite{DDMR}).
Let \(u:\R\times \R^*_+ \rightarrow \R\) be a bounded function which is in \(C^2(\R\times\R^*_+)\cap C^1(\overline{\R\times\R^*_+})\) and  \(c\geq 0 \). If \(u\) satisfies
\begin{equation*}
\left\{
\begin{array}{rcll}
\Delta u+\e^2 \p^2_{ss}u-c u & \geq & 0 \text{ in } & \R \times \R^*_+, \\
\p_\nu u &\leq & 0 \text{ on } & \R\times \{0 \},
\end{array}
\right.
\end{equation*}
then \(u\leq 0\) in \(\R\times \R^*_+\).
\qed
\end{lemma}

\medskip
Let $R_0>0$ be fixed with $R_0<R_\e <\e ^{-1}$, and $\Omega$, $\Omega'$ be two regions, respectively:
\begin{align*}
\Omega = \{ z \in \R^2 \, | \, R_0<|z|<R_\e  \},
\qquad
\Omega' = \Bigl\{ z \in \R^2 \, | \, 2R_0<|z|<\frac{1}{2} R_\e  \Bigr\}.
\end{align*}
We adopt the following lemma presented in \cite{DDMR}.
 \begin{lemma} \label{lem:estimates-3b}
 (\cite{DDMR}).
Let $f:\R^2\rightarrow \R$ be such that $f(\bar{z})=-f(z)$ and $|f(z)| \leq \frac{1}{|z|}$. Let $u$ be a solution of
\begin{align*}
\Delta u +\e^2\p^2_{ss}u =f \text{ in } \Omega,
\end{align*}
such that  $u(\bar{z})=-u(z)$ and
\begin{align*}
|u(z)| &\leq R_0 |\ln \e| , \quad \forall\, z \mbox{ with } |z|= R_0,
\\
|u(z)| &\leq \mathbb{R}_\e, \quad \forall\, z \mbox{ with } |z|= \mathbb{R}_\e .
\end{align*}
Then there is a constant $C$ such that
\begin{equation*}
| u(z)| \leq C |z| \log\Bigl(\frac{2\mathbb{R}_\e}{|z|}\Bigl) ,
\quad
\forall\, z\in \Omega'.
\end{equation*}
\qed
\end{lemma}

\medskip
{\bf Acknowledgements: }
L. Duan was supported by  NSFC  (No.11771167) and The Science  and Technology Foundation of Guizhou Province ($[2001]$ZK008).
Q. Gao was supported by NSFC (No.11931012, No.11871386) and the ``Fundamental Research Funds for the Central Universities" (No. WUT2020IB019).
J. Yang  was supported by  NSFC  (No.11771167 and No. 12171109).

\end{document}